\documentclass[a4paper,final,oneside,notitlepage,11pt]{article}

\usepackage{amssymb}
\usepackage{amsmath}
\usepackage{amstext}
\usepackage[]{geometry}
\usepackage{mathrsfs}
\usepackage{euscript}
\usepackage[all]{xy}
\usepackage{xstring}
\usepackage{slashed}

\def\N {\mathbb{N}}
\def\Z {\mathbb{Z}}

\def\R {\mathbb{R}}
\def\C {\mathbb{C}}
\def\im{\mathrm{i}}
\def\id{\mathrm{id}}
\def\h {\mathrm{H}}

\def\trivlin{\mathbf{I}}

\def\quand{\quad\text{ and }\quad}

\def\quomma{\quad\text{, }\quad}
\def\quere{\quad\text{ where }\quad}

\def\ev{\mathrm{ev}}

\def\hc#1{\mathrm{h}_{#1}}

\def\subset{\subseteq}

\def\nobr{~\hspace{-0.26em}}
\def\maps{\nobr:\nobr}
\def\df{\nobr := \nobr}
\def\eq{\nobr = \nobr}

\let\Oldin\in\renewcommand{\in}{\nobr\Oldin\nobr}
\let\Oldtimes\times\renewcommand{\times}{\nobr\Oldtimes}
\let\Oldotimes\otimes\renewcommand{\otimes}{\nobr\Oldotimes}

\newlength{\widthtmp}
\def\length#1{\settowidth{\widthtmp}{#1}\the\widthtmp}

\def\lli#1{\,_{#1}\!}

\renewcommand{\varepsilon}{\epsilon}
\makeatletter
\def\bigset#1#2{\left\lbrace\;\begin{minipage}[c]{#1}\begin{center}#2\end{center}\end{minipage}\;\right\rbrace}

\makeatother

\def\erf#1{(\ref{#1})}

\newlength{\myl}
\newcommand\sheaf[1]{\unitlength 0.1mm
  \settowidth{\myl}{$#1$}
  \addtolength{\myl}{-0.8mm}
  \begin{picture}(0,0)(0,0)
  \put(3,6){\text{\uline{\hspace{\myl}}}}
  \end{picture}#1\hspace{-0.15mm}}

\newcommand{\ueins}{{\mathrm{U}}(1)}

\newcommand{\spin}[1]{{\mathrm{Spin}}\brackets{#1}}

\newcommand{\so}[1]{{\mathrm{SO}}\brackets{#1}}

\def\diff{\mathcal{D}\!i\!f\!\!f}

\def\hom{\mathcal{H}\!om}

\def\triv#1{\mathcal{T}\!riv(#1)}

\def\brackets#1{\IfStrEq{#1}{-}{}{(#1)}}
\def\subindex#1{\IfStrEq{#1}{-}{}{_{#1}}}
\def\buntech#1#2{\mathcal{B}\hspace{-0.01em}un_{\hspace{-0.1em}#1}^{#2}}
\def\bun#1#2{\buntech{#1}{}\brackets{#2}}
\def\buncon#1#2{\buntech{#1}{\nabla}\hspace{-0.05em}\brackets{#2}}
\def\bunconflat#1#2{\buntech#1{\nabla_{\!0}}\hspace{-0.05em}\brackets{#2}}

\def\grbtech#1{\mathcal{G}\hspace{-0.06em}r\hspace{-0.06em}b_{\hspace{-0.07em}{#1}}}
\def\grb#1#2{\grbtech#1\brackets{#2}}

\def\grbcon#1#2{\grbtech{#1}^{\nabla\!}\brackets{#2}}

\newcommand{\alxydim}[2]{\begin{aligned}\xymatrix#1{#2}\end{aligned}}

\renewcommand{\to}{\nobr\!\xymatrix@R=0cm@C=1.4em{\ar[r] &}\nobr}
\renewcommand{\mapsto}{\!\xymatrix@R=0cm@C=1.4em{\ar@{|->}[r] &}\!}
\renewcommand{\Rightarrow}{\!\xymatrix@R=0cm@C=1.4em{\ar@{=>}[r] &}\!}
\renewcommand{\Leftarrow}{\!\xymatrix@R=0cm@C=1.4em{\ar@{<=}[r] &}\!}
\newcommand{\incl}{\!\xymatrix@R=0cm@C=1.4em{\ar@{^(->}[r] &}\!}

\renewcommand\Leftrightarrow{\!\xymatrix@R=0cm@C=1.4em{\ar@{<=>}[r] &}\!}

\usepackage{amsthm}
\usepackage{graphicx}
\usepackage[margin=2cm,font=normalsize,labelfont=bf]{caption}
\usepackage{verbatim}
\usepackage{enumerate}
\usepackage{fancyhdr}
\usepackage[]{geometry}
\usepackage[normalem]{ulem}
\usepackage{xstring}
\usepackage{needspace}

\makeatletter

\newcounter{denseversion}
\newcounter{authorcounter}
\newcounter{adresscounter}

\def\title#1{\gdef\@title{#1}}
\def\@title{}
\def\subtitle#1{\gdef\@subtitle{#1}}
\def\@subtitle{}

\def\authortagsused{0}
\def\adresstag#1{\if!#1!\else$^{\;#1\;}$\fi}

\renewcommand{\author}[2][]{
  \stepcounter{authorcounter}
  \if!#1!\else\gdef\authortagsused{1}\fi
  \ifnum\value{authorcounter}=1
    \def\@authorstringa{#2\adresstag{#1}}
    \def\@authorstringb{#2}
    \def\@authorstringc{#2\adresstag{#1}}
  \else
    \g@addto@macro\@authorstringa{\ and #2\adresstag{#1}}
    \g@addto@macro\@authorstringb{\ and #2}
    \g@addto@macro\@authorstringc{\\#2\adresstag{#1}}
  \fi}
\def\@author{\ifnum\value{denseversion}=0\@authorstringa\else\@authorstringb\fi}

\def\@adressstringa{}
\def\@adressstringb{}
\newcommand{\adress}[2][]{
  \stepcounter{adresscounter}
  \ifnum\value{adresscounter}=1
    \g@addto@macro\@adressstringa{\ifnum\authortagsused=0\def\br{\\}\else\def\br{, }\fi\adresstag{#1}#2}
    \g@addto@macro\@adressstringb{\def\br{\\}\adresstag{#1}\parbox[t]{14cm}{#2}}
  \else
    \g@addto@macro\@adressstringa{\\[\bigskipamount]\adresstag{#1}#2}
    \g@addto@macro\@adressstringb{\\[\medskipamount]\adresstag{#1}\parbox[t]{14cm}{#2}}
  \fi}
\def\@adress{\ifnum\value{denseversion}=0\@adressstringa\else\@adressstringb\fi}

\def\preprint#1{\gdef\@preprint{#1}}
\def\@preprint{}
\def\keywords#1{\gdef\@keywords{#1}}
\def\@keywords{}
\def\msc#1{\gdef\@msc{#1}}
\def\@msc{}
\def\email#1{
   \gdef\@email{#1}
   \g@addto@macro\@authorstringc{ {\it (#1)}}}
\def\@email{}
\def\dedication#1{\gdef\@dedication{#1}}
\def\@dedication{}

\def\mybaselinestretch#1{\gdef\@mybaselinestretch{#1}}
\def\@mybaselinestretch{}

\def\refname{References}

\mybaselinestretch{1.2}
\renewcommand{\baselinestretch}{\@mybaselinestretch}

\def\denseversion{
  \setcounter{denseversion}{1}
  \newgeometry{left=3cm,right=3cm,top=3cm}
  \mybaselinestretch{1.1}
  \renewcommand{\baselinestretch}{\@mybaselinestretch}
  \normalfont
  \fancyfoot[C]{\itshape{\hspace{2.5cm}--$\,\,$\thepage$\,\,$--}}}

\newlength{\myparskip}
\newlength{\myproofparskip}

\renewcommand{\emph}[1]{\def\reserved@a{it}\ifx\f@shape\reserved@a\uline{#1}\else\textit{#1}\fi}

\newcommand{\mytableofcontents}{
   \ifnum\value{denseversion}=0
     \tableofcontents
   \else
     \renewcommand{\baselinestretch}{0.8}
     \normalfont
     \tableofcontents
     \renewcommand{\baselinestretch}{\@mybaselinestretch}
     \normalfont
   \fi}

\newlength{\zeilenlaenge}
\def\putindent#1{
  \settowidth{\zeilenlaenge}{#1}
  \ifnum\zeilenlaenge>\textwidth
    #1
  \else
    \noindent #1
  \fi
}

\def\href#1#2{#2}

\def\kohyp{
  \usepackage{hyperref}
  \hypersetup{
    linktocpage = true,
    pdftitle = {\@title},
    pdfauthor = {\@author},
    pdfkeywords = {\@keywords},    
    bookmarksopen = true,
    bookmarksopenlevel = 1
  }}  
\def\showkeywords{\begin{flushleft}\footnotesize\textbf{Keywords}: \@keywords\end{flushleft}}
\def\showmsc{\begin{flushleft}\footnotesize\textbf{MSC 2010}: \@msc\end{flushleft}}

\newcounter{mythm}[subsection]
\newcounter{mainthm}

\def\setsecnumdepth#1{
  \setcounter{secnumdepth}{#1}
  \setcounter{mythm}{0}
  \ifnum \c@secnumdepth >0
    \ifnum \c@secnumdepth >1
      \def\themythm{\thesubsection.\arabic{mythm}}
      \numberwithin{equation}{subsection}
      \renewcommand\theequation{\thesubsection.\arabic{equation}}
    \else
      \def\themythm{\thesection.\arabic{mythm}}
      \numberwithin{equation}{section}
      \renewcommand\theequation{\thesection.\arabic{equation}}
    \fi
  \else
    \def\themythm{\arabic{mythm}}
  \fi}

\newenvironment{mythmenv}{\strut\ \setlength{\parskip}{\myproofparskip}}{\setlength{\parskip}{\myparskip}}

\newlength{\mythmskip}
\newlength{\mythmtopskip}
\newtheoremstyle{mythmstylea}{\mythmtopskip}{\mythmskip}{\it}{}{\bf}{.}{0em}{}
\newtheoremstyle{mythmstyleb}{\mythmtopskip}{\mythmskip}{}{}{\bf}{.}{0em}{}

\theoremstyle{mythmstylea}
\newtheorem{mytheorem}[mythm]{Theorem}
\newtheorem{mydefinition}[mythm]{Definition}
\newtheorem{mycorollary}[mythm]{Corollary}
\newtheorem{myproposition}[mythm]{Proposition}
\newtheorem{mylemma}[mythm]{Lemma}

\newtheorem{mymaintheorem}[mainthm]{Theorem}
\newtheorem{mymaincorollary}[mainthm]{Corollary}
\newtheorem{mymainproposition}[mainthm]{Proposition}
\newtheorem{mymaindefinition}[mainthm]{Definition}

\theoremstyle{mythmstyleb}
\newtheorem{myremark}[mythm]{Remark}
\newtheorem{myexample}[mythm]{Example}
\newtheorem{myexercise}[mythm]{Exercise}

\newenvironment{theorem}[1][]{\begin{mytheorem}[#1]\begin{mythmenv}}{\end{mythmenv}\end{mytheorem}}
\newenvironment{definition}[1][]{\begin{mydefinition}[#1]\begin{mythmenv}}{\end{mythmenv}\end{mydefinition}}
\newenvironment{corollary}[1][]{\begin{mycorollary}[#1]\begin{mythmenv}}{\end{mythmenv}\end{mycorollary}}
\newenvironment{proposition}[1][]{\begin{myproposition}[#1]\begin{mythmenv}}{\end{mythmenv}\end{myproposition}}
\newenvironment{lemma}[1][]{\begin{mylemma}[#1]\begin{mythmenv}}{\end{mythmenv}\end{mylemma}}
\newenvironment{remark}[1][]{\begin{myremark}[#1]\begin{mythmenv}}{\end{mythmenv}\end{myremark}}

\newenvironment{maintheorem}[1]{\def\themainthm{#1}\begin{mymaintheorem}\begin{mythmenv}}{\end{mythmenv}\end{mymaintheorem}}
\newenvironment{maindefinition}[1]{\def\themainthm{#1}\begin{mymaindefinition}\begin{mythmenv}}{\end{mythmenv}\end{mymaindefinition}}
\newenvironment{maincorollary}[1]{\def\themainthm{#1}\begin{mymaincorollary}\begin{mythmenv}}{\end{mythmenv}\end{mymaincorollary}}

\renewenvironment{proof}[1][Proof]{\noindent #1. \begin{mythmenv}}{\hfill$\square$\end{mythmenv}\medskip}

\def\tocsection#1{\section*{#1}\addcontentsline{toc}{section}{#1}}

\def\mytitle{}
\def\zmptitle{
  \begin{tabular}{cc}
    \begin{minipage}[c]{0.4\textwidth}
      \begin{flushleft}
        \includegraphics[width=110pt]{../../tex/zmp}
      \end{flushleft}  
    \end{minipage}&
    \begin{minipage}[c]{0.55\textwidth}
      \begin{flushright}
      {\small\sf\@preprint}
      \end{flushright}
    \end{minipage}
  \end{tabular}
  \vskip 2cm}

\def\maketitle{
  \setlength{\parskip}{\myparskip}  
  \newpage
  \noindent
  \mytitle
  \begin{center}
    \LARGE\@title\\
    \if!\@subtitle!\else \smallskip\LARGE\@subtitle\\\fi
    \bigskip
    \if!\@author!\else\bigskip\large\@author\\\fi
    \ifnum\value{denseversion}=0
      \if!\@adress!\else     \bigskip\normalsize\@adress\\\fi
      \if!\@email!\else\ifnum\value{authorcounter}=1\bigskip\normalsize\textit{\@email}\\\else\fi\fi
    \else
    \fi
    \if!\@dedication!\else \bigskip\normalsize{\@dedication}\\\fi
  \end{center}
  \ifnum\value{denseversion}=0\vskip 1.5cm\else\vskip0.5cm\fi
  \thispagestyle{empty}}

\def\kobiburl#1{
   \IfBeginWith
     {#1}
     {http://arxiv.org/abs/}
     {\kobibarxiv{#1}}
     {\kobiblink{#1}}}
\def\kobibarxiv#1{\href{#1}{\texttt{[arxiv:\StrGobbleLeft{#1}{21}]}}}
\def\kobiblink#1{Available as: \href{#1}{\texttt{\StrSubstitute{#1}{_}{\underline{\;\;}}}}}

\def\kobib#1{
  \begin{raggedright}
  \ifnum\value{denseversion}=0\else\small\fi

  \end{raggedright}
  \ifnum\value{denseversion}=0\else
      \noindent
      \if!\@authorstringc!\else
        \ifnum\authortagsused=0\ifnum\value{authorcounter}>1\normalsize\@authorstringc\\[\medskipamount]\else\fi\else\normalsize\@authorstringc\\[\medskipamount]\fi
      \fi
      \if!\@adress!\else\normalsize\@adress\\{}\fi
      \ifnum\authortagsused=0
        \ifnum\value{authorcounter}=1
          \if!\@email!\else\linebreak\normalsize\textit{\@email}\\{}\fi
        \else
        \fi
      \else
      \fi
  \fi
  }

\newenvironment{commentfigure}{}
\newenvironment{sidewayscommentfigure}{\begin{minipage}}{\end{minipage}}

\def\showcomments{ -- Comments suppressed}

\newif\if@fewtab\@fewtabtrue{
  \count255=\time\divide\count255 by 60
  \xdef\hourmin{\number\count255}
  \multiply\count255 by-60\advance\count255 by\time
  \xdef\hourmin{\hourmin:\ifnum\count255<10 0\fi\the\count255}}
\def\ps@draft{
  \let\@mkboth\@gobbletwo
  \def\@oddfoot{
    \hbox to 7 cm{\tiny \versionno\hfil}
    \hskip -7cm\hfil\rm\thepage\hfil{\tiny\draftdate}}
  \def\@oddhead{}
  \def\@evenhead{}
  \let\@evenfoot\@oddfoot}
\def\draftdate{\number\month/\number\day/\number\year\ \ \ \hourmin }
\newcommand\version[1]{
  \typeout{}\typeout{#1}\typeout{}
  \vskip-1.7cm \centerline{\fbox{{\normalsize\tt DRAFT -- #1 -- 
  \draftdate\showcomments}}} \vskip0.92cm}
\def\draft#1{
  \def\versionno{#1}
  \pagestyle{draft}\thispagestyle{draft}
  \gdef\@ntitle{\version\versionno \@title}
  \global\def\draftcontrol{1}}
\global\def\draftcontrol{0}

\makeatother

\usepackage[latin1]{inputenc}
\usepackage[english]{babel}

\def\quot#1{``#1''}

\def\exd#1{{#1^{\vee}}}
\def\can{can}

\def\px#1#2{P_{\!#2}#1}
\def\p{P}

\def\ptx#1#2{\mathcal{P}_{\!#2}#1}
\def\ev{\mathrm{ev}}
\def\pt#1{\mathcal{P}#1}
\def\lt#1{\mathcal{L}#1}

\def\hc#1{\mathrm{h}_{#1}}
\def\pcomp{\nobr\star\nobr}

\def\prev#1{\overline{#1}}

\def\fun#1#2{\mathcal{F}\!un(#1,#2)}
\def\tfun#1#2{2\text{-}\mathcal{F}\!un(#1,#2)}

\def\un{\mathscr{R}}
\def\uncon{\mathscr{R}^{\nabla}}

\def\trcon{\mathscr{T}^{\nabla}}

\def\fus#1#2{\mathcal{F}\!us(#2,#1)}
\def\fuslc#1#2{\mathcal{F}\!us_{\mathrm{lc}}(#2,#1)}

\def\fusbun#1#2{\mathcal{F}\!us\buntech#1{}(#2)}
\def\fusbunconsf#1#2{\mathcal{F}\!us\bunconsf{#1}{#2}}
\def\fusbuncon#1#2{\mathcal{F}\!us\buncon{#1}{#2}}

\def\Omegafus{\Omega_{\!f\!u\!s}}

\def\extcon#1#2#3{\mathcal{E}\hspace{-0.12em}x\hspace{-0.02em}t^{\nabla}_{#1}(#2,#3)}
\def\extconsf#1#2#3{\mathcal{E}\hspace{-0.12em}x\hspace{-0.02em}t^{\nabla}_{#1}{}^{_{\!\!s\!f\!}}(#2,#3)}

\def\fusextconsf#1#2#3{\mathcal{F}\!us\!\,\extconsf{#1}{#2}{#3}}

\def\bunconsf#1#2{\buntech{#1}{\,\nabla}{}^{_{\!\!s\!f}}\hspace{-0.15em}\brackets{#2}}

\def\prodbun#1#2{\mathcal{P}\!r\!od\buntech#1{}(#2^{[2]})}
\def\prodbuncon#1#2{\mathcal{P}\!r\!od\buncon{#1}{#2^{[2]}}}

\def\diffbun#1#2{\diff\hspace{-0.1em}\bun{#1}{#2}}

\def\diffbunconflat#1#2{\diff\hspace{-0.1em}\bunconflat#1{#2}}

\def\diffgrb#1#2{\diff\hspace{-0.1em}\grbtech#1(#2)}
\def\diffgrbcon#1#2{\diff\hspace{-0.1em}\grbcon{#1}{#2}}

\def\ptr#1{\tau_{#1}}

\def\struc#1#2{#1\text{-}\mathcal{L}\!i\!f\!t(#2)}

\def\Untrans{Regression}

\title{Transgression to Loop Spaces and its Inverse, II}
\subtitle{Gerbes and Fusion Bundles with Connection}
\author{Konrad Waldorf}
\adress{Department of Mathematics\\University of California, Berkeley\\970 Evans Hall \#3840\\Berkeley, CA 94720, USA}
\email{waldorf@math.berkeley.edu}
\msc{Primary 53C08,  Secondary 53C27, 55P35}

\kohyp

\begin{document}


\maketitle

\begin{abstract}
We prove that the category of abelian gerbes with connection over a smooth manifold is equivalent to a  certain category of principal bundles over the free loop space. These bundles are equipped with a connection and with a \quot{fusion} product with respect to triples of paths. The equivalence is established by explicit functors in both directions: transgression and regression. We describe applications to geometric lifting problems and loop group extensions. 
\showkeywords
\showmsc
\end{abstract}

\mytableofcontents

\tocsection{Introduction}

This series of papers is about a geometric version of transgression to loop spaces: it relates geometry over a smooth manifold $M$ to geometry over its free loop space $LM$. In terms of the underlying topology,  transgression  covers the homomorphism 
\begin{equation*}
H^k(M,\Z) \to H^{k-1}(LM,\Z)
\end{equation*}
obtained by pullback along the evaluation map $\ev\maps LM \times S^1 \to M$ followed by integration over the fibre. Part I \cite{waldorf9} is concerned with the case $k=2$: there we derive bijections between  bundles over $M$ and certain \quot{fusion maps} on $LM$, both in a setup with and without connections. In the present Part II we look at the case $k=3$ in a setup \emph{with} connections. It is written in a self-contained way -- contents taken from Part I are either reviewed or explicitly referenced. Part III \cite{waldorf11}  treats the case $k=3$ in a setup \emph{without} connections.

Thus, in this article we  want to relate \emph{gerbes with connection} over $M$ to  \emph{bundles with connection} over $LM$. Such a relationship has first been studied by Brylinski and McLaughlin \cite{brylinski1,brylinski4}. They define  transgression as a functor which takes a \quot{Dixmier-Douady sheaf of groupoids with connective structure and curving} over  $M$ (one version of a gerbe with connection) to a hermitian line bundle over $LM$. Brylinski and McLaughlin observed that the  bundles in the image of transgression come equipped with the following  additional structure:
\begin{enumerate}[(i)]
\item 
A metric connection.

\item
A product with respect to loop composition. For two loops $\tau_1$ and $\tau_2$ based at the same point and smoothly composable to a third loop $\tau_1 \pcomp \tau_2$, this product provides fibre-wise linear maps
\begin{equation*}
\lambda\maps P_{\tau_1} \otimes P_{\tau_2} \to P_{\tau_1 \pcomp \tau_2}
\end{equation*}
that are compatible with the homotopy associativity of loop composition. \footnote{Apart from the claim that this product exists, I was unable to find its definition  in the literature. } 

\item
An equivariant structure with respect to the action of orientation-preserving diffeomorphisms of $S^1$ on $LM$.
\end{enumerate}
Unfortunately, it remained unclear whether or not this list is complete, i.e. if it characterizes those line bundles over $LM$ that can be obtained by transgression.

The purpose of this article is to provide a complete solution to this problem. We give a variation of the above additional structure, which is carefully chosen exactly such that transgression becomes an \emph{equivalence of categories}. In other words, we give a complete  loop space formulation of the geometry of  gerbes with connection. The precise formulation of this result is given in Section \ref{sec:results}. The following list  outlines the difference between ours and Brylinski's and McLaughlin's characterization. 
\begin{enumerate}[(i)]
\item 
Not every  connection is allowed: its holonomy around loops in $LM$ is subject to extra conditions related to the rank of the  associated torus in $M$. 

\item
The product is defined with respect to \emph{paths} instead of loops: if $(\gamma_1,\gamma_2,\gamma_3)$ is a triple of paths in $M$ with a common initial point and a common end point, and each path is smoothly composable with the inverse of each other, the product consists of fibre-wise linear maps
\begin{equation*}
\lambda\maps P_{\,\prev {\gamma_2} \pcomp \gamma_1} \otimes P_{\,\prev{\gamma_3} \pcomp \gamma_2} \to P_{\,\prev{\gamma_3} \pcomp \gamma_1}\text{.}
\end{equation*} 
These maps are required to be \emph{strictly} associative for quadruples of paths. 

\item
The equivariant structure is no additional structure at all. We will show that it is already determined by the connection (i).

\item
Additionally, we discover a subtle, new compatibility condition between the connection (i) and the product (ii) related to loop rotation by an angle of $\pi$. 

\end{enumerate}
The main theorem of this article states that the category of \emph{abelian gerbes with connection} over $M$ is equivalent to the category formed by bundles over $LM$ equipped with the structures (i) and (ii) subject to the condition (iv). The equivalence is established by explicit functors in both directions: \emph{transgression} and \emph{regression}. Our transgression functor is essentially  the one of Brylinski and McLaughlin, while our regression functor is a novelty: it performs the construction of a  gerbe with connection over $M$ from a given  bundle  over $LM$. 

A major role in the derivation of our results is played by  generalized manifolds, more precisely: \emph{diffeological spaces}. In particular, we will use a reformulation of differential forms in terms of smooth maps on certain diffeological spaces as an essential tool in the definition of regression. This reformulation has been developed  in joint work with Schreiber \cite{schreiber3,schreiber5}. Further, it turns out that Murray's \emph{bundle gerbes} \cite{murray} are a convenient model for abelian gerbes in the context of this article. 
Nonetheless, the choice of bundle gerbes  is a purely technical issue. Many other models of abelian gerbes are equivalent to bundle gerbes, for example  the Dixmier-Douady sheaves of groupoids  used by Brylinski and McLaughlin \cite{brylinski1}, Deligne cocycles, and $B\C^{\times}$-bundles  \cite{gajer1}. Via such equivalences, our results  apply to each of these models.

We describe two applications of our results, corresponding to two natural examples of bundle gerbes: lifting gerbes \cite{brylinski1,murray}, and gerbes over Lie groups \cite{gawedzki1,meinrenken1,gawedzki2}. In the case of lifting gerbes, our results provide a complete loop space formulation for \emph{geometric lifting problems}. In particular, it clarifies the relationship between spin structures on an oriented Riemannian manifold  and orientations of its loop space, which appears in classical results of Atiyah and Witten, see \cite{atiyah2}. In the case of gerbes over a Lie group $G$, our results identify   a new and interesting subclass of \emph{central extensions} of the loop group $LG$, namely exactly those which lie in the image of transgression. These central extensions can hence be studied  with \emph{finite-dimensional} methods on the group $G$. One line of future research could be to relate the fusion product that we have discovered  on these central extensions to the tensor product of positive energy representations (also called \quot{fusion}). Such a relationship  could lead to a new, finite-dimensional  understanding of the Verlinde algebra. 
A third application that could be addressed in the future concerns the \emph{index gerbe} of Lott \cite{Lott} whose transgression is related to certain determinant bundles over the loop space \cite{Bunke2002}. According to the results of the present article, these bundles carry so far unrecognized additional structure: fusion products and a particular kind of connections.

\paragraph{Acknowledgements.} I gratefully acknowledge  a Feodor-Lynen scholarship, granted by the Alex\-an\-der von Hum\-boldt Foundation.  I thank  Martin Olbermann and Peter Teichner for many exciting discussions.

\setsecnumdepth{2}

\section{Results of this Article}

\label{sec:results}

\subsection{Main Theorem}

Let $M$ be a connected smooth manifold and let $A$ be an abelian Lie group, which may be discrete and may be non-compact. We work with principal $A$-bundles, which reduce to the hermitian line bundles from the introduction  upon setting $A = \ueins$.

We use the theory of diffeological spaces \cite{kriegl1,iglesias1,baez6}. Principal bundles and connections over diffeological spaces have been introduced and studied in Part I  \cite{waldorf9}. For the purpose of this section, it will be enough to accept  that diffeological spaces are like  smooth manifolds -- there are just more of them. Examples are the free loop space $LM \df C^{\infty}(S^1,M)$ and the path space $PM$. The diffeology on the loop space is equivalent to the Fréchet manifold structure \cite[Lemma A.1.7]{waldorf9}, but  more convenient to use. Concerning path spaces, we demand that all paths  have \quot{sitting instants}, so that paths with a common endpoint are always smoothly composable. Spaces of paths with sitting instants do not have a \quot{good} Fréchet manifold structure, but are  nice diffeological spaces. Ultimately, we will  use  diffeological spaces of \emph{thin homotopy classes} of paths and loops; such spaces are by no means manifolds.

Let us start by looking at the category of bundles over $LM$ that we want to consider; a detailed discussion is the content of Section \ref{sec:fusbun}.
We denote by $PM^{[k]}$ the diffeological space of $k$-tuples of paths in $M$ with a common initial point and a common endpoint. We introduce a \emph{fusion product} on a principal $A$-bundle $P$ over $LM$ as a smooth bundle isomorphism
\begin{equation*}
\lambda\maps e_{12}^{*}P \otimes e_{23}^{*}P \to e_{13}^{*}P
\end{equation*}
over $PM^{[3]}$. Here, $e_{ij}\maps PM^{[3]} \to LM$ is the $i\!j$-projection to $PM^{[2]}$ followed by the map $l\maps PM^{[2]}\to LM$ which takes a pair $(\gamma_1,\gamma_2)$ to the loop $\prev{\gamma_2} \pcomp \gamma_1$.
A fusion product is required to satisfy a strict associativity condition over $PM^{[4]}$. A principal $A$-bundle $P$ with a fusion product $\lambda$ will be called \emph{fusion bundle} (Definition \ref{def:fusionbundle}).

Next we add connections. 
There are two conditions we impose for connections on a fusion bundle $(P,\lambda)$. The first  is the obvious one: we call a connection on $P$ \emph{compatible with $\lambda$}, if $\lambda$ is a connection-preserving bundle morphism. The second  is the subtle condition mentioned in the introduction under (iv). A complete definition has to be deferred  (Definition \ref{def:symmetrizing}). 
Roughly, we observe that the spaces $PM^{[k]}$ have an involution which inverts the paths and permutes their order. The map $l$ exchanges this involution with the loop rotation by an angle of $\pi$. Regarding these rotations as paths in $LM$, the parallel transport of the connection on $P$ lifts them to the total space $P$. Now, the connection \emph{symmetrizes} the fusion product $\lambda$ if switching the factors in the fusion product commutes with these lifts.

Finally, we introduce the notion of a \emph{superficial}  connection (Definition \ref{def:superficial}), imposing the conditions mentioned in the introduction under (i). It applies only to connections on bundles \emph{over loop spaces}, but in general without respect to a   fusion product. The idea is that the holonomy of a superficial connection around a loop $S^1 \to LM$ should have two  features of a \quot{surface holonomy}\footnote{The german word for \quot{superficial} (\quot{oberflächlich}) is -- literally -- \quot{surface-like}.}
  around the associated torus $\phi\maps S^1 \times S^1 \to M$. 
The first feature is that the holonomy vanishes   whenever $\phi$ is a rank one map. The second feature is that the holonomy  does not depend on the rank-two homotopy class of $\phi$. In that sense, superficial connections are a precise manifestation of what Segal calls a \quot{string connection} \cite{Segal2001}. 
\begin{maindefinition}{A}
\label{def:fusbuncon}
A \emph{fusion bundle with superficial connection over $LM$} is a  principal $A$-bundle over $LM$ with a fusion product and a compatible, symmetrizing, and superficial connection. 
\end{maindefinition}

Fusion bundles with superficial connection form a monoidal category that we denote by  $\fusbunconsf A {LM}$. The morphisms in this category are bundle morphisms that preserve the connections and the fusion products.

On the side of the geometry over  $M$, we consider a monoidal 2-category $\diffgrb A M$ modelling $A$-banded gerbes over $M$. Its objects are \emph{diffeological $A$-bundle gerbes}: they are very similar to the $\C^{\times}$-bundle gerbes introduced by Murray \cite{murray}, but we allow $\C^{\times}$ to be replaced by $A$, and we allow them to live internal to diffeological spaces instead of smooth manifolds. The latter generalization is purely technical: the inclusion of  \emph{smooth} $A$-bundle gerbes into \emph{diffeological} $A$-bundle gerbes is an equivalence of 2-categories. The same comments apply to the 2-category $\diffgrbcon A M$ of diffeological $A$-bundle gerbes \emph{with connection}. A  detailed exposition of diffeological $A$-bundle gerbes is given in Section \ref{sec:diffgrb}.

In order to compare \emph{2-}categories of diffeological bundle gerbes over $M$ and \emph{1-}categories of fusion bundles over $LM$, we downgrade the 2-categories to categories with morphisms the 2-isomorphism-classes of the former 1-morphisms.
This downgrading is below indicated by the symbol $\hc 1$. Transgression is then a  monoidal  functor
\begin{equation*}
\alxydim{@C=2cm}{\hc 1 \diffgrbcon A M \ar[r]^-{\trcon} & \fusbunconsf A {L M}\text{.}}
\end{equation*}
Its definition is based on the one given by Brylinski and McLaughlin. In Section \ref{sec:trans} we explain in detail that the resulting principal $A$-bundles over $LM$ are fusion bundles with superficial connection in the sense of  Definition \ref{def:fusbuncon}.

Opposed to transgression, its inverse functor -- which we call  \emph{regression} -- is completely new. The reason is that it is only well-defined on the \quot{correct} category of bundles over $LM$, namely the category $\fusbunconsf A {L M}$ we introduce in this article. Regression depends on the choice of a base point $x\in M$, and is a monoidal functor
\begin{equation*}
\alxydim{@C=2cm}{\hc 1 \diffgrbcon A M & \ar[l]_-{\uncon_x} \fusbunconsf A {L M}\text{.}}
\end{equation*}
For different choices of base points one obtains  naturally equivalent regression functors. The idea of the regression functor is very simple. Given a fusion bundle $(P,\lambda)$ over $LM$ it reconstructs the following diffeological bundle gerbe over $M$: its subduction (the diffeological analog of a surjective submersion) is the endpoint evaluation $\ev_1 \maps \px Mx \to M$ on the space of paths starting at $x$. Over its two-fold fibre product $\px Mx^{[2]}$ we have the principal $A$-bundle $l^{*}P$, and the given fusion product $\lambda$ pulls back to a bundle gerbe product. More difficult is the promotion of this simple construction to the setup \emph{with} connections. It includes the construction of a \quot{curving} 2-form on $\px Mx$, for which we use diffeological techniques developed jointly with Schreiber \cite{schreiber3,schreiber5}.  
Regression is defined in Section \ref{sec:untrans}.

We also compute expressions for the surface holonomy of the regressed bundle gerbe $\uncon_x(P)$ in terms of the connection on $P$. The  idea is to decompose a  surface into cylinders, discs, and pairs of pants. We show that the surface holonomy of a cylinder is simply the parallel transport of the connection on $P$, while the one of a disc is related to the fact that a fusion bundle trivializes over the constant loops. The surface holonomy of a pairs of pants is determined by the fusion product. 

Now we are in position to state the main theorem of this article: 

\Needspace{3\baselineskip}
\begin{maintheorem}{A}
\label{th:con}
Let $M$ be a connected  smooth manifold with base point $x\in M$. Then, transgression and regression  form an equivalence of monoidal categories,
\begin{equation*}
\alxydim{@C=2cm}{\hc 1 \diffgrbcon AM  \ar@<0.4em>[r]^{\trcon} & \fusbunconsf A{LM} \ar@<0.4em>[l]^{\uncon_x}\text{.}}
\end{equation*}
Moreover, this equivalence is natural with respect to base point-preserving smooth maps.
\end{maintheorem}

The proof of Theorem \ref{th:con}  consists in an explicit  construction of the two required natural equivalences, and is carried out in Section \ref{sec:proof}.  In the following we describe some immediate implications of Theorem \ref{th:con}.

First we look at the bijection that the equivalence of Theorem \ref{th:con} induces on the \emph{Hom-sets} of the involved categories. If $\mathcal{G}$ and $\mathcal{H}$ are isomorphic bundle gerbes with connection, the Hom-set $\hom(\mathcal{G},\mathcal{H})$ in the category $\hc 1 \diffgrbcon AM$ is  a torsor over the group $\hc 0 \bunconflat AM$ of isomorphism classes of flat principal $A$-bundles over $M$ (Lemma \ref{lem:diffgrbcon}). On the other side, if $P$ and $Q$ are isomorphic fusion bundles with superficial connections, the Hom-set $\hom(P,Q)$ in the category $\fusbunconsf A{LM}$ is a torsor over the group $\fuslc A {LM}$ of locally constant \emph{fusion maps} (Lemma \ref{lem:fusiontorsor}). 
Fusion maps are smooth maps $f\maps  LM \to A$ subject to the  condition 
\begin{equation*}
e_{12}^{*}f \cdot e_{23}^{*}f = e_{13}^{*}f
\end{equation*}
over $PM^{[3]}$, which can be seen as a \quot{de-categorification} of a fusion product for bundles. 
Since the two Hom-sets are in bijection, we get:

\Needspace{3\baselineskip}
\begin{maincorollary}{A}
\label{co:homsets}
Transgression and regression induce a bijection
\begin{equation*}
\hc 0 \bunconflat A M \cong \fuslc A {LM}\text{.}
\end{equation*}
\end{maincorollary}  

In other words, locally constant fusion maps constitute a loop space formulation for the geometry of flat principal $A$-bundles. Indeed, fusion maps have been introduced in Part I of this series of papers, where such formulations are studied. The above corollary  reproduces  \cite[Corollary A]{waldorf9}.

Next are two implications that follow from properties of the functors $\trcon$ and $\uncon_x$. Firstly, as mentioned above, all bundle gerbes in the image of regression have the same subduction, namely the end-point evaluation $\ev_1: \px Mx \to M$ on the space of paths starting at $x$. Thus we get:

\Needspace{2\baselineskip}
\begin{maincorollary}{B}
Every bundle gerbe $\mathcal{G}$ with connection over $M$ is isomorphic to a   bundle gerbe with connection and subduction $\ev_1 \maps \px Mx \to M$, namely to $\uncon_x(\trcon_{\mathcal{G}})$. 
\end{maincorollary}

Secondly, transgression and regression restrict properly to  \emph{flat} bundle gerbes and \emph{flat} fusion bundles  (Propositions \ref{prop:connectionprop1} and \ref{prop:flatreg}): 

\Needspace{2\baselineskip}
\begin{maincorollary}{C}
\label{co:flat}
Transgression and regression form an equivalence between  categories of bundle gerbes with \emph{flat} connection and fusion bundles with \emph{flat} superficial connections. \end{maincorollary}

A similar restriction to \emph{trivial} gerbes and \emph{trivial} bundles is more subtle. 
Let us for simplicity look at the case $A=\ueins$, and identify its Lie algebra with $\R$.  Now, connections on the \emph{trivial} bundle gerbe are just 2-forms $\omega \in \Omega^2(M)$. Superficial connections on the trivial fusion bundle over $LM$ form a group denoted $\Omegafus^1(LM)$; these are 1-forms on $LM$  satisfying a list of properties reflecting the three conditions of Definition \ref{def:fusbuncon}, see Section \ref{sec:fusionforms}.  Variants of such forms are sometimes called \quot{local} or \quot{ultra-local}; see e.g. \cite{Coquereaux1998}.
Transgression restricts properly to trivial gerbes in the sense that the resulting fusion bundles are canonically trivializable \cite[Lemma 3.6]{waldorf13}. Under these trivializations, transgression is given by the  formula
\def\theequation{\arabic{equation}}
\begin{equation}
\label{eq:transformula}
\Omega^2(M) \to \Omegafus^1(LM) \maps \omega \mapsto - \int_{S^1}\ev^{*}\omega\text{.}
\end{equation}
On the other hand, the regression of a trivial fusion bundle with connection $\eta \nobr \in \nobr \Omega^1_{fus}(LM)$ is a trivializable gerbe, but no trivialization is given. Thus, the equivalence of Theorem \ref{th:con} only restricts well to a bijection between \emph{isomorphism classes} of trivial gerbes with connection and \emph{isomorphism classes} of trivial fusion bundles with superficial connections. It is well-known that two connections on the trivial gerbe are isomorphic if and only if they differ by a closed 2-form with integral periods; these form the group $\Omega^2_{cl,\Z}(M)$. On the other side, two  superficial connections on the trivial fusion bundle are isomorphic if and only if they differ by the image of a fusion map under the logarithmic derivative (Lemma \ref{lem:dlogfus}).
Summarizing, we obtain the following statement, which can be understood as a loop space formulation of the geometry of 2-forms:

\begin{maincorollary}{D}
\label{co:forms}
The transgression map \erf{eq:transformula} induces a group isomorphism
\begin{equation*}
\frac{\Omega^2(M)}{\Omega^2_{cl,\Z}(M) } \cong  \frac{\Omega^1_{fus}(LM)}{\mathrm{dlog}  \,\fus{\ueins}{LM}} \;  \text{.}
\end{equation*}
\end{maincorollary}

The last implication of Theorem \ref{th:con} that I want to point out concerns again the case $A=\ueins$.  Then, bundle gerbes with connection over $M$ are classified by  the differential cohomology group $\widehat \h^3(M,\Z)$. Hence, Theorem \ref{th:con} provides a loop space formulation of  differential cohomology:

\Needspace{1\baselineskip}
\begin{maincorollary}{E}
The group of isomorphism classes of fusion bundles over $LM$ with superficial connections is isomorphic to  the degree three differential cohomology  of $M$:
\begin{equation*}
\widehat \h^3(M,\Z) \cong \hc 0 \fusbunconsf {\,\ueins\!\!\!} {LM}\text{.}
\end{equation*}
\end{maincorollary}

\subsection{Application to  Lifting Problems}

\label{sec:lifting}

A \emph{lifting problem} is the problem of lifting the structure group of a principal $G$-bundle $E$ over $M$ to a central extension
\begin{equation*}
1 \to A \to \hat G \to G \to 1
\end{equation*}
of $G$. Lifting problems can be understood completely by looking at the \emph{lifting  gerbe} $\mathcal{G}_E$,  a certain $A$-bundle gerbe \cite{murray}.  Namely, there is an equivalence 
\setcounter{equation}{1}
\def\theequation{\arabic{equation}}
\begin{equation}
\label{eq:lifteq}
\struc {\hat G}E \cong \triv {\mathcal{G}_E}
\end{equation} 
between the category of lifts of $E$ and the category of trivializations of the gerbe $\mathcal{G}_E$. An analogous equivalence for \emph{geometric} lifting problems involving connections exists but requires a more detailed treatment  \cite{gomi3,waldorf13}.

The idea is to transgress the lifting gerbe $\mathcal{G}_E$ to a \emph{principal $A$-bundle} over the loop space. Since transgression is a functor, trivializations become  \emph{sections} of this bundle. Since it is an equivalence of categories, this provides a complete loop space formulation of  geometric lifting problems. It is worked out in full detail in \cite{waldorf13}.

In the following we continue with the special case of a \emph{discrete} Lie group $A$, in order to avoid the above-mentioned issues with connections. Indeed, for discrete $A$, every principal $A$-bundle and every $A$-gerbe has a \emph{unique} connection, and every morphism between bundles or gerbes is connection-preserving. Thus, the lifting gerbe $\mathcal{G}_E$ defines a fusion bundle $P_E \df \trcon_{\mathcal{G}_E}$  over $LM$. We say that a section $\sigma\maps LM \to P$ of a fusion bundle $(P,\lambda)$ over $LM$ is \emph{fusion-preserving}, if 
\begin{equation*}
\lambda(e_{12}^{*}\sigma \otimes e_{23}^{*}\sigma) = e_{13}^{*}\sigma\text{,}
\end{equation*} 
i.e. if it corresponds to a fusion-preserving bundle morphism between $P$ and the trivial fusion bundle. 

\begin{maintheorem}{B}
\label{th:lift}
Let $M$ be a connected smooth manifold, $G$ be a Lie group, and $E$ be a principal $G$-bundle over $M$. Let $\hat G$ be a central extension of $G$ by a discrete  group. Then, transgression and regression induce a bijection
\begin{equation*}
\hc 0 (\struc {\hat G}E) \cong \bigset{3cm}{Fusion-preserving sections of $P_E$}\text{.}
\end{equation*}
\end{maintheorem}

\begin{proof}
We have the following bijections:
\begin{equation*}
\hc 0 (\struc {\hat G}E) \cong\hc 0 \triv {\mathcal{G}_E} \cong \hc 0 \hom(\mathcal{G}_E,\mathcal{I})  \cong \hom(P_E, \trivlin)\text{,}
\end{equation*}
where the first Hom-set is the one of the category $\hc 1 \diffgrbcon AM$, and the second Hom-set  is the one in $\fusbunconsf A{LM}$. The first bijection is induced from the  equivalence \erf{eq:lifteq}, the second is the definition of a trivialization of a gerbe, and the third is induced by Theorem \ref{th:con}. The set on the right is by definition  the set of fusion-preserving sections.\end{proof}

The main motivation for    Theorem \ref{th:lift} is the case of spin structures on an $n$-dimensional oriented Riemannian manifold $M$. In this case, the relevant central extension is
\begin{equation*}
1 \to \Z/2\Z \to \spin n \to \so n \to 1\text{,}
\end{equation*}
and the bundle whose structure group we want to lift is the oriented frame bundle $FM$ of $M$. The transgression of the associated lifting gerbe $\mathcal{G}_{FM}$ is a $\Z/2\Z$-bundle $P_{FM}$ over $LM$, and it plays the role of the \emph{orientation bundle} of $LM$  \cite{atiyah2,mclaughlin1,stolz3}.

From the theory developed in this article we get the additional information that the orientation bundle of $LM$ comes  with a fusion product. Accordingly, among \emph{all}  orientations of $LM$ (the sections of $P_{FM}$) is a subset  of  \emph{fusion-preserving} orientations. In this context, Theorem \ref{th:lift} becomes:
\begin{maincorollary}{E}
Let $M$ be a connected oriented Riemannian manifold. Then, $M$ is spin if and only if there exists a fusion-preserving orientation of $LM$. In this case, transgression and regression induce a bijection
\begin{equation*}
\bigset{3.7cm}{Equivalence classes of spin structures on $M$} \cong \bigset{3.3cm}{Fusion-preserving orientations of $LM$}
\end{equation*}
\end{maincorollary}

I learned about  this correspondence  from an unpublished paper \cite{stolz3} of Stolz and Teichner, which gives a proof using Clifford bimodules \cite[Theorem 9]{stolz3}. It also outlines a proof using lifting gerbes  \cite[Remark 11]{stolz3}. It was the main motivation for the present article  to understand this result from a more general point of view. 

An example that illustrates  the difference between orientations and fusion-preserving orientations  is a real Enriques surface $\mathscr{E}$. Indeed, $\mathscr{E}$  is not spin but its loop space is orientable (its second Stiefel-Whitney class is non-trivial, but transgresses to zero). 
In other words, no orientation of $L\mathscr{E}$ is fusion-preserving.

\subsection{Application to Loop Group Extensions}

Let $G$ be a connected Lie group. We denote by $\left \langle -,-  \right \rangle$ the Killing form and by $\theta,\bar\theta$ the left and right invariant Maurer-Cartan forms, respectively. We consider the differential forms
\begin{equation*}
H \df \frac{1}{6} \left \langle  \theta \wedge [\theta \wedge \theta]  \right \rangle \in \Omega^3(G)
\quand
\rho \df \frac{1}{2}\left \langle  \mathrm{pr}_1^{*}\theta \wedge \mathrm{pr}_2^{*}\bar\theta  \right \rangle \in \Omega^2(G \times G)\text{.}
\end{equation*}
They have the property that the tuple $(H,\rho,0,0)$ is a 4-cocycle in the \emph{simplicial de Rham complex} of $G$, and so represents an element in $\h^4(BG,\R)$.

 A \emph{multiplicative gerbe with connection} over $G$ \cite[Definition 1.3]{waldorf5} is a pair $(\mathcal{G},\mathcal{M})$ consisting of a $\ueins$-bundle gerbe $\mathcal{G}$ over $G$ with connection of curvature $H$ and a connection-preserving 1-isomorphism
\begin{equation*}
\mathcal{M}\maps \mathrm{pr}_1^{*}\mathcal{G} \otimes \mathrm{pr}_2^{*}\mathcal{G} \to m^{*}\mathcal{G} \otimes \mathcal{I}_{\rho}\text{,}
\end{equation*}
where $m\maps G \times G \to G$ denotes the group multiplication, $\mathrm{pr}_1,\mathrm{pr}_2\maps G \times G \to G$ are the projections, and $\mathcal{I}_{\rho}$ denotes the trivial gerbe equipped with the connection $\rho$. The 1-isomorphism $\mathcal{M}$ has to satisfy an associativity constraint\footnote{In general, there is an \quot{associator} instead of a constraint. However, in the present setting with a connected Lie group and gerbes \emph{with connections}, this associator is unique if it exists, see \cite[Example 1.5]{waldorf5}.} over $G \times G \times G$. A multiplicative gerbe with connection determines a class in $\h^4(BG,\Z)$ which is a preimage for the class of $(H,\rho,0,0)$ under the homomorphism induced by the inclusion $\Z \subset \R$ \cite[Proposition 2.3.1]{waldorf5}.

Since transgression is functorial, monoidal, and natural, it takes the multiplicative gerbe $(\mathcal{G},\mathcal{M})$ to a pair $(P,\mu)$ where $P$ is a principal $\ueins$-bundle with connection over $LG$, and $\mu$ is a multiplicative structure on $P$: a connection-preserving bundle isomorphism
\begin{equation*}
\mu\maps \mathrm{pr}_1^{*}P \otimes \mathrm{pr}_2^{*}P \to m^{*}P \otimes \trivlin_{-L\rho}
\end{equation*}
over $LG^2$ which is associative over $LG^3$. 
Such bundle morphisms $\mu$ correspond bijectively to Fréchet Lie group structures on $P$ making it a central extension of $LG$ by $\ueins$ \cite{grothendieck1}, also see \cite[Theorem 3.1.7]{waldorf5}. We obtain a category $\extcon {}{LG}{\ueins}$, whose objects are central extensions of $LG$ by $\ueins$ equipped with a connection such that the multiplication $\mu$ is connection-preserving, and whose morphisms are connection-preserving isomorphisms between central extensions. Summarizing, there is a functor
\begin{equation*}
\trcon \maps \hc 1 \mathcal{M}ult\diffgrbcon {} G \to \extcon {} {LG}\ueins\text{.}
\end{equation*}
For example, the universal central extension of a compact, simple, simply-connected Lie group belongs to the category $\extcon {}{LG}\ueins$, namely as the image of the \emph{basic} gerbe over $G$ \cite[Corollary 3.1.9]{waldorf5}.

We give  a little foretaste how the results of this article can be applied to central extensions. Theorem \ref{th:con} lets us recognize new features of central extensions in the image of transgression. Firstly, they come equipped with a \emph{fusion product}, and their group structure is fusion-preserving (in the sense that the corresponding isomorphism $\mu$ is fusion-preserving). Secondly, the connections are \emph{superficial} and \emph{symmetrizing}. We denote the category of these central extensions by $\fusextconsf{}{LG}{\ueins}$; this is an interesting, new  subclass of central extensions of loop groups. Theorem \ref{th:con} implies:

\begin{maintheorem}{C}
\label{th:lgext}
Let $G$ be a connected Lie group. Then, transgression and regression induce an equivalence of categories:
\begin{equation*}
\hc 1 \mathcal{M}ult\diffgrbcon {} G \cong \fusextconsf {} {LG}\ueins\text{.}
\end{equation*}
\end{maintheorem}

In other words, the central extensions in $\fusextconsf {} {LG}\ueins$ can be understood via a theory of multiplicative gerbes over $G$, in particular, with \emph{finite-dimensional} methods. For example, there exists a complete classification of multiplicative gerbes $(\mathcal{G},\mathcal{M})$ over all compact, simple Lie groups, in terms of constraints for the \emph{level} of the gerbe $\mathcal{G}$, determined in joint work with Gaw\c edzki \cite{gawedzki9}. Via Theorem \ref{th:lgext}, this classification transfers to a complete classification of the central  extensions in $\fusextconsf {} {LG}\ueins$.

\setsecnumdepth{2}

\section{Fusion Bundles and Superficial Connections}

\label{sec:fusbun}

\subsection{Fusion Bundles}

For $G$ a (abelian) Lie group and $X$ a diffeological space, diffeological principal $G$-bundles over $X$ form a (monoidal) sheaf of groupoids with respect to the Grothendieck topology of subductions \cite[Theorems 3.1.5 and 3.1.6]{waldorf9}. \emph{Subductions} are the diffeological analog of maps with smooth local sections. We also recall that the category of diffeological spaces has all fibre products. If $\pi\maps Y \to X$ is a smooth map between diffeological spaces, we write $Y^{[k]}$ for the $k$-fold fibre product of $Y$ over $X$, and we write  $\pi_{i_1...i_k}\maps Y^{[p]} \to Y^{[k]}$ for the projection to the indexed factors.

\begin{definition}
\label{def:product}
Let $\pi\maps Y \to X$ be a subduction between diffeological spaces. 
\begin{enumerate}
\item 
A \emph{product} on a  principal $A$-bundle $P$ over $Y^{[2]}$ is a bundle isomorphism
\begin{equation*}
\lambda\maps \pi_{12}^{*}P \otimes \pi_{23}^{*}P \to \pi_{13}^{*}P
\end{equation*}
over $Y^{[3]}$ that is associative in the sense that the diagram
\begin{equation*}
\alxydim{@R=1.2cm@C=2cm}{\pi_{12}^{*}P \otimes \pi_{23}^{*}P \otimes  \pi_{34}^{*}P \ar[r]^-{\id \otimes \pi_{234}^{*}\lambda} \ar[d]_{\pi_{123}^{*}\lambda \otimes \id} & \pi_{12}^{*}P \otimes \pi_{24}^{*}P \ar[d]^{\pi_{124}^{*}\lambda} \\ \pi_{13}^{*}P \otimes \pi_{34}^{*}P \ar[r]_-{\pi_{134}^{*}\lambda} & \pi_{14}^{*}P}
\end{equation*}
of bundle morphisms  over $Y^{[4]}$ is commutative.

\item

An isomorphism $\varphi\maps P \to P'$ of principal $A$-bundles with  products $\lambda$ and $\lambda'$, respectively, is called \emph{product-preserving}, if the diagram
\begin{equation*}
\alxydim{@=1.2cm}{\pi_{12}^{*}P \otimes \pi_{23}^{*}P \ar[r]^-{\lambda} \ar[d]_{\pi_{12}^{*}\varphi \otimes \pi_{23}^{*}\varphi} & \pi_{13}^{*}P \ar[d]^{\pi_{13}^{*}\varphi} \\\pi_{12}^{*}P' \otimes \pi_{23}^{*}P' \ar[r]_-{\lambda'} & \pi_{13}^{*}P'}
\end{equation*}
of bundle morphisms over $Y^{[3]}$ is commutative.
\end{enumerate}
\end{definition}

Bundles over $Y^{[2]}$  with products,  and product-preserving isomorphisms form a monoidal groupoid in an evident way. A connection on $P$ is called called \emph{compatible} with a product $\lambda$, if $\lambda$ is a connection-preserving bundle morphism. We denote by $\Delta\maps Y \to Y^{[2]}$ the diagonal map. Using dual bundles one can show:
\begin{lemma}
\label{lem:prodbuntriv}
If $P$ is a principal $A$-bundle over $Y^{[2]}$, a product $\lambda$ determines a section of $\Delta^{*}P$. If $P$ carries a  connection compatible with $\lambda$, this section is flat. 
\end{lemma}

The loop space $LX$ of $X$ is the diffeological space of smooth maps $\tau\maps S^1 \to X$, and the path space $PX$ is the diffeological space of smooth maps $\gamma\maps [0,1] \to X$ with \quot{sitting instants}, see \cite[Section 2]{waldorf9}. We will frequently need two important  maps related to spaces of paths and loops: the evaluation map 
\begin{equation*}
\ev\maps \p X \to X \times X\maps \gamma \mapsto (\gamma(0),\gamma(1))\text{,}
\end{equation*}
and the map $l$ which has already appeared in Section \ref{sec:results},
\begin{equation*}
l\maps \p X^{[2]} \to L X\maps (\gamma_1,\gamma_2) \mapsto \prev{\gamma_2} \pcomp \gamma_1\text{,}
\end{equation*}
where $PX^{[k]}$ denotes the $k$-fold fibre product of $PX$ over $X \times X$.
Both maps are smooth, and $\ev$ is a subduction if $X$ is connected, see \cite[Section 2]{waldorf9}.

\begin{definition}
\label{def:fusionbundle}
A \emph{fusion product} on a principal $A$-bundle $P$ over $LX$  is a product $\lambda$ on the bundle $l^{*}P$ over $PX^{[2]}$ in the sense of Definition \ref{def:product}. A pair $(P,\lambda)$ is called \emph{fusion bundle}. A morphism $\varphi\maps P_1 \to P_2$ between fusion bundles is called fusion-preserving, if $l^{*}\varphi\maps l^{*}P_1 \to l^{*}P_2$ is product-preserving. 
\end{definition}

We denote the monoidal groupoid formed by fusion bundles and fusion-preserving bundle morphisms by $\fusbun A {LX}$. A connection on a principal $A$-bundle $P$ over $LX$ is \emph{compatible} with a fusion product $\lambda$, if its pullback  to $l^{*}P$ is compatible with the product $\lambda$. 
 Let $\Delta \maps PX \to LX$ be the smooth map that assigns to a path $\gamma\in PX$ the loop $l(\gamma,\gamma) \in LX$.  Lemma \ref{lem:prodbuntriv} implies:
\begin{lemma}
\label{lem:fusbuntrivflat}
For every fusion bundle $P$  over $LX$ with compatible connection, the fusion product determines a flat section  $\can\maps PX \to \Delta^{*}P$.
\end{lemma}

Notice that one can further pullback this section along $X \to PX\maps x \mapsto \id_x$, so that the restriction of any fusion bundle to the constant loops has a flat trivialization.

Next we explain the additional condition for connections on fusion bundles, which we have mentioned in the introduction as (iv). For an angle $\varphi \in [0,2\pi]$, we denote by $r_{\varphi}\maps LX \to LX$ the loop rotation by $\varphi$. Consider  the \quot{rotation homotopy} $h_{\varphi} \maps [0,1] \to LS^1$ defined by $h_{\varphi}(t)(s)=s \mathrm{e}^{\im t\varphi}$. For any loop $\beta\in LX$, the composite $H_{\varphi}(\beta) \df L\beta \circ h_{\varphi}$ is a path in $LX$ connecting $\beta$ with $r_{\varphi}(\beta)$. 
The parallel transport of the connection on $P$ along the path $H_{\varphi}(\beta)$ defines a bundle automorphism $R_{\varphi}$ of $P$ that covers the rotation $r_{\varphi}$.

We are particularly interested in the half-rotation, i.e. $\varphi=\pi$. Then, the loop rotation  $r_{\pi}$ lifts along the map $l\maps PX^{[2]} \to LX$. Namely, we define the smooth maps
\begin{equation*}
\pi\maps PX^{[n]} \to PX^{[n]}\maps (\gamma_1,...,\gamma_n) \mapsto (\prev {\gamma_n},...,\prev{\gamma_1})\text{,}
\end{equation*}
and obtain the commutative diagram
\begin{equation*}
\alxydim{@=1.2cm}{PX^{[2]} \ar[d]_{\pi} \ar[r]^{l} & LX \ar[d]^{r_{\pi}} \\ PX^{[2]} \ar[r]_{l} &  LX\text{.}}
\end{equation*}
As a consequence, the bundle automorphism $R_{\pi}$ pulls back to an automorphism of $l^{*}P$ that covers the map $\pi$; we denote that also by $R_{\pi}$. 

\begin{definition}
\label{def:symmetrizing}
Let $P$ be a principal $A$-bundle over $LX$ with fusion product $\lambda$. A  connection on $P$ \emph{symmetrizes}  $\lambda$, if
\begin{equation*}
R_{\pi}(\lambda(q_1 \otimes q_2)) = \lambda(R_{\pi}(q_2) \otimes R_{\pi}(q_1))
\end{equation*}
for all  $(\gamma_1,\gamma_2,\gamma_3) \in PX^{[3]}$ and all   $q_1 \in P_{l(\gamma_1,\gamma_2)}$ and $q_2\in P_{l(\gamma_2,\gamma_3)}$.
\end{definition}

\begin{remark}
 It is instructive to assume for a moment that the bundle $P$ together with its connection is pulled back from the \emph{thin loop space} $\lt X$, i.e. the diffeological space of thin homotopy classes of loops, along the projection $\mathrm{pr}\maps LX \to \lt X$. Since the rotation $r_{\varphi}$ projects to the \emph{identity} on $\lt X$, the automorphisms $R_{\varphi}$ are also identities. Then, the condition of Definition \ref{def:symmetrizing} is simply
\begin{equation*}
\lambda(q_1 \otimes q_2) = \lambda(q_2 \otimes q_1)\text{,}
\end{equation*}
i.e. the fusion product is \emph{commutative}.
\end{remark}

A \emph{connection on a fusion bundle} $(P,\lambda)$ is by definition  a connection on $P$ which is compatible with $\lambda$ and symmetrizes $\lambda$. A \emph{morphism between fusion bundles with connections} is a fusion-preserving, connection-preserving bundle morphism. Fusion bundles with connection form a category that we denote by $\fusbuncon A {LX}$. For the Hom-sets of this category, we provide the following lemma used in Corollary \ref{co:homsets}. 

\begin{lemma}
\label{lem:fusiontorsor}
Let $P_1$ and $P_2$ be fusion bundles with connections. Then, the set of connection-preserving, fusion-preserving bundle morphisms $\varphi\maps P_1 \to P_2$ is either empty or a torsor over the group $\fuslc A {LM}$ of locally constant fusion maps. 
\end{lemma}

\begin{proof}
By \cite[Lemma 3.2.4]{waldorf9}, the set of connection-preserving isomorphisms between isomorphic diffeological principal bundles is a torsor over the group of locally constant maps $f \maps LX \to A$.  Adding the condition that these isomorphisms respect the fusion product results in the condition that  $f$ is a fusion map.     
\end{proof}

\subsection{Superficial Connections}

A loop in the loop space, $\tau\in LLX$, is called \emph{thin}, if the adjoint map $\exd\tau\maps S^1 \times S^1 \to X$ defined by $\exd\tau(s,t) \df \tau(s)(t)$ has rank one (in smooth manifolds, this means that the rank of its differential is less or equal than one; in general see \cite[Definition 2.1.1]{waldorf9}). Two loops $\tau_1,\tau_2 \in LLX$ are called \emph{rank-two-homotopic}, if there exists a path $h\in PLLX$ connecting $\tau_1$ with $\tau_2$, whose adjoint map $\exd h\maps [0,1] \times S^1 \times S^1 \to X$ has rank two. 

\begin{definition}
\label{def:superficial}
A connection on a principal $A$-bundle $P$ over $LX$ is called \emph{superficial}, if two conditions are satisfied:
\begin{enumerate}[(i)]
\item 
Thin loops have vanishing holonomy.

\item
Rank-two-homotopic loops have the same holonomy.
\end{enumerate}
\end{definition}

We remark that condition (ii) is a weakened flatness condition, i.e. a flatness implies (ii). If $A$ is a discrete group, then the unique connection on $P$ is superficial.

For later use we shall derive two implications for the parallel transport of superficial connections. Corresponding to the notion of a thin loop is the one of a \emph{thin path}, i.e. one whose adjoint map $[0,1] \times S^1 \to X$ has rank one. For $\gamma\maps [0,1] \to LX$ a path in $LX$, we denote the parallel transport of a connection on a principal $A$-bundle $P$ over $LX$ by
$\ptr{\gamma}\maps P_{\gamma(0)} \to P_{\gamma(1)}$.

\begin{lemma}
\label{lem:thinpath}
Let $P$ be a principal $A$-bundle over $LX$ with superficial connection. If $\gamma_1,\gamma_2 \in PLX$ are thin paths with a common initial point $\beta_0\in LX$  and a common endpoint $\beta_1\in LX$, then $\ptr{\gamma_1} = \ptr{\gamma_2}$. \end{lemma}

\begin{proof}
$l(\gamma_1,\gamma_2) = \prev{\gamma_2} \pcomp \gamma_1$ is a thin loop in $LX$, and has vanishing holonomy by (i). 
\end{proof}

Loops connected by a thin path are called \emph{thin homotopic}, see \cite{barret1}. The lemma states that a superficial connection provides   identifications between  fibres of $P$ over thin homotopic loops, independent of the choice of a thin path. To prepare the second implication, we notice that a rank-two-homotopy $h$ between two paths $\gamma_1,\gamma_2 \nobr \in  \nobr  PLX$  induces paths $h_0,h_1\in PLX$ with $h_0(t) \df h(0)(t)$ and $h_1(t) \df  h(1)(t)$; these connect the end-loops of the paths $\gamma_1$ and $\gamma_2$ with each other. 

\begin{lemma}
\label{lem:pointwisethin}
Let $P$ be a principal $A$-bundle over $LX$ with superficial connection.
Let $h$ be a rank-two-homotopy between paths $\gamma_1$ and $\gamma_2$ in $LX$.  Then, the diagram
\begin{equation*}
\alxydim{@=1.2cm}{P_{\gamma_1(0)} \ar[d]_{\ptr{h_0}} \ar[r]^{\ptr{\gamma_1}} & P_{\gamma_1(1)} \ar[d]^{\ptr{h_1}} \\ P_{\gamma_2(0)} \ar[r]_{\ptr{\gamma_2}} & P_{\gamma_2(1)}}
\end{equation*}
of parallel transport maps is commutative. In particular, if $h$  fixes the end-loops, i.e. the paths $h_0$ and $h_1$ are constant, then $\ptr{\gamma_1}=\ptr{\gamma_2}$. \end{lemma}

\begin{proof}
We first prove the statement for a homotopy $h$ that fixes the end-loops. Then we can form, for each $t$, a loop $h'(t) \df  l(h(t),\id_{\gamma_2}(t)) \in LLX$. These loops form a path $h' \in PLLX$. Its adjoint map still has rank two, so that condition (ii) implies that the holonomies of the loops $h'(0)$ and $h'(1)$ coincide. We calculate
\begin{equation}
\label{eq:loophol}
\mathrm{Hol}_P(l(\gamma_1,\gamma_2)) = \mathrm{Hol}_P(h'(0)) = \mathrm{Hol}_P(h'(1)) = \mathrm{Hol}_P(l(\gamma_2,\gamma_2))=1\text{.}
\end{equation}
The last step uses that the loop $l(\gamma_2,\gamma_2)$ in $LX$ is thin homotopic to a constant loop (which has vanishing holonomy), and
the holonomy of a connection depends only on the thin homotopy class of the loop (this is well-known for smooth bundles; for diffeological bundles see \cite[Proposition 3.2.10]{waldorf9}. As a consequence of \erf{eq:loophol}, we obtain $\ptr{\gamma_1}=\ptr{\gamma_2}$.

Now we deduce  the general statement. 
On the one hand, we  have $\ptr{\gamma_1} = \ptr{\id \pcomp \gamma_1 \pcomp \id}$, since these paths are thin homotopic. On the other hand,  the paths $(\id \pcomp \gamma_1) \pcomp \id$ and $(\prev{h_1} \pcomp \gamma_2) \pcomp h_0$ are rank-two-homotopic and have the same end-loops, so that the commutativity of the diagram follows. 
\end{proof}

\begin{remark}
\label{rem:rank}
The parallel transport of \emph{any} connection on a principal bundle over \emph{any} diffeological space  only depends on the thin homotopy class of the path.  Lemma \ref{lem:pointwisethin} shows that for a \emph{superficial} connection this independence is extended to \emph{larger} equivalence classes of paths. Indeed, consider $\gamma_1,\gamma_2\in PLX$. Let $h\maps [0,1] \to PLX$ be a thin homotopy.  Then, its adjoint $\exd h\maps [0,1]^2 \times S^1 \to X$ has rank two. The converse is in general not true. 
\end{remark}

An important feature of superficial connections is their relation to equivariant structures. In the context of transgression, this relation explains the absence of equivariant structures in our definition of fusion bundles, see Remark \ref{re:equiv}. Let $\diff^{+}(S^1)$ denote the group of orientation-preserving diffeomorphisms of $S^1$.

\begin{proposition}
\label{prop:diffequiv}
Let $P$ be a principal $A$-bundle over $LX$ with superficial connection. Then, parallel transport determines a connection-preserving, $\diff^{+}(S^1)$-equivariant structure on $P$.
\end{proposition} 

\begin{proof}
For every $\phi \in \diff^{+}(S^1)$ and every $\beta\in LX$ the loops $\beta$ and $\beta \circ \phi$ are \emph{thin homotopic}. Indeed, since $\diff^{+}(S^1)$ is connected, there exists a path $\gamma_{\phi} \in P(LS^1)$ with $\gamma_{\phi}(0)= \id_{S^1}$ and $\gamma_\phi(1)=\phi$. Then, the path $L\beta \circ \gamma_{\phi} \in PLX$ connects $\beta$ with $\beta \circ \phi$ and is thin.

Lemma \ref{lem:thinpath} implies that the parallel transport $\tau_{L\beta \circ \gamma_{\phi}}$ of the superficial connection on $P$ along the thin path $L\beta \circ \gamma_{\phi} \in PLX$ is independent of the choice of the path $\gamma_{\phi}$. Thus we have a well-defined map
\begin{equation*}
E\maps \diff^{+}(S^1) \times P \to P\maps (\phi,x) \mapsto \ptr{\gamma_{Lp(x) \circ \gamma_{\phi}}}(x)
\end{equation*}
where $p(x) \in LX$ is the projection of $x\in P$ to the base. The map $E$ is smooth because parallel transport depends smoothly on the path \cite[Proposition 3.2.10]{waldorf9} and the paths $\gamma_{\phi}$ can be chosen \emph{locally} in a smooth way. Further,  $E$ is an action of $\diff^{+}(S^1)$ on $P$: for $\phi,\psi \in \diff^{+}(S^1)$ and paths $\gamma_{\phi},\gamma_{\psi},\gamma_{\psi \circ \varphi} \in PLS^1$ we have the thin paths
\begin{equation*}
L\beta \circ \gamma_{\psi \circ \varphi}
\quand
(L(\beta \circ \phi) \circ \gamma_{\psi}) \pcomp (L\beta \circ \gamma_{\varphi})
\end{equation*}
both connecting $\beta$ with $\beta \circ \psi \circ \varphi$. Thus, the induced parallel transport maps coincide, and the functorality of parallel transport gives
\begin{equation*}
\tau_{L\beta \circ \gamma_{\psi \circ \varphi}} = \tau_{L(\beta \circ \phi) \circ \gamma_{\psi}} \circ \tau_{L\beta \circ \gamma_{\varphi}}\text{.}
\end{equation*}
This shows  $E(\psi,E(\varphi,x)) = E(\psi \circ \varphi,x)$. Finally, the action $E$ covers the action of $\diff^{+}(S^1)$ on $LX$. Thus, $E$ is an equivariant structure on $P$.

The claim that $E$ is connection-preserving means that the bundle automorphism $E_{\phi}\maps P \to P$ is connection-preserving for all $\phi \in \diff^{+}(S^1)$.  If $\gamma \in PLX$ is a path, we denote by $\phi(\gamma)$ the path $\phi(\gamma)(t) \df  \gamma(t) \circ \phi$. Any choice of a path $\gamma_{\phi}$ from $\id_{S^1}$ to $\phi$ induces a rank-two-homotopy between $\gamma$ and $\phi(\gamma)$, and Lemma \ref{lem:pointwisethin} implies that the diagram
\begin{equation*}
\alxydim{@=1.2cm}{P_{\gamma(0)} \ar[d]_{E_{\phi}} \ar[r]^{\ptr\gamma} & P_{\gamma(1)} \ar[d]^{E_{\phi}} \\ P_{\gamma^{\phi}(0)} \ar[r]_{\ptr{\phi(\gamma)}} & P_{\gamma^{\phi}(1)}}
\end{equation*}
is commutative. This shows that $E_{\varphi}$ is connection-preserving. 
\end{proof}

Now we have completed the definition of 
fusion bundles with superficial connection and have thus explained all details of Definition \ref{def:fusbuncon}. 
\label{sec:fusionforms}
It is interesting to look at superficial connections on the \emph{trivial} fusion bundle, i.e. the trivial bundle over $LX$ with the identity fusion product. Like in the introduction, we shall restrict to $A=\ueins$ with the Lie algebra identified with $\R$.   Superficial connections on the trivial fusion bundle form a group that we denote by $\Omegafus^1(LX)$; it consists of 1-forms $\eta \in \Omega^1(LX)$ which 
\begin{enumerate}[(a)]

\item 
are \emph{compatible} with the identity fusion product:
\begin{equation*}
e_{12}^{*}\eta + e_{23}^{*}\eta = e_{13}^{*}\eta\text{,}
\end{equation*}
with the maps $e_{ij} \df  l \circ \mathrm{pr}_{ij}\maps PX^{[3]} \to LX$.  

\item
\emph{symmetrize} the identity fusion product:
\begin{equation*}
\int_{H_{\pi}(l(\gamma_1,\gamma_2))} \eta \;+\; \int_{H_{\pi}(l(\gamma_2,\gamma_3))} \eta \;=\; \int_{H_{\pi}(l(\gamma_1,\gamma_3))} \eta \quad\mod\quad \Z\text{,}
\end{equation*}
for all $(\gamma_1,\gamma_2,\gamma_3) \in PX^{[3]}$, where $H_\pi({\beta})\in PLX$ denotes the homotopy that rotates a loop $\beta$ by an angle of $\pi$.  

\item
are \emph{superficial}:
\begin{equation*}
\int_{\tau} \eta \in \Z
\quand
\int_{\tau_1}  \eta \;=\; \int_{\tau_2} \eta \quad\mod\quad \Z
\end{equation*}
for all thin loops $\tau \in LLX$ and all pairs $(\tau_1,\tau_2)$ of rank-two-homotopic loos in $LX$.
\end{enumerate}
The formulas in (b) and (c) come from the fact that the parallel transport in the trivial bundle is given by  multiplication with the exponential of the integral of the connection 1-form, so that a coincidence of parallel transport maps is equivalent to the integrality of the difference of the integrals. We close this section with a fact used in Corollary \ref{co:forms}: 

\begin{lemma}
\label{lem:dlogfus}
For $\eta_1,\eta_2\in \Omegafus^1(LX)$, there is a bijection
\begin{equation*}
\hom (\trivlin_{\eta_1},\trivlin_{\eta_2}) \cong  \left \lbrace  f \in \fus \ueins {LX} \;|\; \eta_2 = \eta_1 + \mathrm{dlog}(f) \right \rbrace\text{.}
\end{equation*} 
\end{lemma}

\begin{proof}
A morphism between trivial bundles over $LX$ is the same as a smooth map $f\maps LX \to A$. It preserves  connection 1-forms $\eta_1$ and $\eta_2$ if and only if $\eta_2 \nobr = \nobr \eta_1 \nobr + \nobr \mathrm{dlog}(f)$.  It preserves the trivial fusion products if and only $f$ is a fusion map.
\end{proof}

\setsecnumdepth{2}

\section{Diffeological Bundle Gerbes}

\label{sec:diffgrb}

\label{sec:gerbeconnections}

\subsection{Bundle Gerbes}

Let $X$ be a diffeological space.

\begin{definition}
A \emph{diffeological $A$-bundle gerbe} over $X$ is a subduction $\pi\maps Y \to X$ and a bundle $P$ over $Y^{[2]}$ with a product $\lambda$ in the sense of Definition \ref{def:product}. 
\end{definition}

Compared to the classical definition of a bundle gerbe \cite{murray}, we have made two independent generalizations. First we have admitted diffeological spaces and diffeological bundles instead of smooth manifolds and smooth bundles. Secondly, we have admitted a general abelian Lie group $A$ instead of $\C^{\times}$ or $\ueins$. Almost all  statements about $\ueins$-bundle gerbes over smooth manifolds extend without changes to our generalized version of bundle gerbes, mostly due to the fact that principal $A$-bundles over diffeological spaces have all the features of principal $\ueins$-bundles over smooth manifolds \cite[Section 3]{waldorf9}. A list of such  statements that we will use later is:
\begin{lemma}
\label{lem:diffgrb}
\begin{enumerate}[(i)]
\item 
Diffeological $A$-bundle gerbes over a diffeological space $X$ form a monoidal 2-groupoid $\diffgrb A X$. The tensor unit is denoted $\mathcal{I}$: it has the identity subduction and the trivial bundle with the trivial product.

\item The assignment  $X \mapsto \diffgrb AX$ defines a sheaf of 2-groupoids over the site of diffeological spaces; in particular, there are coherent pullback 2-functors.

\item
Let $\mathcal{G}_1$ and $\mathcal{G}_2$ be diffeological bundle gerbes over a diffeological space $X$. Then, the Hom-category $\hom(\mathcal{G}_1,\mathcal{G}_2)$ is a torsor category over the monoidal groupoid $\diffbun A X$ of diffeological principal $A$-bundles over $X$. 

\item
If $f\maps Y_1 \to Y_2$ is a smooth map between subductions over $X$, and $\mathcal{G}$ is a diffeological bundle gerbe with subduction $\pi_2\maps Y_2 \to X$, one obtains via pullback of the bundle and its product another diffeological bundle gerbe $\mathrm{res}_f(\mathcal{G})$ with subduction $\pi_1\maps Y_1 \to X$, together with an isomorphism $\mathrm{res}_f(\mathcal{G}) \cong \mathcal{G}$.
\end{enumerate}
\end{lemma}

\noindent
Over smooth manifolds, \emph{smooth} and \emph{diffeological} bundle gerbes are  the same

\begin{theorem}
\label{th:diffvssmoothgrb}
Let $M$ be a smooth manifold. Then, the inclusion
\begin{equation*}
 \grb A M \to \diffgrb A M
\end{equation*}
of smooth bundle gerbes into diffeological bundle gerbes is an equivalence of 2-categories. 
\end{theorem}

\begin{proof}
 It is an equivalence between the Hom-categories because these are torsor categories over $\bun AX$ and $\diffbun AX$, respectively, and both categories are isomorphic \cite[Theorem 3.1.7]{waldorf9}. 
It remains to prove that it is essentially surjective on objects. Suppose $\mathcal{G}$ is a diffeological bundle gerbe over $M$. Let $M$ be covered by open sets $U_{\alpha}$ with smooth sections $s_{\alpha}\maps   U_{\alpha} \to Y$. For $U$ the disjoint union of the open sets $U_{\alpha}$, the sections define a smooth map $s\maps   U \to Y$, covering the identity on $M$. Then, $\mathrm{res}_{s}(\mathcal{G})$ is a \emph{smooth} bundle gerbe over $M$, again using \cite[Theorem 3.1.7]{waldorf9}. By Lemma \ref{lem:diffgrb} (iv), $\mathrm{res}_s(\mathcal{G})$ is  isomorphic to $\mathcal{G}$.
\end{proof}

\begin{corollary}
Isomorphism classes of diffeological $A$-bundle gerbes over a smooth manifold $M$ are in bijection to the \v Cech cohomology group $\check \h^2(M,\sheaf A)$.
\end{corollary}

Here $\sheaf A$ denotes the sheaf of smooth $A$-valued functions on $M$. We denote the class that represents a bundle gerbe $\mathcal{G}$ in cohomology by $\check c(\mathcal{G}) \in \check \h^2(M,\sheaf A)$.

One important difference between $\ueins$-bundle gerbes and general $A$-bundle gerbes  is the following.
For $\ueins$-bundle gerbes over a smooth manifold $M$ one can use the isomorphism $\check \h^2(M,\sheaf \ueins) \cong \h^3(M,\Z)$, and regard the image of the  class $\check c(\mathcal{G})$ in $\h^3(M,\Z)$ as a characteristic class. For a general $A$-bundle gerbes, this is not possible. In particular, while every $\ueins$-bundle gerbe over a surface trivializes by dimensional reasons, there are  non-trivial $A$-bundle gerbes over surfaces.

Our discussion of connections on diffeological bundle gerbes uses the fact that differential forms on diffeological spaces and connections on diffeological principal bundles have almost\footnote{One difference is that not every  principal bundle over every diffeological space admits a connection. Note that this is already the case over some infinite-dimensional (Fréchet) manifolds.} all features of differential forms and connections over smooth manifolds \cite[Sections 3.2 and A.3]{waldorf9}.

\begin{definition}
\label{def:conn}
Let $\mathcal{G}$ be a diffeological $A$-bundle gerbe over $X$, consisting of a subduction $\pi\maps  Y \to X$ and $A$-bundle $P$ with product $\lambda$. A \emph{connection} on $\mathcal{G}$ is a connection on $P$ compatible with $\lambda$, and a 2-form $B \in \Omega^2_{\mathfrak{a}}(Y)$ such that 
\begin{equation*}
\mathrm{curv}(P) = \pi_2^{*}B- \pi_1^{*}B\text{.}
\end{equation*}
\end{definition}
Here we have denoted by $\Omega^k_{\mathfrak{a}}(X)$ the vector space of $\mathfrak{a}$-valued $k$-forms on $X$, where $\mathfrak{a}$ is the Lie algebra of $A$. The 2-form $B$ is called the \emph{curving} of the connection.
Every diffeological $A$-bundle gerbe $\mathcal{G}$ with connection over $X$ has a curvature $\mathrm{curv}(\mathcal{G}) \in \Omega^3_{\mathfrak{a}}(X)$, which is defined as the unique 3-form whose pullback along the subduction of $\mathcal{G}$ equals $\mathrm{d}B$. If the Lie group $A$ is discrete, a connection is  \emph{no} information, i.e. every $A$-bundle gerbe has exactly one connection, which is flat.

Many familiar facts  generalize from smooth $\ueins$-bundle gerbes with connection to diffeological $A$-bundle gerbes with connection. For example, a connection on the trivial bundle gerbe $\mathcal{I}$ can be identified with its curving, i.e. with a 2-form $\rho \in \Omega^2_{\mathfrak{a}}(X)$. We denote this bundle gerbe with connection by $\mathcal{I}_{\rho}$. Further, we have:

\begin{lemma}
\label{lem:diffgrbcon}
\begin{enumerate}[(i)]
\item
Diffeological $A$-bundle gerbes with connection over $X$ form a monoidal 2-groupoid $\diffgrbcon AX$, whose tensor unit is $\mathcal{I}_0$.

\item
The assignment $X \mapsto \diffgrbcon AX$ defines a sheaf of 2-groupoids over the site of diffeological spaces; in particular, there are coherent pullback 2-functors.

\item
Let $\mathcal{G}_1$ and $\mathcal{G}_2$ be diffeological $A$-bundle gerbes with connection over a diffeological space $X$. Then, the Hom-category $\hom(\mathcal{G}_1,\mathcal{G}_2)$ is a torsor category over the monoidal groupoid $\diffbunconflat A X$ of flat diffeological principal $A$-bundles over $X$. 
\end{enumerate}
\end{lemma}

\noindent
A straightforward analog of Theorem \ref{th:diffvssmoothgrb} is the following
\begin{theorem}
\label{th:diffvssmoothgrbcon}
The inclusion 
\begin{equation*}
\grbcon A M \to \diffgrbcon A M
\end{equation*}
of smooth bundle gerbes with connection into diffeological bundle gerbes with connection is an equivalence of 2-categories.
\end{theorem} 

\begin{corollary}
Isomorphism classes of diffeological $A$-bundle gerbes with connection over a smooth manifold $M$ are in bijection with the Deligne cohomology group $\check \h^2(M,\mathcal{D}_A(2))$.
\end{corollary}

Deligne cohomology (with $A = \C^{\times}$) is discussed in detail in \cite{brylinski1}. In our case, $\mathcal{D}_A(2)$ denotes the complex
\begin{equation*}
\alxydim{}{\sheaf A  \ar[r]^-{\mathrm{dlog}} & \sheaf\Omega^1_{\mathfrak{a}} \ar[r]^-{\mathrm{d}} & \sheaf\Omega^2_{\mathfrak{a}}}
\end{equation*}
of sheaves over $M$. Here, $\sheaf\Omega^k_{\mathfrak{a}}$ denotes the sheaf of $k$-forms with values in the Lie algebra $\mathfrak{a}$ of $A$, $\mathrm{d}$ denotes the exterior derivative, and $\mathrm{dlog}(f)$ denotes the pullback of the Maurer-Cartan form $\theta\in\Omega^1_{\mathfrak{a}}(A)$ along a smooth map $f\maps  U \to A$.  There is a projection
\begin{equation*}
\check \h^2(M,\mathcal{D}_A(2)) \to \check \h^2(M,\sheaf A)
\end{equation*}
that reproduces the characteristic class $\check c(\mathcal{G})$ of the underlying bundle gerbe. We also recall that there is an exact sequence
\begin{equation}
\label{eq:deligne}
\alxydim{@C=1.2cm}{0 \ar[r] & \h^2(M,A) \ar[r] & \check \h^2(M,\mathcal{D}_A(2)) \ar[r]^-{\mathrm{curv}} & \Omega_{\mathfrak{a}}^3(M)}
\end{equation}
saying that diffeological $A$-bundle gerbes with \emph{flat} connection are parameterized by the (singular) cohomology group $\h^2(M,A)$.

If one chooses a \v Cech resolution involving an open cover $\mathscr{U}$ of $M$, then a cocycle $c$ for a gerbe with connection is a tuple $c=(B,A,g)$, with smooth functions $g_{\alpha\beta\gamma} \maps U_\alpha \nobr \cap\nobr U_\beta \nobr\cap\nobr U_\gamma \to A$  and differential forms $A_{\alpha\beta} \in \Omega^1_{\mathfrak{a}}(U_{\alpha} \cap U_{\beta})$ and $B_\alpha \in \Omega^2_{\mathfrak{a}}(U_\alpha)$. We refer to \cite[Section 3]{gawedzki8} 
 for   a detailed discussion of the relation between Deligne cohomology and bundle gerbes.

\subsection{Trivializations}

\label{app:trivs}

Trivializations play an important role in this article because the serve as \quot{boundary conditions} for the surface holonomy that we explain in Section \ref{sec:holonomy}. 

\begin{definition}
\label{def:triv}
A \emph{trivialization} of a diffeological bundle gerbe $\mathcal{G}$ with connection over a diffeological space $X$ is  a 2-form $\rho\in  \Omega_{\mathfrak{a}}^2(X)$ and a 1-isomorphism $\mathcal{T}\maps  \mathcal{G} \to \mathcal{I}_{\rho}$ in the 2-category $\diffgrbcon AX$ of diffeological bundle gerbes with connection over $X$.
\end{definition}

In more detail, if $\pi\maps Y \to X$ is the subduction of $\mathcal{G}$,  $P$ is its bundle with product $\lambda$, and $B$ is its curving, a trivialization  $\mathcal{T}$ is a pair $(T,\tau)$  consisting of a principal $A$-bundle $T$ with connection over $Y$ of curvature $\mathrm{curv}(T) = B-\pi^{*}\rho$, and of a connection-preserving isomorphism
\begin{equation*}
\tau\maps  P \otimes \pi_2^{*}T \to \pi_1^{*}T
\end{equation*}
of bundles over $Y^{[2]}$ that is compatible with the product $\lambda$ in the sense that the diagram
\begin{equation*}
\alxydim{@C=1.5cm@R=1.2cm}{\pi_{12}^{*}P \otimes \pi_{23}^{*}P \otimes \pi_3^{*}T \ar[r]^-{\id \otimes \pi_{23}^{*}\tau} \ar[d]_{\lambda \otimes \id} & \pi_{12}^{*}P \otimes \pi_2^{*}T \ar[d]^{\pi_{12}^{*}\tau} \\ \pi_{13}^{*}P \otimes \pi_3^{*}T \ar[r]_-{\tau} & \pi_1^{*}T }
\end{equation*}
of bundle morphisms over $Y^{[3]}$ is commutative. A 2-isomorphism $\varphi\maps  \mathcal{T}_1 \Rightarrow \mathcal{T}_2$ between two trivializations $\mathcal{T}_i = (T_i,\tau_i)$ is a connection-preserving bundle isomorphism $\varphi\maps  T_1 \to T_2$ such that $\pi_1^{*}\varphi \circ \tau = \tau \circ (\id \otimes \pi_2^{*}\varphi)$.

In Deligne cohomology, a trivialization $\mathcal{T}\maps \mathcal{G} \to \mathcal{I}_{\rho}$ of a bundle gerbe $\mathcal{G}$ with connection and with a Deligne cocycle $c$ allows to extract a 1-cochain $t$ such that $c \eq \mathrm{D}(t) + r(\rho)$, where $\mathrm{D}$ is the differential of the Deligne complex, and $r\maps  \Omega^2_{\mathfrak{a}}(M) \to \sheaf\Omega^2_{\mathfrak{a}}$ is the restriction map. A 2-isomorphism $\varphi\maps  \mathcal{T}_1 \Rightarrow \mathcal{T}_2$ defines  a  Deligne 0-cochain $f$ such that $t_1 = t_2 + \mathrm{D}(f)$.

The action of the groupoid $\diffbunconflat AX$ on the Hom-categories of $\diffgrbcon AX$  from Lemma \ref{lem:diffgrbcon} (iii) restricts to an action on trivializations with fixed $\rho$, which we denote by $\mathcal{T} \otimes P$ for  a trivialization $\mathcal{T}$ and  a flat principal $A$-bundle $P$ over $X$.
 For $\mathcal{T}=(T,\tau)$  we have $\mathcal{T} \otimes P \df  (T \otimes \pi^{*}P, \tau \otimes \id)$. For $\varphi\maps  \mathcal{T}_1 \Rightarrow \mathcal{T}_2$ a 2-isomorphism, and $\beta\maps  P_1 \to P_2$ a connection-preserving isomorphism between principal $A$-bundles over $X$, we  have  $\varphi \otimes \beta \df  \varphi \otimes \pi^{*}\beta$.

We shall provide  two facts about trivializations.

\begin{lemma}
\label{lem:trivexist}
Suppose $W$ is a smooth manifold with $\h^3_{\mathrm{dR}}(W) = 0$ and $\h^2(W,A) = 0$, of which the first is de Rham
 cohomology and the second is singular cohomology. Then, every diffeological $A$-bundle gerbe with connection over $W$ admits a trivialization.
\end{lemma}

\begin{proof}
Consider the curvature $H \in \Omega^3_{\mathfrak{a}}(W)$ of the given bundle gerbe $\mathcal{G}$. By assumption, $H=\mathrm{d}\rho$ for some 2-form $\rho \in \Omega_{\mathfrak{a}}^2(W)$. Now, the new bundle gerbe $\mathcal{G} \otimes \mathcal{I}_{-\rho}$ is flat and has a  characteristic class in $\h^2(W,A)=0$. But if $\mathcal{G} \otimes \mathcal{I}_{-\rho}$ has a trivialization, then also $\mathcal{G}$ has one. 
\end{proof}

\begin{lemma}
\label{lem:sectriv}
Suppose $\mathcal{G}$ is a diffeological bundle gerbe with connection with subduction $\pi\maps Y \to X$, curving $B \in \Omega^2_{\mathscr{a}}(Y)$ and a bundle $P$ with product $\lambda$. Then, every smooth section $\sigma\maps  M \to Y$ determines a  trivialization $\mathcal{T}_\sigma\maps  \mathcal{G} \to \mathcal{I}_{\sigma^{*}B}$. Moreover:
\begin{enumerate}[(a)]
\item 
If $\sigma'$ is another section, the product $\lambda$ defines a 2-isomorphism
\begin{equation*}
\mathcal{T}_{\sigma} \otimes R_{\sigma,\sigma'} \Rightarrow \mathcal{T}_{\sigma'}\text{,}
\end{equation*}
where $R_{\sigma,\sigma'}$ is the pullback of $P$ along $(\sigma,\sigma')\maps M \to Y^{[2]}$.

\item
If $\mathcal{T}\maps \mathcal{G} \to \mathcal{I}_{\rho}$  is a trivialization,  the isomorphism $\tau$ of $\mathcal{T}$ defines a 2-iso\-mor\-phism
\begin{equation*}
\mathcal{T}_{\sigma} \otimes \sigma^{*}T \Rightarrow \mathcal{T}\text{,}
\end{equation*}
where $T$ is the principal $A$-bundle of $\mathcal{T}$. 

\end{enumerate}
\end{lemma}

\begin{proof}
Consider the smooth maps 
\begin{equation*}
\sigma_k\maps Y^{[k]} \to Y^{[k+1]}\maps (y_1,...,y_k) \mapsto (y_1,...,y_k,\sigma(\pi(y_1)))\text{.}
\end{equation*}
The trivialization $\mathcal{T}_{\sigma}$ has the principal $A$-bundle $T \df  \sigma_1^{*}P$ over $Y$, and the isomorphism $\tau \df  \sigma_2^{*}\lambda$ over $Y^{[2]}$. The axioms for $\mathcal{T}_{\sigma} \df (T,\tau)$ follow from the relation between the curvature of $P$ and the curving $B$, and from the associativity of $\lambda$ under pullback along $\sigma_3$.  (a) and (b) follow   from this construction.
\end{proof}

\subsection{Surface Holonomy}

\label{sec:holonomy}

Holonomy for bundle gerbes with structure groups different from $\ueins$ requires special attention. 
It is defined in the following situation. Let $\mathcal{G}$ be a diffeological $A$-bundle gerbe with connection over a diffeological space $X$, let $\Sigma$ be a closed oriented two-dimensional smooth manifold and let $\phi\maps \Sigma \to X$ be a smooth map. The holonomy of $\mathcal{G}$ around $\Sigma$,
\begin{equation*}
\mathrm{Hol}_{\mathcal{G}}(\phi) \in A\text{,}
\end{equation*}
is defined as the pairing between the fundamental class of $\Sigma$ and the class in $\h^2(\Sigma,A)$ that corresponds to the flat gerbe $\phi^{*}\mathcal{G}$, see \erf{eq:deligne}. There are two well-known reformulations:
\begin{enumerate}[(a)]

\item
If $X$ is a smooth manifold $M$, one can use the \emph{Alvarez-Gaw\c edzki formula} \cite{alvarez1,gawedzki3}. Choose an open covering  $\mathscr{U}=\left \lbrace U_{\alpha} \right \rbrace_{\alpha\in A}$ of $M$ over which $\mathcal{G}$ allows to extract a Deligne cocycle $c = (B,A,g)$. Let $T$ be a triangulation of $\Sigma$ subordinate to $\phi^{-1}(\mathscr{U})$, i.e. there is a map $\alpha\maps T \to A$ that assigns to each simplex $\sigma \in T$ an index $\alpha(\sigma)$ such that $\phi(\sigma) \subset U_{\alpha(\sigma)}$.
Then, 
\begin{equation}
\label{eq:aghol}
\mathrm{Hol}_{\mathcal{G}}(\phi)= \mathrm{AG}(\phi^{*}c)
\end{equation}
with
\begin{equation*}
\mathrm{AG}(\phi^{*}c) \df \! \prod_{f \in T} \exp \left ( \int_f \phi^{*}B_{\alpha(f)} \right ) \cdot\! \prod_{e \in \partial f} \exp\left( \int_e \phi^{*}A_{\alpha(f)\alpha(e)} \right )\cdot\! \prod _{v\in \partial e} g^{\varepsilon(f,e,v)}_{\alpha(f)\alpha(e)\alpha(v)}(\phi(v))\text{,}
\end{equation*}
where $\epsilon(f,e,v)$ is $+1$ if $v$ is the end of the edge $e$ in the orientation induced from the face $f$, and $-1$ else. Employing the cocycle condition for $c$ one can show explicitly that $\mathrm{AG}(\phi^{*}c)$ does not depend on the choice of the triangulation, and not on the choice of the representative $c$. Further, by going to a triangulation on whose faces the Poincaré Lemma holds and then using Stokes' Theorem, one can transform $\mathrm{AG}(\phi^{*}c)$ into an expression only involving the product over vertices. This way one proves the coincidence \erf{eq:aghol}.

\item 
Suppose the pullback $\phi^{*}\mathcal{G}$ is trivializable (e.g. for $A=\ueins$, where $\check \h^2(\Sigma,\sheaf{\ueins})=\h^3(\Sigma,\Z)=0$). Then one can choose any  trivialization $\mathcal{S}\maps  \phi^{*}\mathcal{G} \to \mathcal{I}_{\rho}$ with some $\rho\in\Omega^2_{\mathfrak{a}}(\Sigma)$ and finds
\begin{equation}
\mathrm{Hol}_{\mathcal{G}}(\phi) = \exp \left ( \int_{\Sigma} \rho \right )\text{.}
\end{equation}
The coincidence between (a) and (b) has been shown explicitly in \cite{carey2}.
\end{enumerate} 

We will frequently use the following property of  the holonomy of bundle gerbes with connection over smooth manifolds. Though quite fundamental, I was unable to find this statement in the literature.
\begin{proposition}
\label{prop:thinhol}
Let $M$ be a smooth manifold and let $\mathcal{G}$ be a diffeological $A$-bundle gerbe with connection over $M$. Suppose that $\Sigma$ is a closed oriented surface, and that $\phi\maps  \Sigma \to M$ is a smooth rank one map. Then, the holonomy of $\mathcal{G}$ around $\Sigma$ vanishes.
\end{proposition}       

\begin{proof}
By Sard's Theorem the image $\phi(\Sigma)$ of  $\phi$ has Hausdorff dimension $\dim_H(\phi(\Sigma)) \leq 1$. With \cite[Theorem VII.2]{hurewicz1},  $\dim_H(X) \geq \dim(X)$
for any topological space $X$, where $\dim(X)$ denotes its topological dimension. \footnote{I thank Martin Olbermann for bringing up this argument.} Now suppose $\mathscr{U}$ is a covering of $M$ by open sets that admits to extract a Deligne cocycle $c$ for $\mathcal{G}$. Since $\phi(\Sigma)$ is compact, it is already covered by finitely many open sets of $\mathscr{U}$. In this case, \cite[Theorem V.1]{hurewicz1} together with the above bound $\dim(X) \leq 1$ guarantees the existence of a refinement $\mathscr{V}$ which has no non-trivial three-fold intersections. 
Using $\phi^{-1}(\mathscr{V})$ in the Alvarez-Gaw\c edzki formula shows that $\mathrm{Hol}_{\mathcal{G}}(\phi)=1$. 
\end{proof}

There is also a version of holonomy for surfaces with boundary, and it will play an important role in this article. Suppose $\Sigma$ is a compact oriented surface with boundary components $b_1,...,b_n$.  Suppose a smooth map $\phi\maps  \Sigma \to M$ is given, and for each boundary component $b_i$ a trivialization $\mathcal{T}_i$ of the restriction of $\phi^{*}\mathcal{G}$ to  $b_i$.  In this situation, holonomy is a number
\begin{equation}
\label{eq:dbranehol}
\mathscr{A}_{\mathcal{G}}(\phi,\mathcal{T}_1,...,\mathcal{T}_n) \in A
\end{equation}
defined in the following way, see e.g. \cite{gawedzki1,carey2,gawedzki4}. One chooses an open cover $\mathscr{U}$ of $M$,  a triangulation of $\Sigma$ subordinate to $\phi^{-1}(\mathscr{U})$, and a Deligne cocycle $c$ representing the bundle gerbe $\mathcal{G}$. Then one chooses, for each boundary component $b_i$, a Deligne 1-cochain $t^i$ representing the trivialization $\mathcal{T}_i$, i.e. $\phi^{*}c|_{b_i} = \mathrm{D}(t^i)$. In the \v Cech resolution with $c=(B,A,g)$,  these cochains are pairs $t^i = (\Pi^i,\chi^i)$ consisting of smooth functions $\chi^i_{\alpha\beta}\maps U_{\alpha} \cap U_{\beta} \to A$ and 1-forms $\Pi^i_{\alpha} \in \Omega^1_{\mathfrak{a}}(U_{\alpha})$. Then, 
\begin{equation}
\label{eq:agwithboundary}
\mathscr{A}_{\mathcal{G}}(\phi,\mathcal{T}_1,...,\mathcal{T}_n) \df  \mathrm{AG}(\phi^{*}c) \cdot   \prod_{i=1}^n \mathrm{BC}(t^i)^{-1}\text{,}
\end{equation}
where the boundary contributions are
\begin{equation}
\label{eq:boundarycontrib}
\mathrm{BC}(t^i) \df  \prod_{e \in b_i} \exp\left (  \int_{e} \Pi_{\alpha(e)}^{i}   \right ) \cdot \prod_{v \in \partial e} \left (\chi^{i}_{\alpha(e)\alpha(v)}(v) \right )^{\varepsilon(e,v)}\text{,}
\end{equation}
and $\varepsilon(e,v)$ is $+1$ when $v$ is the endpoint of $e$ in the induced orientation of $\partial\Sigma$ and $-1$ else. Employing the various relations between the Deligne cochains one can prove that \erf{eq:agwithboundary} is independent of the choice of the triangulation and of the choices of representatives $c$ and $t^i$. It is clear that for a closed surface  $\mathscr{A}_{\mathcal{G}}(\phi)=\mathrm{Hol}_{\mathcal{G}}(\phi)$.

In the situation where a trivialization $\mathcal{S}\maps  \phi^{*}\mathcal{G} \to \mathcal{I}_{\rho}$ over all of $\Sigma$ exists, one has $\mathrm{BC}(t^i) = \mathrm{Hol}_{P_i}(b_i)$, where $P_i$ is a (flat) principal $A$-bundle with connection over $b_i$ such that $\mathcal{T}_i \otimes P_i \cong \mathcal{S}|_{b_i}$. The bundle $P_i$ exists and is unique up to isomorphism due to Lemma \ref{lem:diffgrbcon} (iii). Thus,
\begin{equation}
\label{eq:holtriv}
\mathscr{A}_{\mathcal{G}}(\phi,\mathcal{T}_1,...,\mathcal{T}_n) = \exp \left ( \int_{\Sigma} \rho \right )
\cdot   \prod_{i=1}^n \mathrm{Hol}_{P_i}(b_i)^{-1} \text{,}
\end{equation}
These issues are discussed in detail in the literature \cite{gawedzki1,carey2,gawedzki4}. In \cite{brylinski1} there is also an interpretation of $\mathscr{A}_{\mathcal{G}}$ as a section of a certain bundle over the mapping space $C^{\infty}(\Sigma,M)$. We will come back to this point of view in Section \ref{sec:constrconn}.

\begin{lemma}
\label{lem:dbraneholprop}
Holonomy for surfaces with boundary has the following properties:
\begin{enumerate}[(a)]
\item 
It only depends on the isomorphism classes of the trivializations $\mathcal{T}_i$. Further, if $P$ is a principal $A$-bundle over one of the boundary components $b_i$, then
\begin{equation*}
\mathscr{A}_{\mathcal{G}}(\phi,\mathcal{T}_1,...,\mathcal{T}_i \otimes P, ..., \mathcal{T}_n) = \mathrm{Hol}_{P}(b_i)^{-1} \cdot \mathscr{A}_{\mathcal{G}}(\phi,\mathcal{T}_1,...,\mathcal{T}_n)\text{.}
\end{equation*}

\item
It only depends on the thin homotopy class of the map $\phi$: if $h\maps  [0,1] \times \Sigma \to X$ is a smooth rank two map that fixes the boundary of $\Sigma$ point-wise, and $\phi_t(s) \df  h(t,s)$,  then 
\begin{equation*}
\mathscr{A}_{\mathcal{G}}(\phi_0,\mathcal{T}_1,...,\mathcal{T}_n)= \mathscr{A}_{\mathcal{G}}(\phi_1,\mathcal{T}_1,...,\mathcal{T}_n)\text{.}
\end{equation*}

\item
It satisfies the following \quot{gluing law}. Suppose $\beta\maps  S^1 \to \Sigma$ is a simple loop, so that  $\Sigma' \df  \Sigma \setminus \mathrm{im}(\beta)$ is again a smooth manifold with two new boundary components called $b$ and $\bar b$. If $\phi\maps \Sigma \to X$ is a smooth map, and $\mathcal{T}\maps \phi^{*}\mathcal{G}|_b \to \mathcal{I}_0$ is any trivialization, then
\begin{equation*}
\mathscr{A}_{\mathcal{G}}(\phi,\mathcal{T}_1,...,\mathcal{T}_n) = \mathscr{A}_{\mathcal{G}}(\phi',\mathcal{T}_1,...,\mathcal{T}_n,\mathcal{T},\mathcal{T})\text{,}
\end{equation*}
where $\phi'$ is the restriction of $\phi$ to $\Sigma'$.

\item
It is invariant under orientation-preserving diffeomorphisms: if $\varphi\maps  \Sigma_1 \to \Sigma_2$ is an orientation-preserving diffeomorphism (also preserving the labelling of the boundary components), and $\phi\maps \Sigma_2 \to X$ is a smooth map, then
\begin{equation*}
\mathscr{A}_{\mathcal{G}}(\phi \circ \varphi,\varphi^{*}\mathcal{T}_1,...,\varphi^{*}\mathcal{T}_2) = \mathscr{A}_{\mathcal{G}}(\phi,\mathcal{T}_1,...,\mathcal{T}_n)\text{.}
\end{equation*}

\end{enumerate}
\end{lemma}

\begin{proof}
(a) is clear since 2-isomorphic trivializations have cohomologous 1-cochains, and the tensor factor $P$ contributes exactly the holonomy of $P$ to the formula \erf{eq:agwithboundary}. (b) is a standard argument: the difference of the two holonomies is the integral of the pullback $\phi^{*}H$ of the curvature $H$ of $\mathcal{G}$ over $[0,1] \times \Sigma$; but since $\phi$ has rank two, $\phi^{*}H=0$. To see (c) one uses the same triangulation for $\Sigma$ and $\Sigma'$. Then, the two additional terms in $\mathscr{A}_{\mathcal{G}}(\phi',\mathcal{T}_1,...,\mathcal{T}_n,\mathcal{T},\mathcal{T})$ cancel due to the opposite orientations of $b$ and $\bar b$ as part of the boundary of $\partial\Sigma$. (e) is shown by pulling back a triangulation of $\Sigma_2$ to $\Sigma_1$. 
\end{proof}

In a few parts of this article, we need a yet more advanced version of surface holonomy, namely for surfaces whose boundaries have corners, like the unit square $Q \df  [0,1]^2$. Here we require for each boundary component $b$ a \emph{boundary record} $\mathscr{B} = \left \lbrace \mathcal{T}^{\,r}, \varphi^{rs} \right \rbrace$, a structure which has a trivialization $\mathcal{T}^{\,r}\maps  \phi^{*}\mathcal{G}|_{b^r} \to \mathcal{I}_0$ for each smooth part $b^r$ of the boundary component $b$, and a 2-isomorphism $\varphi^{rs}\maps  \mathcal{T}^{\,r}|_v \Rightarrow \mathcal{T}^{s}|_v$ over each corner $v$ that separates $b^{r}$ from $b^{s}$, in the direction of the orientation of $b$. If a boundary component has no corners, a boundary record is just a trivializations, as before. 
In order to define the holonomy $\mathscr{A}_{\mathcal{G}}$, one chooses Deligne 1-cochains $t^{r}=(G^{r},\Pi^{r})$ for each trivialization $\mathcal{T}^{\,r}$ as before. Then one obtains  Deligne 2-cochains $f^{\,rs}$ for the 2-isomorphisms $\varphi^{rs}$ satisfying $t^{r}|_v = t^{s}|_v + \mathrm{D}f^{\,rs}$. One chooses a triangulation such that the corners are vertices, and replaces the former boundary contribution \erf{eq:boundarycontrib} of the component $b$ by
\begin{equation*}
\mathrm{BC} \left (\left \lbrace t^{r} \right \rbrace, \left \lbrace f^{rs} \right \rbrace \right ) \df  \prod_{b^r \subset b} \mathrm{BC}(t^r) \cdot  \prod_{v = b^r\cap b^s} f_{\alpha(v)}^{\,rs}(v) \text{.}
\end{equation*}
Employing all the conditions between the cochains $c$, $t^r$ and $f^{\,rs}$ one can show that 
\begin{equation*}
\mathscr{A}_{\mathcal{G}}(\phi,\mathscr{B}_1,...,\mathscr{B}_n) \df  \mathrm{AG}(\phi^{*}c) \cdot \prod_{i=1}^{n} \mathrm{BC} \left (\left \lbrace t^{\,i,r} \right \rbrace , \left \lbrace f^{\,i,rs} \right \rbrace \right)^{-1}
\end{equation*}
is independent of all choices.

All statements of Lemma \ref{lem:dbraneholprop} stay true for boundaries with corners. We will only need the following refinement of (a). Let $\mathscr{B}$ be a boundary record for some boundary component $b$. Suppose for a smooth component $b^r$ we have a principal $A$-bundle $P$ with connection over $b^r$ together with  trivializations $\psi_0,\psi_1\maps  P \to \trivlin_0$ of $P$ over the endpoints of $b^r$.
We denote by $\mathrm{PT}(P,\psi_0,\psi_1) \in A$ the number defined by
\begin{equation}
\label{eq:defbigpt}
\ptr{\gamma}(\psi_0(\gamma(0))) \cdot \mathrm{PT}(P,\psi_0,\psi_1) = \psi_1(\gamma(1))\text{,}
\end{equation}
where $\gamma\maps  [0,1] \to b^r$ is some orientation-preserving parameterization of $b^r$, and $\ptr\gamma$ denotes the parallel transport in $P$ along $\gamma$. Let $\mathcal{T}^r$ be the trivialization over $b^r$ in the boundary record $\mathscr{B}$, and let $\mathcal{S}^r$ be some other trivializations. Suppose $\phi\maps  \mathcal{T}^r \otimes P \Rightarrow \mathcal{S}^r$ is a 2-isomorphism. In this situation, we form a new boundary record $\mathscr{B}\,'$, in which $\mathcal{T}^r$ is replaced by $\mathcal{S}^r$, the 2-isomorphism  $\varphi^{r-1,r}$ is replaced by $\phi \circ (\varphi^{r-1,r} \otimes \psi_0^{-1})$ and $\varphi^{r,r+1}$ is replaced by $(\varphi^{r,r+1} \otimes \psi_1) \circ  \phi^{-1}$. Then, the refined version of Lemma \ref{lem:dbraneholprop} (a) is
\begin{equation}
\label{eq:trivrep}
\mathscr{A}_{\mathcal{G}}(\phi,\mathscr{B}_1,...,\mathscr{B},...,\mathscr{B}_n) = \mathrm{PT}(P,\psi_0,\psi_1)^{-1} \cdot  \mathscr{A}_{\mathcal{G}}(\phi,\mathscr{B}_1,...,\mathscr{B}\,',... \mathscr{B}_n)\text{.}
\end{equation}
If the boundary component $b$ has no corners,  $\mathrm{PT}(P,\psi_0,\psi_1) = \mathrm{Hol}_P(b)$, and \erf{eq:trivrep} coincides with Lemma \ref{lem:dbraneholprop} (a).

\setsecnumdepth{2}

\section{Transgression}

\label{sec:trans}

In this section we define the transgression functor
\begin{equation*}
\trcon\maps  \diffgrbcon A M \to \fusbunconsf A{LM}\text{.}
\end{equation*}
As long as possible, we will work over a general diffeological space $X$. In order to define the functor $\trcon$, we first associate a  principal $A$-bundle $L\mathcal{G}$ over $LX$ to a  diffeological $A$-bundle gerbe $\mathcal{G}$ over $X$ with connection. Then, we  equip $L\mathcal{G}$ successively with a fusion product and a compatible, symmetrizing and superficial connection. Only for the connection we have to restrict the construction to  smooth manifolds.

\subsection{The Principal Bundle over $LX$}

\label{sec:transgression_bundle}

\def\lg#1{L#1}

Our construction is inspired from Brylinski's transgression  \cite{brylinski1}, and its adaption to bundle gerbes  \cite{waldorf5}.
Let $\mathcal{G}$ be a diffeological bundle gerbe over $X$ with connection. The diffeological principal $A$-bundle $\lg{\mathcal{G}}$ over $LX$ is defined as follows. Its fibre over a loop $\tau \in LX$ is 
\begin{equation*}
\lg{\mathcal{G}}_{\,\tau} \df  \hc 0 \hom(\tau^{*}\mathcal{G},\mathcal{I}_0)\text{;}
\end{equation*}
i.e. the set of isomorphism classes of  trivializations of the pullback $\tau^{*}\mathcal{G}$ of $\mathcal{G}$ to the circle. 
Next we define a \emph{diffeology} on the disjoint union
\begin{equation*}
\lg{\mathcal{G}} \df  \bigsqcup_{\tau\in LX} \lg{\mathcal{G}}_{\,\tau}\text{,}
\end{equation*}
making it into a diffeological space. For this purpose, we have to select a class of \emph{plots}: maps $c\maps U \to \lg{\mathcal{G}}$, where $U \subset \R^n$ is  open  and $n \in \N_0$ is arbitrary. We define a map $c\maps  U \to L\mathcal{G}$ to be a plot, if the following two conditions are satisfied: 
\begin{enumerate}
\item 
Its projection $\mathrm{pr} \circ c\maps  U \to LX$ is smooth.

\item
For all $u\in U$ there exists an open neighborhood $W \subset U$ of $u$ and a  trivialization $\mathcal{T}\maps  c_W^{*}\mathcal{G} \to \mathcal{I}_{\rho}$, where $c_W$ is the smooth map
\begin{equation*}
\alxydim{}{S^1 \times W \ar[r]^-{\id \times c} & S^1 \times \lg{\mathcal{G}} \ar[r]^-{\id \times \mathrm{pr}} & S^1 \times LX \ar[r]^-{\ev} &X\text{,}}
\end{equation*}
such that $c(w) = \iota_w^{*}\mathcal{T}$ for all $w\in W$, where the map $\iota_w\maps  S^1 \to S^1 \times W$ is defined by $\iota_w(z) \df  (z,w)$. 
\end{enumerate}
We check that this defines a diffeology on $L{\mathcal{G}}$. Firstly, if $c$ is constant, put $W \df  U$ and choose a trivialization $\mathcal{S}$ representing $c(u) \in L\mathcal{G}$. For $\mathrm{pr}_1\maps S^1 \times W \to S^1$ the projection, $\mathcal{T} \df  \mathrm{pr}_1^{*}\mathcal{S}$ is a trivialization of $c_W^{*}\mathcal{G}$ satisfying $c(w) = \mathcal{S} = \iota_w^{*}\mathcal{T}$. Secondly, if $f\maps V \to U$ is a smooth map, 
 a trivialization $\mathcal{T}$ of $c_W^{*}\mathcal{G}$ pulls back along $\id \times f\maps  S^1 \times U' \to S^1 \times  U$, and thus implies that $c \circ f$ is also a plot.  Thirdly, there is a gluing axiom for plots which is  satisfied since our definition is manifestly local. 

Our next objective is to prove that $L\mathcal{G}$ is a diffeological principal $A$-bundle  over $LX$. 
It is clear that our diffeology on $L\mathcal{G}$  makes the projection $\mathrm{pr}\maps  \lg{\mathcal{G}} \to LX$ a smooth map. Moreover, it  has to be a subduction. Subduction means that every plot $c\maps  U \to LX$ of the base lifts locally to a plot of $L\mathcal{G}$. Indeed, for $u \in U$, choose a contractible open neighborhood $W$ of $u$ and a trivialization $\mathcal{T}$ of the pullback of $\mathcal{G}$ along \begin{equation*}
\alxydim{}{S^1 \times W \ar[r]^-{\id \times c} & S^1 \times LX \ar[r]^-{\ev} & X\text{.}}
\end{equation*}
Such trivializations exist by Lemma \ref{lem:trivexist}.
Then, $\tilde c\maps  W \to L\mathcal{G}\maps  w \mapsto \iota_w^{*}\mathcal{T}$ defines a plot of $\lg{\mathcal{G}}$ such that $\mathrm{pr} \circ \tilde c = c|_W$.  Next is the definition of an action of $A$ on $\lg{\mathcal{G}}$. Principal $A$-bundles with (flat) connection over $S^1$ can be identified up to isomorphism with their holonomy evaluated around $S^1$, i.e. there is a canonical group isomorphism
\begin{equation}
\label{eq:holid}
\hc 0 \bunconflat A {S^1}\cong A\text{.}
\end{equation}
We denote by $P_a$ a flat principal $A$-bundle over $S^1$ associated to $a\in A$ by \erf{eq:holid}. The action of $A$ on $\lg{\mathcal{G}}$ is defined by $(\mathcal{T},a)\mapsto \mathcal{T} \otimes P_a$. Isomorphic choices of $P_a$ or of a representative $\mathcal{T}$ give isomorphic trivializations. 
 In order to verify that $\lg{\mathcal{G}}$ is a diffeological principal $A$-bundle over $LX$, it  remains to check that the map
\begin{equation*}
\tau\maps  \lg{\mathcal{G}}\times A \to \lg{\mathcal{G}} \times_{LX} \lg{\mathcal{G}}\maps  (\mathcal{T},a)\mapsto (\mathcal{T}, \mathcal{T} \otimes P_a)
\end{equation*}
is a diffeomorphism. This is straightforward and left as an exercise.
Summarizing, we have constructed a principal $A$-bundle $L\mathcal{G}$ over $LX$ from a diffeological bundle gerbe $\mathcal{G}$ with connection. 

We also want to transgress morphisms between bundle gerbes.  If $\mathcal{A}\maps \mathcal{G}_1 \to \mathcal{G}_2$
is a 1-isomorphism, we define 
\begin{equation*}
L\mathcal{A} \maps L\mathcal{G}_1 \to L\mathcal{G}_2 \maps \mathcal{T} \mapsto \mathcal{T} \circ \mathcal{A}^{-1}\text{.}
\end{equation*}
This is well-defined, fibre-preserving and $A$-equivariant. In order to prove that $L\mathcal{A}$ is smooth, one considers a plot $c\maps U \to L\mathcal{G}_1$ and proves in a straightforward way  that $c' \df  L\mathcal{A} \circ c$ is a plot of $L\mathcal{G}_2$. The compatibility with composition, identities and  2-isomorphisms  can literally be taken from \cite[Proposition 3.1.3]{waldorf5}.

Summarizing this section, we have constructed a  functor
\begin{equation*}
L\maps \hc  1 \diffgrbcon AX \to \diffbun A{LX}\text{.}
\end{equation*}
\cite[Proposition 3.1.4]{waldorf5} shows that $L$ is monoidal. Furthermore, it is natural with respect to smooth maps, see \cite[Eq. 3.1.4]{waldorf5}.

\subsection{The Fusion Product}

\label{sec:constrfusion}

In this section we define a  fusion product on $L\mathcal{G}$.
Consider a triple $(\gamma_1,\gamma_2,\gamma_3) \in \p X^{[3]}$, and the  loops $\tau_{ij}\df l(\gamma_i,\gamma_j)$, for $ij \in \left \lbrace 12,23,13 \right \rbrace$. Let $\mathcal{T}_{ij}\maps \tau_{ij}^{*}\mathcal{G} \to \mathcal{I}_0$ be trivializations over these three loops. Consider the two maps   $\iota_1,\iota_2\maps [0,1] \to S^1$  defined by $\iota_1(t)\df  \frac{1}{2}t$ and $\iota_2(t) \df  1-\frac{1}{2}t$.
We define a relation $\sim_{\mathcal{G}}$ with 
\begin{equation}
\label{eq:relfus}
(\mathcal{T}_{12},\mathcal{T}_{23}) \sim_{{\mathcal{G}}} \mathcal{T}_{13}
\end{equation}
if and only if there exist 2-isomorphisms
\begin{equation}
\label{eq:2isos}
\phi_1\maps \iota_1^{*}\mathcal{T}_{12} \Rightarrow \iota_1^{*}\mathcal{T}_{13}
\quomma
\phi_2\maps \iota_2^{*}\mathcal{T}_{12} \Rightarrow \iota_1^{*}\mathcal{T}_{23}
\quand
\phi_3\maps \iota_2^{*}\mathcal{T}_{23} \Rightarrow \iota_2^{*}\mathcal{T}_{13}
\end{equation}
between trivializations of bundle gerbes over the interval $[0,1]$ such that
\begin{equation}
\label{eq:condfus2}
\phi_1|_0 = \phi_{3}|_0 \bullet \phi_{2}|_0
\quand
\phi_1|_1 = \phi_{3}|_1 \bullet \phi_{2}|_1\text{.}
\end{equation}
These  are equations between restrictions (i.e. pullbacks) of the 2-isomorphisms \erf{eq:2isos} to the points $0$ and $1$ in $[0,1]$, respectively. The symbol \quot{$\bullet$} denotes the vertical composition of 2-isomorphisms in the 2-groupoid $\diffgrbcon AX$.  Below we prove that the relation $\sim_{\mathcal{G}}$ is the graph of a (unique) fusion product $\lambda_{\mathcal{G}}$ on $L\mathcal{G}$, i.e.
\begin{equation*}
\lambda_{\mathcal{G}} (\mathcal{T}_{12} \otimes \mathcal{T}_{23}) = \mathcal{T}_{13}
\qquad\text{ if and only if}\qquad
(\mathcal{T}_{12},\mathcal{T}_{23}) \sim_{\mathcal{G}} \mathcal{T}_{13}\text{.}
\end{equation*}

Before that we mention that the fusion product $\lambda_{\mathcal{G}}$ extends the transgression functor $L$ from the previous section  to a functor
\begin{equation}
\label{eq:functorfus}
L\maps \hc  1 \diffgrbcon AX \to \fusbun A{LX}\text{.}
\end{equation}
Indeed,  that a morphism $L\mathcal{A}:L\mathcal{G}_1 \to L\mathcal{G}_2$ is fusion-preserving follows by composing  the 2-isomorphisms \erf{eq:2isos} horizontally from the right with the 2-isomorphisms $\id_{\mathcal{A}^{-1}}$. Similarly, one checks that the new functor \erf{eq:functorfus} is still monoidal and natural in $M$.

Now we prove that the relation $\sim_{\mathcal{G}}$ defines a fusion product on $L\mathcal{G}$. We will be very detailed since this fusion product has -- so far as I know -- not been described anywhere else. However, in all  calculations that involve  $\lambda_{\mathcal{G}}$ we  only use  the relation $\sim_{\mathcal{G}}$; thus it is well possible to skip the following and to continue reading with Section \ref{sec:constrconn}.

\label{sec:prooffusion}

We need the following lemma about 1-isomorphisms and 2-isomorphisms between gerbes over points and intervals. We denote the trivial principal $A$-bundle with the trivial flat connection by $\trivlin_0$, and denote by $m_a$ its automorphism that simply multiplies with an element $a\in A$. 

\begin{lemma}
\label{lem:gerbehomspoint}
Let $W$ be a contractible smooth manifold. Let $\mathcal{G}_1$ and $\mathcal{G}_2$ be diffeological $A$-bundle gerbes with connection over $W$, and let $\mathcal{A}$, $\mathcal{B}:\mathcal{G}_1 \to \mathcal{G}_2$ be 1-isomorphisms. 
\begin{enumerate}[(a)]
\item 
There exists a 2-isomorphism $\varphi\maps \mathcal{A} \Rightarrow \mathcal{B}$.

\item
If $\varphi_1$ and $\varphi_2$ are  2-isomorphisms, there is a unique  $a\in A$ such that $\varphi_2 = \varphi_1 \otimes m_a$.
\end{enumerate}
\end{lemma}

\begin{proof}
We recall  that the category $\hom(\mathcal{G}_1,\mathcal{G}_2)$ of 1-isomorphisms and 2-isomorphisms between $\mathcal{G}_1$ and $\mathcal{G}_2$ is a torsor category over the category $\bunconflat AW $ of  flat principal $A$-bundles over $W$ (see Lemma \ref{lem:diffgrbcon} (iii)). The claims follow then from the fact that the functor
$\mathcal{B}A  \to \bunconflat A W$
that sends the single object of $\mathcal{B}A$ to   $\trivlin_0$ and  a morphism $a\in A$ to $m_a$, is an equivalence of categories. 
\end{proof}

Now we proceed in four steps.

\paragraph{First step.}
We show that the relation $\sim_{\mathcal{G}}$ is the graph of a map
\begin{equation*}
\lambda_{\mathcal{G}}\maps  e_{12}^{*}\lg{\mathcal{G}} \times e_{23}^{*}\lg{\mathcal{G}} \to e_{13}^{*}\lg{\mathcal{G}}\text{.}
\end{equation*}
Let  trivializations $\mathcal{T}_{12}$ and $\mathcal{T}_{23}$ be given. We have to show that: 
\begin{enumerate}
\item 
there exists a trivialization $\mathcal{T}_{13}$ such that $(\mathcal{T}_{12},\mathcal{T}_{23}) \sim_{\mathcal{G}} \mathcal{T}_{13}$.

\item
two trivializations $\mathcal{T}_{13}$ and $\mathcal{T}_{13}'$ with that property are 2-isomorphic.
\end{enumerate}
To see 1., choose some $\mathcal{T}_{13}$ and any 2-isomorphisms $\phi_i$ as in \erf{eq:2isos}. These exist because of Lemma \ref{lem:gerbehomspoint} (a).  The identities \erf{eq:condfus2} may not be satisfied. The error can by Lemma \ref{lem:gerbehomspoint} (b) be expressed by group elements $a_k\in A$, for $k=0,1$, say $\phi_1|_k \otimes m_{a_k} = \phi_3|_k \bullet \phi_2|_k$. Pick a principal $A$-bundle $P$ over $S^1$ with global holonomy $a_0^{-1}a_1$. Then, we have
\begin{equation}
\label{eq:claimexistence}
(\mathcal{T}_{12},\mathcal{T}_{23}) \sim_{\mathcal{G}} \mathcal{T}_{13} \otimes P\text{,}
\end{equation} 
which shows the existence. 

In order to see 2., we consider two trivializations $\mathcal{T}_{13}$ and $\mathcal{T}_{13}^{\,\prime}$  satisfying $(\mathcal{T}_{12},\mathcal{T}_{23}) \sim_{{\mathcal{G}}} \mathcal{T}_{13}$ and $(\mathcal{T}_{12},\mathcal{T}_{23}) \sim_{{\mathcal{G}}} \mathcal{T}_{13}^{\,\prime}$. This means we have triples $(\phi_1,\phi_2,\phi_3)$ and $(\phi_1',\phi_2',\phi_3')$ of 2-iso\-mor\-phisms both satisfying the identity \erf{eq:condfus2}. We may assume $\phi_2=\phi_2'$ after possibly changing the maps $\phi_i$.
Let $P$ be a principal $A$-bundle over $S^1$ that represents the difference between $\mathcal{T}_{13}$ and $\mathcal{T}_{13}'$ in terms of a 2-isomorphism $\varphi\maps \mathcal{T}_{13} \otimes P \Rightarrow \mathcal{T}_{13}'$. We claim that $P$ is isomorphic to the trivial bundle over $S^1$, so that $\mathcal{T}_{13} \cong \mathcal{T}_{13}'$.  Indeed, for the pullbacks of $\varphi$ along $\iota_1$ and $\iota_2$ we obtain flat trivializations $t_k\maps \iota_k^{*}P \to \trivlin_0$ by requiring the diagrams
\begin{equation*}
\alxydim{@C=1.5cm@R=1.2cm}{\iota_1^{*}\mathcal{T}_{13} \otimes \iota_1^{*}P \ar@{=>}[r]^-{\iota_1^{*}\varphi} \ar@{=>}[d]_{\phi_1^{-1} \otimes \id} &  \iota_1^{*}\mathcal{T}_{13}' \ar@{=>}[d]^{\phi_1'^{-1}} \\ \iota_1^{*}\mathcal{T}_{12} \otimes \iota_1^{*}P \ar@{=>}[r]_-{\id \otimes t_1} & \iota_1^{*}\mathcal{T}_{12}}
\quand
\alxydim{@C=1.5cm@R=1.2cm}{\iota_2^{*}\mathcal{T}_{13} \otimes \iota_2^{*}P \ar@{=>}[r]^-{\iota_2^{*}\varphi} \ar@{=>}[d]_{\phi_3^{-1} \otimes \id} &  \iota_2^{*}\mathcal{T}_{13}' \ar@{=>}[d]^{\phi_3'^{-1}} \\ \iota_2^{*}\mathcal{T}_{23} \otimes \iota_2^{*}P \ar@{=>}[r]_-{\id \otimes t_2} & \iota_2^{*}\mathcal{T}_{23}}
\end{equation*}
to commute. Restricting these diagrams in turn to the points $k=0,1$, and using the identities \erf{eq:condfus2} one obtains commutative diagrams
\begin{equation*}
\alxydim{@C=1.5cm@R=1.2cm}{\mathcal{T}_{12}|_{k} \otimes P|_{k} \ar@{=>}[r]^-{t_1|_{k}} \ar@{=>}[d]_{\phi_2|_{k} \otimes \id} &  \mathcal{T}_{12}|_{k} \ar@{=>}[d]^{\phi_2|_{k}} \\ \mathcal{T}_{23}|_{k} \otimes P|_{k} \ar@{=>}[r]_-{\id \otimes t_2|_{k}} & \mathcal{T}_{23}|_{k}}
\end{equation*}
for $k=0,1$. These show that $t_1|_{k}=t_2|_{k}$ for both $k=0,1$. This means for the holonomy of $P$ that  $\mathrm{Hol}_P(S^1)=1$, which implies $P \cong \trivlin_0$. 

\paragraph{Second step.}
We show that  $\lambda_{\mathcal{G}}$ respects the $A$-action on $L\mathcal{G}$ in the sense that for $\mathcal{T}_{12}$ and $\mathcal{T}_{23}$  trivializations, and  $a_{12},a_{23}\in A$,
\begin{equation}
\label{eq:fusionequiv}
\lambda_{\mathcal{G}}(\mathcal{T}_{12} \cdot a_{12},  \mathcal{T}_{23} \cdot a_{23}) = \lambda_{\mathcal{G}}(\mathcal{T}_{12} , \mathcal{T}_{23} ) \cdot a_{12} \cdot a_{23} \text{.}
\end{equation}
This implies simultaneously that $\lambda_{\mathcal{G}}$ descends to the tensor product $e_{12}^{*}\lg{\mathcal{G}} \otimes e_{23}^{*}\lg{\mathcal{G}}$, and that it is there $A$-equivariant. To prove \erf{eq:fusionequiv},
let $P_{12}$ and $P_{23}$ be principal $A$-bundles with connections over $S^1$ of holonomy $a_{12}$ and $a_{23}$, respectively. Further, we choose a trivialization $\mathcal{T}_{13}$ such that $(\mathcal{T}_{12},\mathcal{T}_{23}) \sim_{\mathcal{G}} \mathcal{T}_{13}$, together with 2-isomorphisms $\phi_1$, $\phi_2$ and $\phi_3$ satisfying \erf{eq:condfus2}. It is to show that
\begin{equation}
\label{eq:new2isos}
(\mathcal{T}_{12} \otimes P_{12}, \mathcal{T}_{23} \otimes P_{23}) \sim_{\mathcal{G}} \mathcal{T}_{13} \otimes P_{12} \otimes P_{23}\text{.}
\end{equation} 
Indeed, let $t\maps \iota_1^{*}P_{23} \to \trivlin_0$ and $s:\iota_2^{*}P_{12} \to \trivlin_0$ be flat trivializations. Then, the 2-isomorphisms
\begin{equation*}
\phi_1' \df  \phi_1 \otimes t^{-1}
\quomma
\phi_2' \df  \phi_2 \otimes s \otimes t^{-1}
\quand
\phi_3' \df  \phi_3  \otimes s^{-1}
\end{equation*}
satisfy \erf{eq:condfus2} and thus prove \erf{eq:new2isos}.

\paragraph{Third step.}
We show that  $\lambda_{\mathcal{G}}$ is smooth. We start with a plot $c\maps U \to e_{12}^{*}L\mathcal{G} \times e_{23}^{*}L\mathcal{G}$, and have to show that $\lambda_{\mathcal{G}} \circ c\maps U \to e_{13}^{*}L\mathcal{G}$ is again a plot. Consider the projection
\begin{equation}
\label{eq:baseproj}
\alxydim{}{U \ar[r]^-{c} & e_{12}^{*}L\mathcal{G} \times e_{23}^{*}L\mathcal{G} \ar[r] & PX^{[3]} \ar[r]^-{e_{13}} & PX^{[2]} \ar[r]^{l} & LX\text{,}}
\end{equation} 
which is a plot of $LX$. Since the projection $p\maps L\mathcal{G} \to LX$ is a subduction, every point $u\in U$ has an open neighborhood $W$ such that the restriction of \erf{eq:baseproj} to $W$ lifts to a plot $c_{13}\maps W \to e_{13}^{*}L\mathcal{G}$. Consider  the plots $c_{12}\maps W \to e_{12}^{*}L\mathcal{G}$  and $c_{23} \maps W \to e_{23}^{*}L\mathcal{G}$ obtained from $c$ by the projections to the  factors and restriction to $W$. The definition of the diffeology on $L\mathcal{G}$ allows us to assume (after a possible refinement of $W$ to a contractible open neighborhood of $u$) that the plots $c_{ij}$ of $e_{ij}^{*}L\mathcal{G}$ are defined by trivializations $\mathcal{T}_{ij}$ of the pullbacks of $\mathcal{G}$ to $W \times S^1$. Now we work like in   Step one. By Lemma \ref{lem:gerbehomspoint} there exist 2-isomorphisms \begin{equation*}
\phi_1\maps \iota_1^{*}\mathcal{T}_{12} \Rightarrow \iota_1^{*}\mathcal{T}_{13}
\quomma
\phi_2\maps \iota_2^{*}\mathcal{T}_{12} \Rightarrow \iota_1^{*}\mathcal{T}_{23}
\quand
\phi_3\maps \iota_2^{*}\mathcal{T}_{23} \Rightarrow \iota_2^{*}\mathcal{T}_{13}
\end{equation*}
of trivializations over $W \times [0,1]$, for which the identities \erf{eq:condfus2} are satisfied up to an error captured by elements $a_k\in A$. For $P$ a principal $A$-bundle over $S^1$ with holonomy $a_0^{-1}a_1$, consider the new trivialization $\mathcal{T}_{13}' \df  \mathcal{T}_{13} \otimes \mathrm{pr}_2^{*}P$, where $\mathrm{pr}_2:W \times S^1 \to S^1$ is the projection. We claim that the map $c_{13}':W \to e_{13}^{*}L\mathcal{G}:w \mapsto \iota_w^{*}\mathcal{T}_{13}$ is (a) a plot of $e_{13}^{*}L\mathcal{G}$, and (b)
equal to the composition $\lambda_{\mathcal{G}} \circ c$. 
Indeed, (a) is clear since this is exactly the definition of plots of $L\mathcal{G}$, and (b) follows from checking the identities \erf{eq:condfus2} point-wise using the argument of  step one.

\paragraph{Fourth step.}
We check that $\lambda_{\mathcal{G}}$ is associative.
 We assume in the obvious notation trivializations $\mathcal{T}_{12}$, $\mathcal{T}_{23}$ and $\mathcal{T}_{34}$. We assume further trivializations $\mathcal{T}_{13}$, $\mathcal{T}_{14}$ and $\mathcal{T}_{24}$ such that $(\mathcal{T}_{12},\mathcal{T}_{23}) \sim_{\mathcal{G}} \mathcal{T}_{13}$, $(\mathcal{T}_{13},\mathcal{T}_{34}) \sim_{\mathcal{G}} \mathcal{T}_{14}$ and $(\mathcal{T}_{23},\mathcal{T}_{34}) \sim_{\mathcal{G}} \mathcal{T}_{24}$. We have to show that 
\begin{equation}
\label{eq:fusasso}
(\mathcal{T}_{12},\mathcal{T}_{24}) \sim_{\mathcal{G}} \mathcal{T}_{14}\text{.}
\end{equation}
Employing the definition of the relation $\sim_{\mathcal{G}}$, we have for $(ijk) \in \left \lbrace (123),(134),(234) \right \rbrace$ 2-isomorphisms
\begin{equation}
\phi_1^{ijk}\maps \iota_1^{*}\mathcal{T}_{ij} \Rightarrow \iota_1^{*}\mathcal{T}_{ik}
\quomma
\phi_2^{ijk}\maps \iota_2^{*}\mathcal{T}_{ij} \Rightarrow \iota_1^{*}\mathcal{T}_{jk}
\quand
\phi_3^{ijk}\maps \iota_2^{*}\mathcal{T}_{jk} \Rightarrow \iota_2^{*}\mathcal{T}_{ik}
\end{equation}
satisfying \erf{eq:condfus2} for each $(ijk)$. Like in Step one, we may  assume that $\phi_2^{234} = \phi^{134}_2 \circ \phi_2^{123}$. Then, consider the 2-isomorphisms
\begin{equation}
\phi^{124}_1 \df  \phi_1^{134} \bullet \phi^{123}_1
\quomma
\phi^{124}_2 \df  \phi_1^{124} \bullet \phi_2^{123}
\quand
\phi^{124}_3 \df  \phi_3^{134} \bullet (\phi_3^{124})^{-1}\text{.}
\end{equation}
They satisfy \erf{eq:condfus2} and hence prove \erf{eq:fusasso}.

\subsection{The Superficial Connection}

\label{sec:constrconn}

In this section we define a connection on $L\mathcal{G}$. We restrict ourselves to a smooth manifold $M$ instead  of a general diffeological space.
We first give an interpretation of the holonomy $\mathscr{A}_{\mathcal{G}}$ for surfaces with boundary defined in Section \ref{sec:holonomy} in terms of the principal $A$-bundle $L\mathcal{G}$. Then we use $\mathscr{A}_{\mathcal{G}}$ to define a connection on $L\mathcal{G}$. The discussion  follows \cite[Section 6.2]{brylinski1}. 

Suppose $\Sigma$ is a compact oriented surface with boundary divided into components $b_1,...,b_n$, and suppose  $f_i\maps S^1 \to \Sigma$ are smooth, orientation-preserving parameterizations of $b_{i}$. 
We denote by $D^{\infty}(\Sigma,M)$ the diffeological space of smooth maps $\phi:\Sigma \to M$.
Consider the smooth maps
\begin{equation*}
c_i\maps D^{\infty}(\Sigma,M) \to LM \maps \phi \mapsto \phi \circ f_i
\end{equation*}
and the diffeological space
\begin{equation*}
P_{\Sigma} \df  c_1^{*}L\mathcal{G} \times_{D^{\infty}(\Sigma,M)} ... \times_{D^{\infty}(\Sigma,M)} c_n^{*}L\mathcal{G}
\end{equation*}
A point in $P_{\Sigma}$ is a  smooth map $\phi:\Sigma \to M$ and for each $i=1,...,n$ an isomorphism class of trivializations $\mathcal{T}_i\maps c_i(\phi)^{*}\mathcal{G} \to \mathcal{I}_0$. Applying the holonomy for surfaces with boundary point-wise, we obtain a map
\begin{equation*}
\mathscr{A}_{\mathcal{G}} \maps P_{\Sigma} \to A \text{.} 
\end{equation*}
Here we have already used that $\mathscr{A}_{\mathcal{G}}(\phi,\mathcal{T}_1,...,\mathcal{T}_n)$ only depends on the isomorphism classes of the trivializations, see Lemma \ref{lem:dbraneholprop} (a). 

\begin{lemma}
\label{lem:holsmooth}
The  map $\mathscr{A}_{\mathcal{G}}$ is smooth.
\end{lemma}

For $A=\ueins$ and with a Fréchet manifold structure on $P_{\Sigma}$, this is essentially \cite[Theorem 6.2.4]{brylinski1}. The adaption to the general group $A$ and the diffeological setting is straightforward and left out for brevity.

We use the smooth map $\mathscr{A}_{\mathcal{G}}$ to equip the principal $A$-bundle $\lg{\mathcal{G}}$ over $LM$ with a connection. Using a general theory developed in \cite{schreiber3,waldorf9} and summarized in Appendix \ref{app:functorsandforms}, one can define a connection on $L\mathcal{G}$ by specifying a smooth function  $F\maps PL\mathcal{G} \to A$ obeying  two conditions: the first condition is that $F$ is an object in a certain category $\fun {L\mathcal{G}}A$ which is isomorphic to a category of $\mathfrak{a}$-valued 1-forms on $L\mathcal{G}$. The second condition is the usual condition for connection 1-forms. The relation between a connection defined in this way and the  map $F$ is as follows: if $\gamma \in PLM$ is a path,  $\tilde\gamma \in PL\mathcal{G}$ is a lift of $\gamma$ to $L\mathcal{G}$,  and $q \df  \tilde\gamma(0)$, then the parallel transport $\ptr\gamma$ of the connection defined by $F$ satisfies
\begin{equation}
\label{eq:partransLG}
\ptr\gamma(q) = \tilde\gamma(1) \cdot F(\tilde\gamma)\text{.}
\end{equation}

In the following we define such a map $F$. Let $\tilde \gamma \in P\lg{\mathcal{G}}$ be a path with a projection $\gamma \in PL M$, and with an adjoint map $\exd\gamma\maps C_{0,1} \to M$, where $C_{t_1,t_2} \df  [t_1,t_2] \times S^1$ is the standard cylinder. Cylinders will always be oriented such that the orientation of the \emph{end}-loop $t \mapsto (t_2,t)$ coincides with the induced orientation on the boundary.
We put
\begin{equation*}
F(\tilde\gamma) \df  \mathscr{A}_{\mathcal{G}}(\gamma,\tilde\gamma(0),\tilde\gamma(1))\text{.}
\end{equation*}
In order to prove that $F$ is an object in the category $\fun {L\mathcal{G}}A$ we have to check three conditions. First we infer from Lemma \ref{lem:dbraneholprop} (b) that $F(\tilde\gamma)$ only depends on the thin homotopy class of the map $\exd\gamma\maps C_{0,1} \to M$, and thus by Remark \ref{rem:rank} only on the thin homotopy class of $\gamma \in PLM$. Secondly, since the map $PL\mathcal{G} \to P_{C_{0,1}}\maps \tilde\gamma \mapsto (\phi(\gamma),\tilde\gamma(0),\tilde\gamma(1))$ is smooth, $F$ is with Lemma \ref{lem:holsmooth} a composition of smooth maps and thus smooth. The third condition concerns the compatibility of $F$ with the composition of paths. 

 Suppose $\tilde\gamma_1$ and $\tilde\gamma_2$ are composable paths in $L\mathcal{G}$. Since we have orientation-preserving diffeomorphisms $\varphi_1:C_{0,1/2} \to C_{0,1}$ and $\varphi_2\maps C_{1/2,1}\to C_{0,1}$, the diffeomorphism invariance of $\mathscr{A}_{\mathcal{G}}$ (see Lemma \ref{lem:dbraneholprop} (d)) implies
\begin{eqnarray*}
F(\tilde\gamma_1) \,=\, \mathscr{A}_{\mathcal{G}}(\gamma_1,\tilde\gamma_1(0),\tilde\gamma_1(1)) &=& \mathscr{A}_{\mathcal{G}}(\gamma_1 \circ \varphi_1, \tilde\gamma_1(0),\tilde\gamma_1(1))
\\
F(\tilde\gamma_2) \,= \,\mathscr{A}_{\mathcal{G}}(\gamma_2,\tilde\gamma_2(0),\tilde\gamma_2(1)) &=& \mathscr{A}_{\mathcal{G}}(\gamma_2 \circ \varphi_2, \tilde\gamma_2(0),\tilde\gamma_2(1))\text{.}
\end{eqnarray*}
Since $\tilde\gamma_1(1)=\tilde\gamma_2(0)$, the gluing property (Lemma \ref{lem:dbraneholprop} (c)) implies
\begin{equation*}
F(\tilde\gamma_2 \pcomp \tilde\gamma_1) = F(\tilde\gamma_2) \cdot F(\tilde\gamma_1)\text{.}
\end{equation*}
All this proves that $F$ is an object in the groupoid $\fun {L\mathcal{G}} A$, and thus defines a Lie-algebra valued 1-form $\omega_{\mathcal{G}}\in\Omega^1_{\mathfrak{a}}(\lg{\mathcal{G}})$. 
\begin{proposition}
\label{prop:connectionprop1}
The 1-form $\omega_{\mathcal{G}}\in\Omega^1_{\mathfrak{a}}(\lg{\mathcal{G}})$ is a connection on $\lg{\mathcal{G}}$. Its holonomy around a loop $\tau\in LLM$ is inverse to the surface holonomy of $\mathcal{G}$ around the associated torus $\exd\tau\maps S^1 \times S^1 \to M$.
Its curvature is given by 
\begin{equation*}
\mathrm{curv}(\omega_{\mathcal{G}}) = - \int_{S^1} \ev^{*}\mathrm{curv}(\mathcal{G}) \in \Omega^2_{\mathfrak{a}}(\lg{\mathcal{G}})  \text{.}
\end{equation*}
\end{proposition}

\begin{proof}
In terms of the groupoid $\fun AG$, the condition that the 1-form $\omega_{\mathcal{G}}$ defined by the object $F$ is a connection is
\begin{equation}
\label{eq:connequiv}
g(1) \cdot F(\tilde \gamma g) = F(\tilde \gamma) \cdot g(0)
\end{equation}
for all $g \in PA$ and $\tilde\gamma \in P\lg{\mathcal{G}}$ (see Appendix \ref{app:functorsandforms}). Let  $P_0$ and $P_1$ be principal $A$-bundles over $S^1$ with holonomies $g(0)$ and $g(1)$, respectively. Then, we compute with Lemma \ref{lem:dbraneholprop} (a)
\begin{multline*}
F(\tilde\gamma g) = \mathscr{A}_{\mathcal{G}}(\gamma,(\tilde\gamma g)(0),(\tilde\gamma g)(1)) = \mathscr{A}_{\mathcal{G}}(\gamma,\tilde\gamma(0)  \otimes P_0,\gamma (1) \otimes P_1)\\= \mathscr{A}_{\mathcal{G}}(\gamma,\tilde\gamma(0),\gamma (1)) \cdot g(1)^{-1} \cdot g(0)\text{.}
\end{multline*}
This shows \erf{eq:connequiv}. Comparing the parallel transport prescription \erf{eq:partransLG} with the definition of holonomy in diffeological bundles \cite[Definition 3.2.11]{waldorf9}, we see that
\begin{equation*}
\mathrm{Hol}_{L\mathcal{G}}(\tau) = \mathrm{Hol}_{\mathcal{G}}(\exd\tau)^{-1}\text{,}
\end{equation*}
for loops $\tau\in LLX$, as claimed. The formula for the curvature can be deduced from the fact that the curvature is  determined by the holonomy.   
\end{proof}

We infer from Propositions \ref{prop:thinhol}, \ref{prop:connectionprop1} and Lemma \ref{lem:dbraneholprop} (b):

\begin{corollary}
The connection $\omega_{\mathcal{G}}$ on $L\mathcal{G}$ is superficial.
\end{corollary}

In remains to prove that the two  conditions relating connections with fusion products   are satisfied. The first condition is:

\begin{proposition}
The connection $\omega_{\mathcal{G}}$ on $L\mathcal{G}$ is compatible with the fusion product $\lambda_{\mathcal{G}}$. 
\end{proposition}

\begin{proof}
It suffices to prove that the fusion product commutes with parallel transport along paths $\Gamma \in P(PM^{[3]})$, i.e.
\begin{equation*}
\ptr{\gamma_{13}}(\lambda_{\mathcal{G}}(\mathcal{T}_{12},\mathcal{T}_{23})) = \lambda_{\mathcal{G}}(\ptr{\gamma_{12}}(\mathcal{T}_{12}),\ptr{\gamma_{23}}(\mathcal{T}_{23}))
\end{equation*}
where $\gamma_{ij} \df  P(l \circ e_{ij})(\Gamma) \in PLM$. We recall from the definition of the connection $\omega_{\mathcal{G}}$ that $\mathcal{T}_{ij}' = \ptr{\gamma_{ij}}(\mathcal{T}_{ij})$ means that 
\begin{equation}
\label{eq:parallelcond}
\mathscr{A}_{ij} \df  \mathscr{A}_{C}(\gamma_{ij}, \mathcal{T}_{ij},\mathcal{T}_{ij}')=1\text{.}
\end{equation}
We pick  trivializations $\mathcal{T}_{12}'$ and $\mathcal{T}_{23}'$ satisfying \erf{eq:parallelcond}, i.e. $\mathscr{A}_{12}=\mathscr{A}_{23}=1$, and choose $\mathcal{T}_{13}'$ such that $\lambda_{\mathcal{G}}(\mathcal{T}_{12}',\mathcal{T}_{23}')=\mathcal{T}_{13}'$. Then it remains to show that $\mathscr{A}_{13}=1$.

We use the notation $C=[0,1] \times S^1$ and $Q\df  [0,1] \times [0,1]$.
Consider the maps $\iota_k\maps Q \to C$ defined by $\iota_1(t,s) \df  (t,\mathrm{e}^{\im\pi s})$ and $\iota_2 (t,s)\df  (t,\mathrm{e}^{-\im\pi s})$ that reproduce for  fixed $t$ the maps $\iota_0$, $\iota_1$ from Section \ref{sec:constrfusion}. We  choose triangulations $T_{12}$, $T_{23}$ and $T_{13}$ of $C$ in such a way that $\iota_1^{*}T_{12} = \iota_{1}T_{13}$, $\iota_2^{*}T_{12} = \iota_{1}^{*}T_{23}$ and $\iota_2^{*}T_{23} = \iota_{2}^{*}T_{13}$ and such that $T_{ij}$ is subordinated to the open cover $\gamma_{ij}^{-1}\mathscr{U}$, for $\mathscr{U}$  an open cover of $M$ for which we have a Deligne cocycle $c$ for the bundle gerbe $\mathcal{G}$. Such triangulations can easily be constructed from three appropriate triangulations of $Q$. We  further choose Deligne 1-cochains $t_{12}$, $t_{12}'$ and $t_{23}$, $t_{23}'$ and $t_{13}$, $t_{13}'$ representing the respective trivializations. We claim that
\begin{equation}
\label{eq:multiplicativity}
\mathscr{A}_{13} = \mathscr{A}_{12} \cdot \mathscr{A}_{23} \text{,}
\end{equation}
which in turn proves that $\mathscr{A}_{13}=1$. We recall that each $\mathscr{A}_{ij}$ consists of three factors, namely  the Alvarez-Gaw\c edzki-Formula $\mathrm{AG}(\gamma_{ij}^{*}c)$ and the two boundary contributions $\mathrm{BC}(t_{ij})$ and $\mathrm{BC}(t_{ij}')^{-1}$. The multiplicativity \erf{eq:multiplicativity} holds in fact separately for all three terms. Indeed $\mathrm{AG}(\gamma_{13}^{*}c)= \mathrm{AG}(\gamma_{12}^{*}c) \cdot \mathrm{AG}(\gamma_{23}^{*}c)$ is a simple consequence of our choices of the triangulations. For the boundary parts, it suffices to prove the identity
\begin{equation}
\label{eq:multibc}
\mathrm{BC}(t_{13}) = \mathrm{BC}(t_{12}) \cdot \mathrm{BC}(t_{23})\text{,}
\end{equation}
which is true whenever Deligne 1-cochains $t_{12}$, $t_{23}$ and $t_{13}$ represent trivializations satisfying $(\mathcal{T}_{12},\mathcal{T}_{13}) \sim_{\mathcal{G}} \mathcal{T}_{13}$. Indeed, the 2-isomorphisms $\phi_1$, $\phi_2$ and $\phi_3$ in the latter relation determine Deligne 2-cochains $h_1$, $h_2$ and $h_3$ such that 
\begin{equation*}
\iota_1^{*}\beta_{13} = \iota_1^{*}\beta_{12} + \mathrm{D}(h_1)
\quomma
\iota_1^{*}\beta_{23} = \iota_2^{*}\beta_{12} + \mathrm{D}(h_2)
\quand
\iota_2^{*}\beta_{13} = \iota_{2}^{*}\beta_{23} + \mathrm{D}(h_3)
\end{equation*}
and the identities \erf{eq:condfus2} imply that $h_1(k) = h_3(k) \cdot h_2(k)$ for $k=0,1$. Feeding this into the definition of the boundary terms $\mathrm{BC}(t_{ij})$ shows \erf{eq:multibc}. 
\end{proof}

The second condition is:

\begin{proposition}
\label{prop:symm}
The connection $\omega_{\mathcal{G}}$  on $\lg{\mathcal{G}}$ symmetrizes the fusion product $\lambda_{\mathcal{G}}$. 
\end{proposition}

For the proof we need the following  formula for the parallel transport in $L\mathcal{G}$ along a certain class of paths obtained from orientation-preserving diffeomorphisms $\varphi:S^1 \to S^1$. Notice that each such diffeomorphism is smoothly homotopic to the identity on $S^1$. 
\begin{lemma}
\label{lem:partranspullback}
Let $\varphi:S^1 \to S^1$ be an orientation-preserving diffeomorphism, and let $h\maps [0,1] \to LS^1$ be a smooth map with $h(0)=\id_{S^1}$ and $h(1) = \varphi$. Let $\beta\in LM$ be a loop and let $\mathcal{T} \in L\mathcal{G}$ be a trivialization over $\beta$. Then,
$\ptr{L\beta \circ h}(\mathcal{T}) = \varphi^{*}\mathcal{T}$.
\end{lemma}

\begin{proof}
The path $\gamma \df  L\beta \circ h \in PLM$ lifts to $L\mathcal{G}$ by putting $\tilde\gamma(t) \df  h(t)^{*}\mathcal{T}$. 
Employing \erf{eq:partransLG} and the definition of $F$, we have to show that
\begin{equation}
\label{eq:equivholcomp}
\mathscr{A}_{\mathcal{G}}(\exd{(L\beta \circ h)},\mathcal{T},\varphi^{*}\mathcal{T})=1\text{.}
\end{equation}
Notice that $\exd{(L\beta \circ h)} = \beta \circ \exd h$, so that we  have a trivialization $\mathcal{S} \df  (\exd h)^{*}\mathcal{T}$ of $\mathcal{G}$ over the whole surface $C_{0,1}$. Moreover, $\mathcal{S}|_{\left \lbrace 0\right \rbrace \times S^1} = \mathcal{T}$ and $\mathcal{S}|_{\left \lbrace 1\right \rbrace \times S^1} = \varphi^{*}\mathcal{T}$. It follows that all factors in the holonomy formula \erf{eq:holtriv} vanish. This proves \erf{eq:equivholcomp}. \end{proof}

\begin{proof}[Proof of Proposition \ref{prop:symm}]
With Lemma \ref{lem:partranspullback} applied to $\varphi = r_\pi$, the rotation by an angle of $\pi$, the  condition of Definition \ref{def:symmetrizing} is
\begin{equation}
\label{eq:connfuscomp2}
r_{\pi}^{*}\lambda_{\mathcal{G}}(\mathcal{T}_{12},\mathcal{T}_{23}) = \lambda_{\mathcal{G}}(r_{\pi}^{*}\mathcal{T}_{23},r_{\pi}^{*}\mathcal{T}_{12})\text{.}
\end{equation}
In order to show that this is true, let us choose $\mathcal{T}_{13}$ such that $(\mathcal{T}_{12},\mathcal{T}_{23}) \sim_{\mathcal{G}} \mathcal{T}_{13}$, together with 2-isomorphisms $\phi_1$, $\phi_2$ and $\phi_3$ satisfying the identities \erf{eq:condfus2}. We claim that $(r_{\pi}^{*}\mathcal{T}_{23},r_{\pi}^{*}\mathcal{T}_{12}) \sim_{\mathcal{G}} r_{\pi}^{*}\mathcal{T}_{13}$; this proves \erf{eq:connfuscomp2}. We use the commutative  diagram
\begin{equation*}
\alxydim{@=1.2cm}{[0,1] \ar[r]^-{\iota_1} \ar[d]_{a} & S^1 \ar[d]^{r_{\pi}} \\ [0,1] \ar[r]_-{\iota_2} & S^1\text{,}}
\end{equation*}
in which $a(t) \df  1-t$. It allows us to define new 2-isomorphisms
\begin{multline*}
\phi_1' \df  a^{*}\phi_3 \maps \iota_1^{*}r_{\pi}^{*}\mathcal{T}_{23} \Rightarrow \iota_1^{*}r_{\pi}^{*}\mathcal{T}_{13}
\quomma
\phi_2' \df  a^{*}\phi_2^{-1}\maps \iota_2^{*}r_{\pi}^{*}\mathcal{T}_{23} \Rightarrow \iota_1^{*}r_{\pi}^{*}\mathcal{T}_{12}
\\\quand
\phi_3' \df  a^{*}\phi_1\maps \iota_2^{*}r_{\pi}^{*}\mathcal{T}_{12} \Rightarrow \iota_2^{*}r_{\pi}^{*}\mathcal{T}_{13}\text{.}
\end{multline*}
These satisfy  the identities \erf{eq:condfus2}, and so prove the claim. 
\end{proof}

Summarizing,  the fusion bundle $(L\mathcal{G},\lambda_{\mathcal{G}})$ comes with a  compatible, symmetrizing and superficial connection. We leave it as an exercise to check that  $\omega_{\mathcal{G}}$ is compatible with the tensor product, that it is natural in $M$, and that the isomorphisms $L\mathcal{A}:L\mathcal{G}_1 \to L\mathcal{G}_2$ are connection-preserving. This completes the definition of the transgression functor
\begin{equation*}
\trcon\maps \hc 1 \diffgrbcon AM \to \fusbunconsf A{LM}\text{.}
\end{equation*}

\begin{remark}
\label{re:equiv}
Let us recover the equivariant structure on $\lg{\mathcal{G}}$ of Brylinski and McLaughlin \cite{brylinski1,brylinski4}. Their observation is the following: for $\diff^{+}(S^1)$  the group of orientation-preserving diffeomorphisms of $S^1$,
\begin{equation*}
\widetilde E:\diff^{+}(S^1) \times L\mathcal{G} \to L\mathcal{G}\maps (\varphi,\mathcal{T}) \mapsto \varphi^{*}\mathcal{T}
\end{equation*}
defines an equivariant structure on $L\mathcal{G}$. 
By inspection and Lemma \ref{lem:partranspullback} one checks that the equivariant structure $\widetilde E$ coincides with the equivariant structure $E$ obtained from the superficial connection $\omega_{\mathcal{G}}$ via Proposition \ref{prop:diffequiv}. \end{remark}

\setsecnumdepth{2}

\section{\Untrans}

\label{sec:untrans}

In this section we define the regression functor
\begin{equation*}
\uncon_x\maps \fusbunconsf A{LX} \to \diffgrbcon A X
\end{equation*}
for a general connected diffeological space $X$ with  base point $x\in X$. We first construct an underlying functor in a setup \emph{without} connections, and then promote it to the setup \emph{with} connection. In Section \ref{sec:holreg} we derive formulas for the surface holonomy of a gerbe in the image of $\un_x$.

\subsection{Reconstruction of Bundle Gerbes}

Suppose $\pi\maps Y \to X$ is a subduction. Let $\prodbun AY$ denote the groupoid of principal $A$-bundles over $Y^{[2]}$ with products in the sense of  
Definition \ref{def:product}. A  functor
\begin{equation*}
\mathcal{F}\maps \prodbun A {Y} \to \hc 1 \diffgrb A {X}
\end{equation*}
is defined as follows. To a\ principal $A$-bundle $P$ over $Y^{[2]}$ with  product $\lambda$ it assigns the obvious diffeological bundle gerbe $\mathcal{F}(P)$: its subduction is $\pi:Y \to X$, its bundle is $P$ and its product is $\lambda$.
To a product-preserving isomorphism $\varphi\maps P_1 \to P_2$ it assigns a  1-isomorphism
\begin{equation*}
\mathcal{F}(\varphi)\maps \mathcal{F}(P_1) \to  \mathcal{F}(P_2)\text{.}
\end{equation*}
As an isomorphism between bundle gerbes \cite[Definition 2]{waldorf1}, $\mathcal{F}(\varphi)$ consists of a subduction $\zeta\maps Z \to Y^{[2]}$, a principal $A$-bundle $Q$ over $Z$, and a bundle isomorphism 
\begin{equation*}
\alpha\maps \pi_{13}^{*}P_1 \otimes \zeta_2^{*}Q \to \zeta_1^{*}Q \otimes \pi_{24}^{*}P_2
\end{equation*}
over $Z^{[2]} = Z \times_X Z$, where $\zeta_i\maps Z^{[2]} \to Z$ denote the two projections. Moreover, $\alpha$ is required to be compatible with the products $\lambda_1$ and $\lambda_2$. To define $\mathcal{F}(\varphi)$, we put $Z \df  Y^{[2]}$, $\zeta\df  \id$, and $Q \df  P_2$. The isomorphism $\alpha$ is
\begin{equation*}
\alxydim{@=1.3cm}{\pi_{13}^{*}P_1 \otimes \pi_{34}^{*}P_2 \ar[r]^-{\pi_{13}^{*}\varphi \otimes \id} & \pi_{13}^{*}P_2 \otimes \pi_{34}^{*}P_2 \ar[r]^-{\pi_{134}^{*}\lambda_2} & \pi_{14}^{*}P_2 \ar[r]^-{\pi_{124}^{*}\lambda_2^{-1}} & \pi_{12}^{*}P_2 \otimes \pi_{24}^{*}P_2\text{.}}
\end{equation*}
The compatibility condition with $\lambda_1$ and $\lambda_2$ follows from the fact that $\varphi$ is product-preserving. For two composable, product-preserving isomorphisms $\varphi:P_1 \to P_2$ and $\psi\maps P_2 \to P_3$, a \quot{compositor}  2-isomorphism
$\mathcal{F}(\psi) \circ \mathcal{F}(\varphi) \Rightarrow \mathcal{F}(\psi\circ \varphi)$
can be constructed from the product of $P_3$. This shows that $\mathcal{F}$ is a functor.

Now we consider the  subduction $\ev_1\maps \px Xx \to X\maps \gamma \mapsto \gamma(1)$. The functor
\begin{equation}
\label{eq:regwithoutcon}
\alxydim{}{\fusbun A {L X} \ar[r]^-{l^{*}} & \prodbun A {\px Xx} \ar[r]^-{\mathcal{F}}   & \hc 1 \diffgrb A X}
\end{equation}
defines the regression functor $\un_x$ in a setup without connections.

\subsection{Reconstruction of Connections}

\label{sec:reconcon}

Let $P$ be a fusion bundle with a compatible and superficial connection (it does not have to be symmetrizing for the construction of the regression functor). Compatibility implies that the diffeological $A$-bundle gerbe $\un_x(P) =\mathcal{F}(l^{*}P)$ is already equipped with one part of a connection, namely with a connection on its bundle $l^{*}P$. In order to complete it to a gerbe connection, 
it remains to construct a curving: a 2-form $B_P \nobr\in\nobr \Omega^2_{\mathfrak{a}}(\px Xx)$ such that
\begin{equation}
\label{eq:curvcond}
\mathrm{pr}_2^{*}B_P - \mathrm{pr}_1^{*}B_P = \mathrm{curv}(l^{*}P)\text{.}
\end{equation}

In the following we construct this 2-form using Theorem \ref{th:functorsvsforms}, which is the 2-dimensional analog of the method we have used in Section \ref{sec:constrconn} to define a connection 1-form on $L\mathcal{G}$. 
As explained in more detail in Appendix \ref{app:functorsandforms}, a \emph{bigon} $\Sigma$ in a diffeological space $W$ is a path $\Sigma\maps [0,1] \to PW$ in the space of  paths (with fixed endpoints) in $W$. The space of bigons in $W$ is denoted by $BW$. In order to use Theorem \ref{th:functorsvsforms} we first construct a smooth map 
\begin{equation*}  
G_P\maps B(\px Xx) \to A\text{.}
\end{equation*}
From a given a bigon $\Sigma \in B(\px Xx)$ (see Figure \ref{fig:bigon}) and $t\in [0,1]$ we extract three paths:
\begin{equation}
\label{eq:bigonpaths}
\Sigma^o(t) \df  \Sigma(0)(t)
\quomma
\Sigma^u(t) \df  \Sigma(1)(t)
\quand
\Sigma^m(t) \df  P\ev_t(B\ev_1(\Sigma)) 
\end{equation}
where $\ev_1\maps \px Xx \to X$ and $\ev_t\maps PX \to X$ are evaluation maps, and $P\ev_t$ and $B\ev_1$ are the maps induced on spaces of paths and spaces of bigons, respectively. More explicitly, the third path is given by $\Sigma^m(t)(\sigma) = \Sigma(\sigma)(t)(1)$. \begin{figure}
\begin{center}
\includegraphics{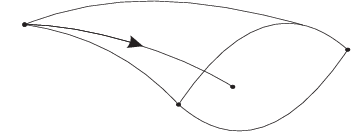}\setlength{\unitlength}{1pt}\begin{picture}(0,0)(188,700)\put(19.15625,745.30063){$x$}\put(142.27625,724.15786){$\Sigma$}\end{picture}
\qquad\quad
\includegraphics{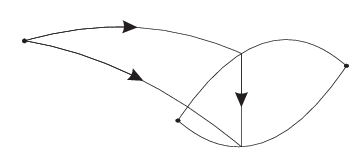}\setlength{\unitlength}{1pt}\begin{picture}(0,0)(187,596)\put(19.15625,640.71043){$x$}\put(73.51909,659.86526){$\Sigma^o(t)$}\put(60.96983,618.04205){$\Sigma^u(t)$}\put(139.24715,616.30951){$\Sigma^m(t)$}\end{picture}
\end{center}
\caption{The picture on the left shows a bigon $\Sigma$ in $\px Xx$: it can be regarded as a bigon in $X$, which has for each of its points  a chosen path connecting $x$ with that point. The picture on the right shows the three paths associated to a bigon $\Sigma$ and $t \in [0,1]$.}
\label{fig:bigon}
\end{figure}
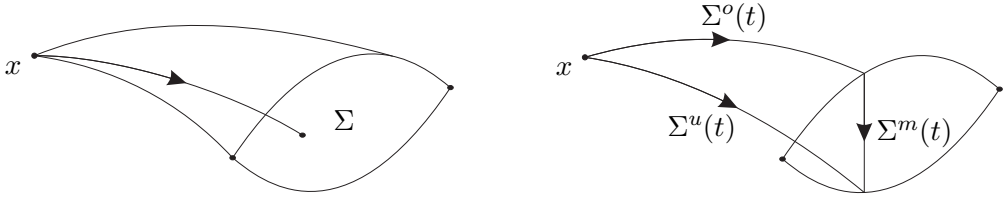
The endpoints of these three paths are:
\begin{multline*}
\ev(\Sigma^o(t)) = (x,\Sigma(0)(t)(1))
\quomma
\ev(\Sigma^u(t)) = (x, \Sigma(1)(t)(1)) 
\\\quand
\ev(\Sigma^m(t)) = (\Sigma(0)(t)(1), \Sigma(1)(t)(1))\text{.}
\end{multline*}
This allows us to define a map
\begin{equation*}
\gamma_{\Sigma}\maps [0,1] \to L X\maps t \mapsto l(\Sigma^m(t) \pcomp \Sigma^o(t),\id \pcomp \Sigma^u(t))\text{,}
\end{equation*}
which is smooth as the composition of smooth maps, and has sitting instants because the bigon $\Sigma$ has (by definition) appropriate sitting instants. Thus, $\gamma_{\Sigma} \in PLX$. 
For $\beta_0 \df \Sigma(0)(0) \in \px Xx$ and $\beta_1 \df  \Sigma(1)(1)\in \px Xx$, the path $\gamma_{\Sigma}$ starts at $\tau_0 \df  l(\id \pcomp \beta_0,\id \pcomp \beta_0)$, and ends at $\tau_1 \df  l(\id \pcomp \beta_1, \id \pcomp \beta_1)$. We recall from Lemma \ref{lem:fusbuntrivflat} that over the points $\tau_0,\tau_1\in LX$ the bundle $P$ has distinguished points $\can_0 \in P_{\tau_0}$ and $\can_1 \in P_{\tau_1}$. Now we define $G_P(\Sigma) \in A$  by
\begin{equation*}
\ptr{\gamma_{\Sigma}}(\can_0) = \can_1 \cdot G_P(\Sigma) \text{,}
\end{equation*}
where $\ptr{\gamma_{\Sigma}}$ denotes the parallel transport in $P$ along the path $\gamma_{\Sigma}$. 

Before we proceed we shall give an alternative formulation of $G_{P}(\Sigma)$. We consider the map $[0,1] \to B\px Mx:\sigma \mapsto \Sigma_{\sigma}$ defined by $\Sigma_\sigma(s) \df  \Sigma(\sigma \phi(s))$, where $\phi$ is a smoothing function assuring that each $\Sigma_{\sigma}$ has appropriate sitting instants. We obtain a smooth map $h \maps [0,1]^2 \to LX$ defined by $h(\sigma,t) \df  \gamma_{\Sigma_\sigma}(t)$. 

\begin{lemma}
\label{lem:GPref}
We have
\begin{equation*}
\displaystyle G_P(\Sigma) =  \exp \left ( - \int_{[0,1]^2} h^{*}\mathrm{curv}(P) \right )\text{.}
\end{equation*}
In particular, $G_P(\Sigma)$  depends only on the curvature of $P$.  
\end{lemma}

\begin{proof}
The map $h$ is a homotopy with fixed endpoints between the path $o_{\Sigma} \in PLX$ defined by $o_{\Sigma}(t) \df  l(\id \pcomp \Sigma^o(t),\id \pcomp \Sigma^o(t))$ and $\gamma_{\Sigma}$. Stokes' Theorem implies
\begin{equation*}
\ptr{o_{\Sigma}}(\can_0) = \ptr{\gamma_{\Sigma}}(\can_0)  \cdot \exp \left ( \int_{[0,1]^2} h^{*}\mathrm{curv}(P) \right )\text{.}
\end{equation*}
But the path $o_{\Sigma}$ factors through $PX \to PX^{[2]} \to LX$, i.e. through the domain of the flat section $\can:PX \to P$. Thus, $\ptr{o_{\Sigma}}(\can_0)= \can_1$. This shows the claim. 
\end{proof}

Now we come to the main point of the construction of the map $G_P$.

\begin{lemma}
\label{lem:reconcurving}
The map $G_P\maps B\px Xx \to A$ satisfies the conditions of Theorem \ref{th:functorsvsforms}. \end{lemma}

\begin{proof}
We  check the following four conditions:

1.)  $G_P \maps  B\px Xx \to A$ is smooth. This is straightforward and left out for brevity. 

2.) $G_P$ is constant on thin homotopy classes of bigons. Suppose  that $h\in PB\px Xx$ is a thin homotopy between bigons $\Sigma$ and $\Sigma'$. By definition, this is a path in $B\px Xx$ with $h(0)=\Sigma$ and $h(1)=\Sigma'$, such that the adjoint map $\exd h\maps [0,1]^3 \to \px Xx$ has rank two and the maps $h^{o},h^{u}:[0,1]^2 \to \px Xx$ defined by $h^{o}(s,t) \df  \exd h(s,0,t)$ and $h^{u}(s,t) \df  \exd h(s,1,t)$ have rank one. 
 Further, it keeps the \quot{endpoints} fixed, i.e. $h(t)(0,0) = \beta_0$ and $h(t)(1,1)=\beta_1$ for all $t\in [0,1]$, in the notation of Section \ref{sec:reconcon}. Notice that $P\gamma(h) \in PPLX$ is a homotopy between the paths $\gamma_{\Sigma}$ and $\gamma_{\Sigma'}$ with fixed \quot{end points} $\tau_0$ and $\tau_1$.  One checks that its adjoint has rank two; thus, since our connection is superficial, we have $\ptr{\gamma_{\Sigma}}= \ptr{\gamma_{\Sigma'}}$ by Lemma \ref{lem:pointwisethin} and hence $G_P(\Sigma) = G_P(\Sigma')$. 

3.) $G_P$ respects the vertical composition of bigons. We consider  two vertically composable bigons $\Sigma_1,\Sigma_2 \in B\px Xx$, i.e.  $\Sigma_1^u(t) = \Sigma^o_2(t)$. We define the path
\begin{equation*}
\beta\maps [0,1] \to \p X^{[3]}\maps t \mapsto (\Sigma^m_1(t) \pcomp \Sigma_1^o(t),\id \pcomp   \Sigma_1^u(t), \prev{\Sigma_2^m(t)} \pcomp \Sigma_2^u(t))\text{,}
\end{equation*}
and find the relations
\begin{equation*}
Pl(Pe_{12}(\beta)) \sim \gamma_{\Sigma_1}
\quomma
Pl(Pe_{23}(\beta)) \sim \gamma_{\Sigma_2}
\quand
Pl(Pe_{13}(\beta)) \sim \gamma_{\Sigma_2 \bullet \Sigma_1}\text{,}
\end{equation*}
where \quot{$\sim$} indicates a rank-two-homotopy between the adjoint maps $[0,1] \times S^1 \to X$. 
The rank-two homotopies ensure that the superficial connection on $P$ has equal parallel transport maps along the respective paths.  Using that the connection on $P$ is compatible with the fusion product $\lambda$, we obtain a commutative diagram
\begin{equation*}
\alxydim{@=1.2cm}{P_{\tau_0} \otimes P_{\tau_0} \ar[r]^-{\lambda} \ar[d]_{\ptr{\gamma_{\Sigma_1}} \otimes \ptr{\gamma_{\Sigma_2}}} & P_{\tau_0} \ar[d]^{\ptr{\gamma_{\Sigma_2 \bullet \Sigma_1}}} \\ P_{\tau_1} \otimes P_{\tau_1} \ar[r]_-{\lambda} & P_{\tau_1}\text{.}}
\end{equation*}
Since we have $\lambda(\can_k \otimes \can_k) = \can_k$ for $k=0,1$, this implies 
\begin{multline*}
\can_1 \cdot G_{P}(\Sigma_2 \bullet \Sigma_1) = \ptr{\gamma_{\Sigma_2 \bullet \Sigma_1}}(\can_0)=\ptr{\gamma_{\Sigma_1}} (\can_0 ) \otimes \ptr{\gamma_{\Sigma_2}}(\can_0) \\= \lambda (\can_1 \cdot G_{P}(\Sigma) \otimes \can_1 \cdot  G_{P}(\Sigma')) = \can_1 \cdot G_{P}(\Sigma) \cdot G_{P}(\Sigma') \text{,}
\end{multline*}
showing that $G_P$ respects the vertical composition.

4.) $G_P$ respects the horizontal composition of bigons. 
If $\Sigma_1,\Sigma_2\in B\px Xx$ are horizontally composable bigons, we have the simple relation
$\gamma_{\Sigma_2 \pcomp \Sigma_1} = \gamma_{\Sigma_2} \pcomp \gamma_{\Sigma_1}$,
which immediately implies the required condition:
$G_{P}(\Sigma_2 \pcomp \Sigma_1) = G_{P}(\Sigma_1) \cdot G_{P}(\Sigma_2)$.
\end{proof}

Thus, by Theorem \ref{th:functorsvsforms}, there exists a unique 2-form $B_P \in \Omega^2_{\mathfrak{a}}(\px Xx)$ such that
\begin{equation*}
G_P(\Sigma) = \exp \left ( - \int_{[0,1]^2}\Sigma^{*}B_P \right )
\end{equation*}
for all $\Sigma \in B\px Xx$. That this 2-form $B_P$ is a curving for the bundle gerbe $\mathcal{F}(l^{*}P)$ is the content of the next lemma:

\begin{lemma}
\label{lem:reconcurving2}
The 2-form $B_P$ satisfies the identity \erf{eq:curvcond}.
\end{lemma}

\begin{proof}
Consider $Z \df  l^{*}P = \px Xx^{[2]} \lli{l}\times_p P$ with its two projections $p:Z \to \ptx Xx^{[2]}$ and $\tilde l\maps Z \to P$. The identity \erf{eq:curvcond} we have to prove is an equation over $\ptx Xx^{[2]}$, but we can equivalently check its pullback  along the subduction $p$, which is
\begin{equation}
\label{eq:curvZ}
p^{*}\mathrm{pr}_2^{*}B_P - p^{*}\mathrm{pr}_1^{*}B_P = \tilde l^{*}\mathrm{d}\omega_P
\text{,}
\end{equation}
where $\omega_P$ is the connection 1-form on $P$. 

Let $F\maps PP \to A$ be the smooth map corresponding to $\omega_P$ (see Appendix \ref{app:functorsandforms}). Looking at \erf{eq:morfun} we see that \erf{eq:curvZ} is true if and only if $\tilde l^{*}F$ is a 1-morphism in $\tfun ZA$ going from $p^{*}\mathrm{pr}_2^{*}G_P$ to $p^{*}\mathrm{pr}^{*}G_P$. This means, in turn, that Equation \erf{eq:pseudotrans} must be satisfied:
\begin{equation}
\label{eq:axiom2fun}
F(\Sigma_P(1)) \cdot G(\Sigma_1)  = G(\Sigma_{2}) \cdot F(\Sigma_P(0))\text{,}
\end{equation}
where we have written $\Sigma_P \df  B\tilde l(\Sigma)\in B P$ and $\Sigma_i \df  B\mathrm{pr}_i(Bp(\Sigma))\in B\px Xx$ for $i=1,2$. We consider two  paths $\tau^u,\tau^o$ in $L X$ (see Figure \ref{fig:bigonindoublepathspace}).
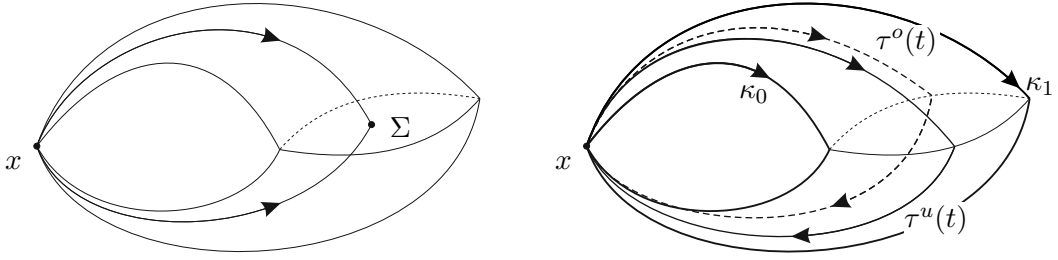
\begin{figure}
\begin{center}
\hspace{-2cm}
\includegraphics{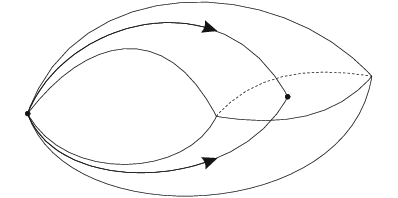}\setlength{\unitlength}{1pt}\begin{picture}(0,0)(268,366)\put(77.72655,398.03046){$x$}\put(221.78806,410.17890){$\Sigma$}\end{picture}
\quad
\includegraphics{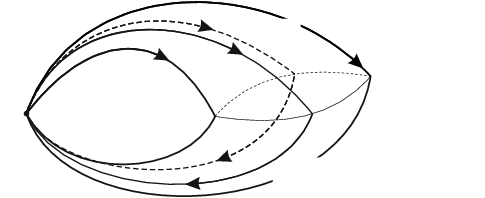}\setlength{\unitlength}{1pt}\begin{picture}(0,0)(73,366)\put(-159.87345,398.03046){$x$}\put(-90.97512,424.59392){$\kappa_0$}\put(-39.94866,443.25978){$\tau^o(t)$}\put(15.99591,427.67348){$\kappa_1$}\put(-29.35786,376.04069){$\tau^u(t)$}\end{picture}
\hspace{-4cm}\strut
\end{center}
\caption{On the left is a bigon in $\px Xx^{[2]}$: it can be seen as a bigon in $X$ which has for each of its points two   paths connecting it with the base point $x$. On the right it is shown how such pairs of paths can be combined to a loop.}
\label{fig:bigonindoublepathspace}
\end{figure}
They are given by
\begin{equation*}
\tau^o(t) \df  l(Bp(\Sigma)(0)(t))
\quand
\tau^u(t) \df  l(Bp(\Sigma)(1)(t))\text{.}
\end{equation*}
They start both at a common loop $\kappa_0$ and end at a common loop $\kappa_1$.
Below we prove \erf{eq:axiom2fun} by showing in two steps that
\begin{equation}
\label{eq:curvcond2}
F(\Sigma_P(0)) \cdot F(\Sigma_P(1))^{-1} = \mathrm{Hol}_P(\prev{\tau^u} \pcomp \tau^o)^{-1} =G(\Sigma_1) \cdot G(\Sigma_2)^{-1} \text{.}
\end{equation}

For the  equality on the left hand side of  \erf{eq:curvcond2} we use the following general fact: if $\gamma\in LY$ is a loop in some diffeological space $Y$,  $P$ is a principal $A$-bundle with connection over $Y$, and $\gamma$ lifts to a \emph{loop} $\tilde\gamma\in LP$, then $\mathrm{Hol}_P(\gamma) = F(\tilde\gamma)^{-1}$, where $F\maps PP \to A$ is the smooth map that corresponds to the connection.
 Here, with $Y=L X$, we have
\begin{equation*}
F(\Sigma_P(0)) \cdot F(\Sigma_P(1))^{-1} = F(\prev{\Sigma_P(1)}\pcomp \Sigma_P(0)) = \mathrm{Hol}_P(\prev{\tau^u} \pcomp \tau^o)^{-1} \text{,}
\end{equation*}
where the first equality is the functorality of $F$, and the second is the above mentioned general fact in combination with the definitions of $\tau^{o}$, $\tau^{u}$ and $\Sigma_P$.

For the equality on the right hand side of  \erf{eq:curvcond2} we use the reformulation of $G_P$ in terms of the curvature of $P$ given by Lemma \ref{lem:GPref}. We consider the smooth map $\beta\maps [0,1]^2 \to PX^{[4]}$ given by
\begin{equation*}
\beta(s,t) \df  (\id \pcomp \Sigma_1(s,t) \;,\;  \Sigma^{m}_{1,s}(t) \pcomp \Sigma^o_1(t) \; ,\; \Sigma^m_{2,s}(t) \pcomp \Sigma_2^o(t) \;,\; \id \pcomp \Sigma_2(s,t))\text{.}
\end{equation*}
Here, $\Sigma_{i,\sigma}$ the bigon defined by $\Sigma_{i,\sigma}(s,t) \df  \Sigma_i(\sigma \phi(s), t)$, where $\phi$ is a smoothing function. 
The four components of $\beta(s,t)$ are paths starting at $x$ and ending at $\Sigma_1(s,t)(1) = \Sigma_2(s,t)(1)$. Over $PX^{[4]}$, we infer from the fact that the fusion product is connection-preserving the formula
\begin{equation}
\label{eq:curvadd}
e_{12}^{*}\mathrm{curv}(P) + e_{23}^{*}\mathrm{curv}(P) + e_{34}^{*}\mathrm{curv}(P) = e_{14}^{*}\mathrm{curv}(P)\text{,}
\end{equation}
as well as $e_{ij}^{*}\mathrm{curv}(P) = -e_{ji}^{*}\mathrm{curv}(P)$.
Now we pull \erf{eq:curvadd} back along $\beta$ to $[0,1]^2$ and integrate. To abbreviate the notation, we write
\begin{equation*}
C_{ij} \df  \exp \left ( \int_{[0,1]^2} \beta^{*}e_{ij}^{*}\mathrm{curv}(P) \right ) \text{.}
\end{equation*}
We identify the four terms as follows.
\begin{enumerate}[(i)]
\item 
With Lemma \ref{lem:GPref}, we see that $C_{12} = G_P(\Sigma_1)$ and $C_{34}= G_P(\Sigma_2)^{-1}$.

\item
Consider for a general diffeological space $Y$ the map $\partial\maps BY \to LY$ that sends a bigon $\Sigma \in BY$ to the loop $\partial\Sigma \df  \prev{\Sigma(0)}\pcomp \Sigma(1)$, which goes counter-clockwise around  $\Sigma$ \cite[Section 3.2]{schreiber5}. By Stokes' Theorem, we have $C_{23} = \mathrm{Hol}_{P}(\partial(e_{23}\circ \beta))$, where $Y=LX$ in this case. Consider the family of bigons $h\in PBLX$ defined by
\begin{equation*}
\exd h\maps [0,1]^3 \to LX\maps (\sigma,s,t) \mapsto l(\Sigma^{m}_{1,\sigma s}(t) \pcomp \Sigma^o_1(t) \; ,\; \Sigma^m_{2,\sigma s}(t) \pcomp \Sigma_2^o(t))\text{,} \end{equation*}
which is a rank two map. Thus, $P\partial(h) \in PLLX$ is a homotopy between a loop $P\partial(h)(0) \in LLX$ and $\partial(e_{23}\circ \beta)$. Notice that $P\partial(h)(0)$ is  a thin loop. Thus, since the connection $P$ is superficial,
\begin{equation*}
C_{23} =\mathrm{Hol}_{P}(\partial(e_{23}\circ \beta)) =\mathrm{Hol}_{P}(P\partial(h)(0))=1\text{.}
\end{equation*} 

\item
Similar to (ii), we apply Stokes' Theorem and have $C_{14} = \mathrm{Hol}_P(\partial(e_{14}\circ \beta))$. Here, the loop $\partial(e_{14}\circ \beta)$ is rank-two-homotopic to the loop $\prev{\tau^{o}} \pcomp \tau^{u}$. Thus,
\begin{equation}
C_{14} = \mathrm{Hol}_P(\partial(e_{14}\circ \beta)) = \mathrm{Hol}_P(\prev{\tau^{o}} \pcomp \tau^{u})\text{.}
\end{equation}

\end{enumerate}
Identifications (i), (ii) and (iii) together with \erf{eq:curvadd} show that 
\begin{equation*}
G_P(\Sigma_1) \cdot G_P(\Sigma_2)^{-1} = \mathrm{Hol}_P(\prev{\tau^{u}} \pcomp \tau^{o})^{-1}\text{.}
\end{equation*}
This is the second part of the proof of \erf{eq:curvcond2}.
\end{proof}

This completes the construction of a connection on  $\mathcal{F}(l^{*}P)$, and we denote by $\uncon_x(P)$ the resulting diffeological $A$-bundle gerbe with  connection. Notice that
 Lemma \ref{lem:GPref} implies one part of  Corollary \ref{co:flat}:
\begin{proposition}
\label{prop:flatreg}
If the connection on $P$ is a flat,  the bundle gerbe $\uncon_x(P)$ is flat. 
\end{proposition}

It remains to discuss the functorality of the connection on regressed bundle gerbes. 
Suppose  $\varphi\maps P_1 \to P_2$ is an isomorphism between fusion bundles with connection. We have to equip the 1-isomorphism $\un_x(\varphi) = \mathcal{F}(l^{*}\varphi)$ with a connection compatible with the connections on $\uncon_x(P_1)$ and $\uncon_x(P_2)$. Since the principal $A$-bundle of the 1-isomorphism $\un_x(\varphi)$ is $l^{*}P_2$, it has already a connection, and its isomorphism (which is a composition of $\lambda_2$ and $\varphi$) is connection-preserving. It remains to check the condition for the curvings:
\begin{equation}
\label{eq:curvcond3}
\mathrm{pr}_2^{*}(B_{P_2}) - \mathrm{pr}_1^{*}(B_{P_1}) = \mathrm{curv}(l^{*}P_2)\text{.}
\end{equation}
Indeed, since $P_1$ and $P_2$ are isomorphic as fusion bundles with connection, they have the same curvature and thus $B_{P_1}=B_{P_2}$ by Lemma \ref{lem:GPref}. But then, \erf{eq:curvcond3} coincides with  \erf{eq:curvcond} for $P_2$ and is hence proved by Lemma \ref{lem:reconcurving2}. Thus we have defined a 1-isomorphism $\uncon_x(\varphi)$ in the 2-category $\diffgrbcon AX$. 

Finally, we have to check that the connections on the 1-isomorphisms $\uncon_x(\varphi)$ compose well, i.e. that the \quot{compositor} 2-isomorphism $\uncon_x(\psi) \circ \uncon_x(\varphi) \Rightarrow \uncon_x(\psi \circ \varphi)$ associated to  isomorphisms $\varphi:P_1 \to P_2$ and $\psi:P_2 \to P_3$  is connection-preserving. This condition only concerns the connections on the bundles (not the 2-forms $B_{P_i}$), and it follows from the condition that  $\psi$ and the fusion product $\lambda_3$ are connection-preserving isomorphisms. 
This finishes the construction of the regression functor $\uncon_x$. It is straightforward to see that all constructions are natural with respect to smooth, base point-preserving maps.

\subsection{Holonomy of the Reconstructed Gerbe}

\label{sec:holreg}

It is  interesting to express the surface holonomy of the regressed bundle gerbe $\uncon_x (P)$ in terms of the  connection  and the  fusion product on $P$. Below we provide formulas for cylinders, discs, and pairs of pants. Combining these by means of the gluing formula of Lemma \ref{lem:dbraneholprop} (c) yields the holonomy of a general surface.

The present section is purely complementary; its results are not used anywhere else in this paper. We work over a connected smooth manifold $M$.
We use some results developed for the proof of Theorem \ref{th:con} in Section \ref{sec:trafterreg}. There, we construct an element $p_{\mathcal{T}} \in P$
in the fibre over a loop $\tau$ from a given trivialization of $\tau^{*}\uncon_x(P)$, see Lemma \ref{lem:pTindep}. The first computation concerns the surface holonomy of $\uncon_x(P)$ around a cylinder.

\begin{proposition}
\label{prop:partransreg}
Let $\gamma\in PLM$ be a path and let $\exd\gamma\maps [0,1] \times S^1 \to M$ be its adjoint map, i.e. $\exd\gamma(t,z) \df  \gamma(t)(z)$.  Let $\mathcal{T}_0$ and $\mathcal{T}_1$ be trivializations of $\uncon_x(P)$ over the loops $\gamma(0)$ and $\gamma(1)$, respectively. Then, 
\begin{equation*}
\ptr{\gamma} (p_{\mathcal{T}_0})  = p_{\mathcal{T}_1} \cdot \mathscr{A}(\exd\gamma,\mathcal{T}_0,\mathcal{T}_1)\text{,} \end{equation*}
where $\mathscr{A}$ is the surface holonomy of $\uncon_x(P)$, and $\tau_{\gamma}$ is the parallel transport in $P$.
\end{proposition}

The proposition is proved by Lemmata \ref{lem:pTconnpres} and \ref{lem:calcterms}. Applying Proposition \ref{prop:partransreg} to a loop in $LM$, we obtain

\begin{corollary}
\label{co:holregtorus}
Let $\tau \in LLM$ be a loop in $LM$, and let $\exd\tau:S^1 \times S^1 \to M$ be its adjoint map, i.e. $\exd\tau(t,s) = \tau(t)(s)$. Then,
$\mathrm{Hol}_{\uncon_x(P)}(\exd\tau) = \mathrm{Hol}_P(\tau)^{-1}$.
\end{corollary}

Next we discuss briefly the surface holonomy of a disc $D$. Since it is of no further relevance for this article, we shall only present the result. 
There is a path $\delta \in PLD$ -- unique up to rank-two-homotopy -- that starts with the constant loop $\delta(0) \df  \id_{x}$ at some boundary point $x\in \partial D$, and ends with a loop $\delta(1)$ which is an orientation-preserving  parameterization of $\partial D$.

\begin{proposition}
\label{prop:holdisc}
Let $\phi\maps D \to M$ be a smooth map, let $\beta \df  L\phi(\delta(1)) \in LM$ be the boundary loop in the parameterization given by the path $\delta$, and let $\mathcal{T}$ be a trivialization of $\uncon_x (P)$ over $\beta$. Then, 
\begin{equation*}
p_{\mathcal{T}} \cdot \mathscr{A}(\phi,\mathcal{T}) = \ptr{PL\phi(\delta)}(\can(\id_{\phi(x)}))  \text{,}
\end{equation*}
where $\mathscr{A}(\phi,\mathcal{T})$ is the surface holonomy of  $\uncon_x (P)$,  $\tau_{PL\phi(\delta)}$ is the parallel transport in $P$ along the path $PL\phi(\delta) \in PLM$, and $\can$ is the canonical section of $P$ over the constant loops.
\end{proposition}

The proof of Proposition \ref{prop:holdisc} is similar as the one of Proposition \ref{prop:partransreg}, and we omit it for the sake of brevity. 
As an application, let us compute the surface holonomy of a 2-sphere. In order to apply the gluing formula of Lemma \ref{lem:dbraneholprop} (c), we choose a simple loop $\beta\maps S^1 \to S^2$ such that $\Sigma \setminus \mathrm{im}(\beta)$ is a disjoint union of two  discs $D^o$ and $D^u$. We can arrange the labels and orientations such that  $D^o$ has the orientation of $S^2$, and the orientations of both $D^{o}$ and $D^{u}$ induce the one prescribed by $\beta$ on their boundaries. We can further chose the base points of $D^u$ and $D^o$ so that they coincide in $S^2$. A given smooth map $\phi\maps S^2 \to M$ restricts to smooth maps $\phi^{o}$ and $\phi^{u}$, and the gluing formula shows that, for any bundle gerbe $\mathcal{G}$,
\begin{equation}
\mathrm{Hol}_{\mathcal{G}}(\phi) =  \mathscr{A}_{\mathcal{G}}(\phi,\mathcal{T},\mathcal{T}) =  \mathscr{A}_{\mathcal{G}}(\phi^o,\mathcal{T}) \cdot \mathscr{A}_{\mathcal{G}}(\phi^{u},\mathcal{T})^{-1} \text{.}
\end{equation}
Now consider the two paths $\delta^o\in PLD^{o}$ and $\delta^u\in PLD^{u}$ that combine to a loop $\tau \df \prev{\delta^u} \pcomp \delta^o \nobr\in\nobr LLS^2$. Proposition \ref{prop:holdisc} implies:
\begin{corollary}
Let $\phi:S^2 \to M$ be a smooth map, and let $LL\phi(\tau) \in LLM$ the  loop associated by the above construction. Then, $\mathrm{Hol}_{\uncon_x(P)}(\phi) = \mathrm{Hol}_{P}(\tau)^{-1}$.
\end{corollary}

Finally, let us look at a pair of pants. We assume the pair of pants in the standard form $\mathcal{P}$ shown in Figure \ref{fig:string}, as a unit disc with two  discs removed from its interior. Let $\phi\maps \mathcal{P} \to M$ be a smooth map, and let $\mathcal{T}_{12},\mathcal{T}_{23}$ and $\mathcal{T}_{13}$ be trivializations of  $\phi^{*}\uncon_x(P)$ over the three boundary components of $\mathcal{P}$. 
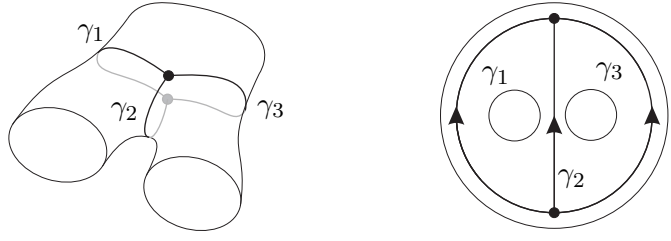
\begin{figure}
\begin{center}
\hspace{2.4cm}\includegraphics{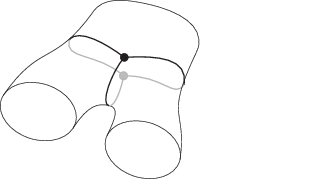}\setlength{\unitlength}{1pt}\begin{picture}(0,0)(205,53)\put(73.13357,126.18170){$\gamma_1$}\put(140.45017,97.67282){$\gamma_3$}\put(85.72195,96.65943){$\gamma_2$}\end{picture}
\includegraphics{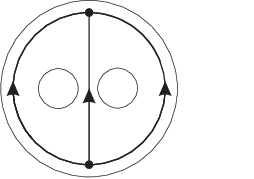}\setlength{\unitlength}{1pt}\begin{picture}(0,0)(391,46)\put(284.46860,103.15531){$\gamma_1$}\put(327.12321,104.94510){$\gamma_3$}\put(312.87742,64.19820){$\gamma_2$}\end{picture}
\end{center}
\caption{A pair of pants with a \quot{string tanga}, on the left in a form that explains the terminology and on the right in a standard form.}
\label{fig:string}
\end{figure}
To compute the surface holonomy $\mathscr{A}(\phi,\mathcal{T}_{12},\mathcal{T}_{23},\mathcal{T}_{13})$ of $\uncon_x(P)$, we choose a triple $(\gamma_1,\gamma_2,\gamma_3) \in P\mathcal{P}^{[3]}$ of paths  in $\mathcal{P}$ that form, in the given terminology, a \quot{string tanga}, see Figure \ref{fig:string}. We cut the pair of pants open along the paths $\gamma_i$, resulting in three cylinders $C_{12}$, $C_{23}$ and $C_{13}$ carrying the restrictions $\phi_{ij}$ of $\phi$, which we regard as paths $\phi_{ij} \in PLM$.  

\begin{proposition}
Let $\phi:\mathcal{P} \to M$ be a smooth map, and let paths $\phi_{ij} \in PLM$ and trivializations $\mathcal{T}_{ij}$ be chosen as described above. Then,
\begin{equation*}
\lambda( \ptr{\phi_{12}} (p_{\mathcal{T}_{12}})\otimes  \ptr{\phi_{23}} (p_{\mathcal{T}_{23}})) \cdot \mathscr{A}(\phi,\mathcal{T}_{12},\mathcal{T}_{23},\mathcal{T}_{13}) = \ptr{\phi_{13}} (p_{\mathcal{T}_{13}})  \text{,}
\end{equation*}
where $\lambda$ is the fusion product on $P$, $\tau_{\phi_{ij}}$ denotes the parallel transport in $P$, and $\mathscr{A}$ is the surface holonomy of $\uncon_x(P)$.
\end{proposition}

\begin{proof}
In order to apply the gluing formula  Lemma \ref{lem:dbraneholprop} (c) to $\mathscr{A}(\phi,\mathcal{T}_{12},\mathcal{T}_{23},\mathcal{T}_{13})$, we have to choose trivializations $\mathcal{T}_{ij}^{\,\prime}$ over the new boundary components $l(P\phi(\gamma_i),P\phi(\gamma_j))$. We  may chose them such that
$\lambda(\mathcal{T}_{12}^{\,\prime} \otimes \mathcal{T}_{23}^{\,\prime}) = \mathcal{T}_{13}^{\,\prime}$.
We recall that this relation implies the existence of 2-isomorphisms
\begin{equation*}
\phi_1\maps \iota_1^{*}\mathcal{T}_{12}^{\,\prime} \Rightarrow \iota_1^{*}\mathcal{T}_{13}^{\,\prime}
\quomma
\phi_2\maps \iota_2^{*}\mathcal{T}_{12}^{\,\prime} \Rightarrow \iota_1^{*}\mathcal{T}_{23}^{\,\prime}
\quand
\phi_3\maps \iota_2^{*}\mathcal{T}_{23}^{\,\prime} \Rightarrow \iota_2^{*}\mathcal{T}_{13}^{\,\prime}
\end{equation*}
that satisfy the condition $\phi_1 = \phi_{3} \bullet \phi_{2}$ over the two common end points of the paths $\gamma_i$. Feeding this information into the gluing formula, we obtain
\begin{equation*}
\mathscr{A}(\phi,\mathcal{T}_{12},\mathcal{T}_{23},\mathcal{T}_{13}) = \mathscr{A}(\phi_{12},\mathcal{T}_{12},\mathcal{T}_{12}^{\,\prime})^{-1} \cdot \mathscr{A}(\phi_{23},\mathcal{T}_{23},\mathcal{T}_{23}^{\,\prime})^{-1} \cdot \mathscr{A}(\phi_{13},\mathcal{T}_{13},\mathcal{T}_{13}^{\,\prime})\text{.}
\end{equation*}
Rewriting the holonomies of $\phi_{ij}$ using  Proposition \ref{prop:partransreg}  shows the claim. 
\end{proof}

\setsecnumdepth{2}

\section{Proof of Theorem \ref{th:con}}

\label{sec:proof}

We prove that transgression $\trcon$ and regression $\uncon$ form an equivalence of categories by constructing natural equivalences
\begin{equation*}
\mathcal{A}\maps \uncon_x \circ \trcon \Rightarrow \id_{\hc 1\diffgrbcon AM}
\quand
\varphi \maps \trcon \circ \uncon_x \Rightarrow \id_{\fusbunconsf A {L M}}
\text{.}
\end{equation*}
Since the pair ($\trcon$, $\uncon_x$) forms an equivalence of categories in which $\trcon$ is monoidal, it follows that also  $\uncon_x$ is  monoidal, and that ($\trcon$, $\uncon_x$) is a monoidal equivalence as claimed in Theorem \ref{th:con}. 

\subsection{Regression after Transgression}

\label{sec:regaftertr}

We have to associate to each diffeological bundle gerbe $\mathcal{G}$ with connection over $M$ a natural 1-isomorphism 
\begin{equation*}
\mathcal{A}_{\mathcal{G}}\maps \uncon_x(L\mathcal{G}) \to \mathcal{G}
\end{equation*}
in the 2-groupoid $\diffgrbcon AM$. We proceed in three steps. First we construct the 1-isomorphism $\mathcal{A}_{\mathcal{G}}$ in the category $\diffgrb AM$, i.e. without connections. In the second step we add the connections. In the third step we prove that $\mathcal{A}_{\mathcal{G}}$ is natural in $\mathcal{G}$. 

\subsubsection{Construction of the Isomorphism $\mathcal{A}_{\mathcal{G}}$}

We recall that the bundle gerbe $\un_x(L\mathcal{G})$ has the subduction $\ev_1\maps \px Mx \to M$ and over $\px Mx^{[2]}$ it has the bundle $l^{*}\lg\mathcal{G}$ equipped with its product $\lambda_{\mathcal{G}}$. 
To construct the 1-isomorphism $\mathcal{A}_{\mathcal{G}}$
we  have to specify a subduction $\zeta\maps Z \to \px Mx \times_M Y$. We put $Z \df  \px Mx \times_M Y$ and $\zeta$ the identity.  Next we have to construct a principal $A$-bundle $Q$ over $Z$.
We choose a lift $y_0\in Y$  of the  base point $x\in M$, i.e. $\pi(y_0)=x$. Consider a point $(\gamma,y)\in Z$, i.e. $\ev(\gamma)=(x,\pi(y))$. The fibre of the bundle $Q$ over the point $(\gamma,y)$ consists of triples $(\mathcal{T},t_0,t)$ where $\mathcal{T}\maps \gamma^{*}\mathcal{G} \to \mathcal{I}_0$ is a connection-preserving trivialization. Explicitly, $\mathcal{T}$ has a principal $A$-bundle $T$ over  $Y_{\gamma} \df  [0,1] \lli{\gamma}\times_{\pi} Y$ and an isomorphism $\tau\maps P \otimes \pi_2^{*}T \to \pi_1^{*}T$
of principal $A$-bundles over $Y_{\gamma}^{[2]}$. Further, $t_0$ is an element in $T$ projecting to the point $(0,y_{0}) \in Y_{\gamma}$, and $t$ is an element  in $T$ projecting to $(1,y)\in Y_{\gamma}$.
We identify two triples $(\mathcal{T},t_0,t)$ and $(\mathcal{T}',t_0',t')$ if there exists a  2-isomorphism $\varphi\maps \mathcal{T} \Rightarrow \mathcal{T}'$ with $\varphi(t_0) = t_0'$ and $\varphi(t)= t'$.
The total space $Q$ of the principal $A$-bundle we are going to construct is the disjoint union of equivalence classes of triples over all points of $Z$.
The evident projection is denoted $p:Q \to Z$.

A diffeology on $Q$ is defined similar to the one on $L\mathcal{G}$ performed in Section \ref{sec:transgression_bundle}. A map $c:U \to Q$ is a plot, if the composite $p \circ c:U \to Z$ is smooth, and if every point $u\in U$ has an open neighborhood $W \subset U$ such that
\begin{enumerate}
\item 
 
there exists a trivialization $\mathcal{T}\maps c_W^{*}\mathcal{G} \to \mathcal{I}_{\rho}$, where $c_W$ is given by
\begin{equation*}
\alxydim{@C=1.5cm}{[0,1] \times W \ar[r]^-{\id \times p \circ c} & [0,1] \times Z \ar[r]^-{\mathrm{pr}} & [0,1] \times \px Mx \ar[r]^-{\ev} & M\text{.}}
\end{equation*}

\item
there exist smooth sections $s^a\maps W \to a^{*}T$ and $s^e\maps W \to e^{*}T$, where $T$ is the principal $A$-bundle of $\mathcal{T}$, defined over $Y_W \df  ([0,1] \times W) \lli{c_W}\times_{\pi} Y$, and the smooth maps 
$a,e:W \to Y_W$ are defined by 
\begin{equation*}
a(w) \df  (0,w,y_0)
\quand 
e(w) \df  (1,w,y_w)\text{,}
\end{equation*}
where $y_w$ is the $Y$-component of $p(c(w)) \in Z$.

\item
$c(w) = (\iota_w^{*}\mathcal{T}, s^a(w),s^e(w))$
for all $w\in W$. 
\end{enumerate}
Along the lines of the discussion in Section \ref{sec:transgression_bundle} one can check that this defines a diffeology on $Q$ that makes the projection $p:Q \to Z$ smooth. The action of $A$ on $Q$ is the action of $A$ on the element $t\in T$ in the triples $(\mathcal{T},t_0,t)$.
It is straightforward to check that these definitions yield a diffeological principal $A$-bundle $Q$ over $Z$.

The last ingredient we need in the construction of the 1-isomorphism $\mathcal{A}_{\mathcal{G}}$ is an isomorphism
\begin{equation*}
\alpha\maps l^{*}L\mathcal{G} \otimes \zeta_2^{*}Q \to \zeta_1^{*}Q \otimes P \end{equation*}
between bundles over $Z^{[2]} \df  Z \times_M Z$. Here, $\zeta_1,\zeta_2\maps Z^{[2]} \to Z$ are the two projections, and $l^{*}L\mathcal{G}$ and $P$ are pulled back to $Z^{[2]}$ along the evident maps. Fibrewise over a point $(\gamma_1,\gamma_2,y_1,y_2)$, $\alpha$ is a map
\begin{equation*}
\alpha\maps  \lg\mathcal{G}_{\tau} \otimes Q_{\gamma_2,y_2} \to Q_{\gamma_1,y_1} \otimes P_{y_1,y_2} \text{,}
\end{equation*}
where $\tau \df  l(\gamma_1,\gamma_2)$. For $\mathcal{T}\maps  \tau^{*}\mathcal{G} \to \mathcal{I}_0$ representing an element in $L\mathcal{G}_{\tau}$, $t_0 \in T_{0,y_0}$ and $t_1\in T_{\frac{1}{2},y_1}$, we define
\begin{equation*}
\alpha(\mathcal{T} \otimes (\mathcal{T}_2, t_0,t_{2})) \df  q \otimes(\mathcal{T}_{1},t_0,t_{1})\text{,}
\end{equation*}
where $\mathcal{T}_k \df  \iota_k^{*}\mathcal{T}$, the point $q\in P_{y_1,y_2}$ is chosen arbitrarily and $t_2 \in T_{-1,y_2}$ is determined by $\tau(q \otimes t_2) = t_1$. This is independent of the choice of $q$, smooth and $A$-equivariant.

\begin{lemma}
The pair $(Q,\alpha)$ defines a 1-isomorphism $\mathcal{A}_{\mathcal{G}}\maps  \un_x(L\mathcal{G}) \to \mathcal{G}$. 
\end{lemma}

\begin{proof}
We have to show that $\alpha$ is compatible with the fusion products \cite[Definition 2]{waldorf1}. This can be checked fibrewise over a point $(\gamma_1,\gamma_2,\gamma_3,y_1,y_2,y_3)\in Z^{[3]}$.  We choose trivializations $\mathcal{T}_{12}$, $\mathcal{T}_{23}$ and $\mathcal{T}_{13} \df  \lambda_{\mathcal{G}}(\mathcal{T}_{12},\mathcal{T}_{23})$, as well as elements $q_{12}$, $q_{23}$ and $q_{13} \df  \lambda(q_{12},q_{23})$. We recall from Section \ref{sec:constrconn} that the first relation implies the existence of 2-isomorphisms
\begin{equation*}
\phi_1\maps \iota_1^{*}\mathcal{T}_{12} \Rightarrow \iota_1^{*}\mathcal{T}_{13}
\quomma
\phi_2\maps \iota_2^{*}\mathcal{T}_{12} \Rightarrow \iota_1^{*}\mathcal{T}_{23}
\quand
\phi_3\maps \iota_2^{*}\mathcal{T}_{23} \Rightarrow \iota_2^{*}\mathcal{T}_{13}
\end{equation*}
between trivializations of bundle gerbes over the interval $[0,1]$ such that
$\phi_1|_k = \phi_{3}|_k \bullet \phi_{2}|_k$
for $k=0,1$.
We start with an element in $P_{y_1,y_2} \otimes P_{y_2,y_3} \otimes Q_{\gamma_3,y_3}$, namely the element $x \df  q_{12} \otimes q_{23} \otimes (\iota_2^{*}\mathcal{T}_{23},t_0,t_1)$. A computation shows
\begin{equation*}
\zeta_{13}^{*}\alpha((\lambda \otimes \id)(x)) =  (\iota_1^{*}\mathcal{T}_{12}, \phi_1^{-1}(\phi_3(t_0)), \phi_1^{-1}(\tau_{13}(q_{13} \otimes \phi_3(t_1)))) \otimes \mathcal{T}_{13}\text{.} \end{equation*}
On the other hand, we get
\begin{equation*}
(\id \otimes \lambda_{\mathcal{G}})(\zeta_{12}^{*}\alpha \otimes \id)(\id \otimes \zeta_{23}^{*}\alpha)(x) = (\iota_1^{*}\mathcal{T}_{12},\phi_2^{-1}(t_0),\tau_{12}(q_{13}  \otimes \phi_2^{-1}(t_1))) \otimes \mathcal{T}_{13} \text{.}
\end{equation*}
The coincidence of the two results proves that $(Q,\alpha)$ is a 1-isomorphism.
\end{proof}

We remark that different choices of the lift $y_0$ of $x$ lead to 2-isomorphic 1-morphisms. So, as a morphism in $\hc 1 \diffgrb A M$, $\mathcal{A}_{\mathcal{G}}$ is independent of that choice. 

\subsubsection{Construction of a Connection on $\mathcal{A}_{\mathcal{G}}$}

We define a connection on $Q$ in the same way as we defined the connection on $L\mathcal{G}$, namely by specifying an object $F_Q \in \fun Q A$, see Appendix \ref{app:functorsandforms}. If $\tilde\gamma \in PQ$ is a path in $Q$, let $\gamma \in P\px Mx$ denote its projection to $\px Mx$, and let $\exd\gamma\maps [0,1]^2 \to M$ be the adjoint map defined by $\exd\gamma(s,t)\df \gamma(s)(t)$. We use the holonomy $\mathscr{A}_{\mathcal{G}}$ from Section \ref{sec:gerbeconnections} to associate to $\tilde\gamma$ a number. Here we are in the situation that our surface $\Sigma \df [0,1]^2$ has one boundary component with corners. We thus have to specify a \emph{boundary record}, i.e.  trivializations over the smooth parts, and 2-isomorphisms over the corners: 
\begin{enumerate}
\item 
Over $b_l \df  [0,1] \times \left \lbrace 0 \right \rbrace$, the map $\exd\gamma$ is constant with value $x$. It follows that the pullback $(\exd\gamma)^{*}\mathcal{G}|_{b_l}$ has the smooth section $\sigma_{y_0}(s) \df  (s,y_0) \in \Sigma \lli{\exd\gamma} \times_{\pi} Y$, which defines by Lemma \ref{lem:sectriv} a trivialization $\mathcal{T}_l$.

\item
Over $b_r \df [0,1] \times \left \lbrace 1 \right \rbrace $ we also have a section of the subduction of $\phi(\exd\gamma)^{*}\mathcal{G}|_{b_r}$, determined by the projection $\beta \in PY$ of the given path $\tilde\gamma$ to $Y$. This section is $\sigma_r(s) \df  ((s,1),\beta(s))\in \Sigma \lli{\exd\gamma}\times_{\pi} Y$. The corresponding trivialization is denoted $\mathcal{T}_r$.

\item
Over the remaining boundary components $b_k \df  \left \lbrace k \times [0,1] \right \rbrace $ for $k=0,1$ we choose representatives $(\mathcal{T}^k,t_0^k,t^k)$ of the given elements $\tilde\gamma(k) \in Q$. In particular, these contain trivializations $\mathcal{T}^1$ and $\mathcal{T}^2$. 

\item
Over the corners $(0,0)$ and $(0,1)$ we choose any  2-isomorphisms $\varphi_{0l}\maps \mathcal{T}^0\Rightarrow \mathcal{T}_{l}$ and $\varphi_{0r}\maps \mathcal{T}^0 \Rightarrow \mathcal{T}_r$ and consider the elements $u_l \df  \varphi_{0l}(t_0^0) \in T_l|_{0,y_0}$ and $u_r \df \varphi_{0r}(t^0)\in T_l|_{0,\beta(0)}$. Consider the evident paths $\alpha_l(t) \df  (t,y_0)$ and $\alpha_r(t) \df (t,\beta(t))$ in $b_l \lli{\exd\gamma}\times_{\pi} Y$ and $b_r \lli{\exd\gamma}\times_{\pi} Y$, respectively. Then, there are unique 2-isomorphisms $\varphi_{1l}\maps \mathcal{T}^1 \Rightarrow \mathcal{T}_l$ and $\varphi_{1r}\maps \mathcal{T}^1 \Rightarrow \mathcal{T}_r$ over the corners $(1,0)$ and $(1,1)$, respectively, such that $\varphi_{1l}(t_0^1) = \ptr{\alpha_{l}}(u_l)$ and $\varphi_{1r}(t^1) = \ptr{\alpha_r}(u_r)$, see Lemma \ref{lem:gerbehomspoint}. 
\end{enumerate} 
Summarizing,
\begin{equation*}
\mathscr{B} \df  \left \lbrace \mathcal{T}_l,\mathcal{T}_r,\mathcal{T}^0,\mathcal{T}^1,\varphi_{0l},\varphi_{0r}^{-1},\varphi_{1l}^{-1},\varphi_{1r} \right \rbrace
\end{equation*}
is a boundary record for the boundary of $\Sigma$, and we define:
\begin{equation}
\label{eq:condefQ}
F_Q(\tilde\gamma) \df  \mathscr{A}_{\mathcal{G}}(\exd\gamma,\mathscr{B}) \in A\text{.}
\end{equation}
Next we show in three steps that the resulting map $F_Q\maps PQ \to A$  is an object in $\fun QA$. 
\begin{enumerate}
\item 
\emph{Smoothness} follows from Lemma \ref{lem:holsmooth} and the fact that all trivializations we have used to form the boundary record $\mathscr{B}$ depends smoothly on the paths $\tilde\gamma$\maps $\mathcal{T}^0$ and $\mathcal{T}^1$ in terms of the diffeology on $Q$, and $\mathcal{T}_l$ and $\mathcal{T}_r$ in terms of the diffeology on $L\mathcal{G}$. 

\item
\emph{Thin homotopy invariance}. A thin homotopy $h \in PPQ$ between paths $\tilde\gamma$ and $\tilde\gamma'$ induce rank two homotopy between $\exd\gamma$ and $\exd\gamma'$,  a thin homotopy $h_Y$ between the paths $\beta$ and $\beta'$, and a thin homotopy $h_1$ between $\ev_1(\gamma)$ and $\ev_1(\gamma')$. We denote by $\mathcal{T}_r$
and $\mathcal{T}_r'$ the trivializations determined by $\beta$ and $\beta'$, respectively, and by $\mathscr{B}$ and $\mathscr{B}\,'$ the corresponding boundary records. Consider the trivialization $\mathcal{T}_Y$ of $h_1^{*}\mathcal{G}$ coming from the section defined by $h_Y$. It constitutes a boundary record $\mathscr{B}_1$ for $h_1$, which has by construction the same trivialization over $\ev_1(\gamma)$ as $\mathscr{B}$ and the same trivialization over $\ev_1(\gamma)$ as $\mathscr{B}'$. A slight generalization of Lemma \ref{lem:dbraneholprop} (b) shows then that
\begin{equation*}
\mathscr{A}_{\mathcal{G}}(\exd\gamma,\mathscr{B}) = \mathscr{A}_{\mathcal{G}}(\exd{(\gamma')},\mathscr{B}') \cdot \mathscr{A}_{\mathcal{G}}(h_1,\mathscr{B}_1)\text{.}
\end{equation*}
But  $\mathscr{A}_{\mathcal{G}}(h_1,\mathscr{B}_1)=1$ because the trivialization $\mathcal{T}_Y$ has the vanishing 2-form, since $h_Y$ has rank one. This shows that $F_Q(\tilde \gamma) = F_Q(\tilde\gamma')$. 

\item
\emph{Functorality}. For composable paths $\tilde\gamma_1,\tilde\gamma_2$ in $Q$ we must have
\begin{equation*}
F_Q(\tilde\gamma_2 \pcomp \tilde\gamma_1) = F_Q(\tilde\gamma_1) \cdot F_Q(\tilde\gamma_2)\text{.}
\end{equation*}
This is a simple application of the gluing formula Lemma \ref{lem:dbraneholprop} (c) for $\mathscr{A}_{\mathcal{G}}$, similar to the discussion in Section \ref{sec:constrconn}. 
\end{enumerate}
We conclude that $F_Q$ is an object in $\fun QA$ and thus defines a 1-form $\omega_Q \in \Omega^1_{\mathfrak{a}}(Q)$.

\begin{lemma}
The 1-form $\omega_Q$ is a connection on $Q$.
\end{lemma}

\begin{proof}
Similar to the proof of Proposition \ref{prop:connectionprop1} we have to prove
\begin{equation*}
g(1) \cdot F_Q(\tilde \gamma g) = F_Q(\tilde \gamma) \cdot g(0)
\end{equation*}
for all $g \in PA$ and $\tilde \gamma \in PQ$. But the only difference between $F_Q(\tilde\gamma g)$ and $F_Q(\tilde\gamma)$ is that we have $t^k.g(k)$ instead of $t^k$ in the representatives $(\mathcal{T}^k,t_0^k,t^k)$. This produces a difference of $g(0)\cdot g(1)^{-1}$ in the holonomy formula.
\end{proof}

Next is the check that the connection $\omega_{Q}$ on $Q$ makes $\mathcal{A}_{\mathcal{G}}$ a  1-isomorphism in the 2-groupoid $\diffgrbcon AM$ of bundle gerbes with connection. The conditions we have to prove are summarized in the following lemma. There, we denote by  $B_{L\mathcal{G}} \in \Omega^2_{\mathfrak{a}}(\px Mx)$  the curving of $\uncon_x(L\mathcal{G})$ and by $B\in \Omega^2_{\mathfrak{a}}(Y)$  the curving of the given bundle gerbe $\mathcal{G}$.
\begin{lemma}
\label{lem:concheck}
\begin{enumerate}[(i)]
\item 
The isomorphism $\alpha$ is connection-preserving.

\item
$\mathrm{curv}(Q) = B_{L\mathcal{G}} - B$. 
\end{enumerate}
\end{lemma}

\begin{proof}[Proof of Lemma \ref{lem:concheck} (i)]
Let $\Gamma$ be a path in $\zeta_1^{*}Q \times L\mathcal{G}$. We denote by $\tilde\gamma \in PL\mathcal{G}$ its projection to the second factor and by $\gamma \in PLM$ its further projection to the base of $L\mathcal{G}$. Our plan is to compute
\begin{equation}
\label{eq:concheck1}
F(\tilde\gamma) = \mathscr{A}_{\mathcal{G}}(\exd\gamma,\mathcal{T}_0,\mathcal{T}_1)\text{,}
\end{equation}
where $\mathcal{T}_0$ and $\mathcal{T}_1$ are trivializations representing the given elements $\tilde\gamma(0)$ and $\tilde\gamma(1)$, respectively.
The first step is to cut the cylinder $C=[0,1] \times S^1$ on which $\exd\gamma$ is defined open along the lines $e_l \df  [0,1] \times \left \lbrace 0 \right \rbrace$ and $e_r \df  [0,1] \times \left \lbrace \frac{1}{2} \right \rbrace$. The two parts can be identified with  squares $\Sigma^u$ and $\Sigma^o$ using embeddings $j^o \df  \id \times \iota_1\maps \Sigma^o \to C$ and $j^u \df  \id \times \iota_2\maps \Sigma^u \to C$, of which $j^o$ is orientation-reversing, and $j^u$ orientation-preserving.

We need more notation: by $\tilde\beta^o\in PQ$ we denote the projection of $\Gamma$ to $Q$, by $\beta^o\in PZ$ its projection to the base of $Q$, and by $\beta^o_r \in PY$ its further projection to $Y$. 
The boundary of $\Sigma^o$ consists of four smooth parts denoted $e_l^o$, $e_r^o$, $b_0^o$ and $b_1^o$. The boundary of $\Sigma^u$ is labeled similarly (see Figure \ref{fig:cyclindercut}).
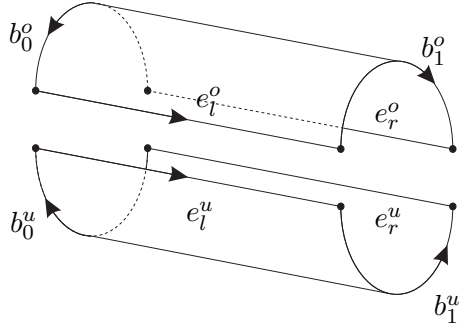
\begin{figure}
\begin{center}
\includegraphics{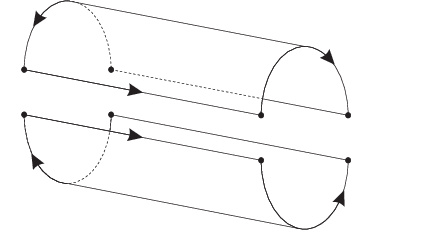}\setlength{\unitlength}{1pt}\begin{picture}(0,0)(543,691)\put(478.65033,764.13969){$e^o_r$}\put(408.20712,726.74391){$e^u_l$}\put(342.30756,722.81849){$b^u_0$}\put(496.09417,788.80394){$b^o_1$}\put(341.84268,793.94400){$b^o_0$}\put(411.80712,769.94391){$e^o_l$}\put(478.65033,724.53969){$e^u_r$}\put(501.01909,691.20879){$b^u_1$}\end{picture}
\end{center}
\caption{A cylinder cut open into two squares.}
\label{fig:cyclindercut}
\end{figure}
The trivializations $\mathcal{T}_0$ and $\mathcal{T}_1$ pull back to trivializations $\mathcal{T}_0^o$, $\mathcal{T}_0^u$, $\mathcal{T}_1^o$ and $\mathcal{T}_1^u$ over the  boundary parts $b^{u,o}_{0,1}$. We also have to equip the new boundary parts $e^{u,o}_{l,r}$ with trivializations. Firstly, notice that $\exd\gamma|_{e_l}$ is constantly equal to $x$, so that $y_0$ defines a (constant) section $\sigma_{y_0}$ into the subduction of $(\exd\gamma)^{*}\mathcal{G}|_{e_{l}}$. Let $\mathcal{T}_l$ be the corresponding trivialization. Secondly, the path $\beta^o_r$ defines a section into the subduction of $(\exd\gamma)^{*}\mathcal{G}|_{e_r}$. We denote the corresponding trivialization by $\mathcal{T}_r$. The trivializations $\mathcal{T}_l$ and $\mathcal{T}_r$ pull back to trivializations $\mathcal{T}_l^u$, $\mathcal{T}^u_r$, $\mathcal{T}^o_l$ and $\mathcal{T}^o_r$. 

It remains to equip the eight points $v^{o,u}_{(0,1),(l,r)}$ with 2-isomorphisms. We  choose some 2-isomorphisms $\varphi^{o}_{0,l}$ and $\varphi^{o}_{0,r}$. We also choose representatives $\tilde\beta^o(0) = (\mathcal{T}^o_0,a_0,b_0^o)$ and $\tilde\beta^o(1)=(\mathcal{T}_1^o,a_1,b_1^o)$, with the trivializations coinciding with the ones pulled back from $\mathcal{T}_1$ and $\mathcal{T}_2$. Parallel transport in the bundles of the trivializations determine further 2-isomorphisms $\varphi^o_{1,l}$ and $\varphi^o_{1,r}$. This determines a boundary record $\mathcal{B}^o$ for $\Sigma^o$, and we note that $\mathscr{A}_{\mathcal{G}}(\exd\gamma \circ j^o, \mathcal{B}^o) = F_Q(\tilde\beta^{o})$. Further, we may use $\varphi^u \df  \varphi^o$ over all four points of $\Sigma^u$, and thus obtain a boundary record $\mathcal{B}^u$ for $\Sigma^u$. All together,  \erf{eq:concheck1} becomes 
\begin{equation}
\label{eq:concheck2}
F(\tilde\gamma)  = F_Q(\tilde\beta^{o})^{-1} \cdot \mathscr{A}_{\mathcal{G}}(\exd\gamma \circ j^u, \mathcal{B}^u)\text{.}
\end{equation}
Now we consider the path $\alpha^{-1} \circ \Gamma$ in $P \times \zeta_2^{*}Q$. We denote its projection to $Q$ by $\tilde\beta^u$, its projection to the base by $\beta^u\in PZ$ and its further projection to $Y$ by $\beta_r^u\in PY$. 
In order to compute $F_Q(\tilde\beta^u)$, we choose representatives  $\tilde\beta^u(0) = (\mathcal{T}^u_0,a_0,b_0^u)$ and $\tilde\beta^u(1)=(\mathcal{T}_1^u,a_1,b_1^u)$, with the trivializations coinciding with the ones pulled back from $\mathcal{T}_1$ and $\mathcal{T}_2$, and the $a$'s like above. We use the 2-isomorphism $\varphi^u_{0,l}$ from above, but have to choose new 2-isomorphisms $\psi^{u}_{0,r}$. They determine 2-isomorphisms $\varphi^u_{1,r}$ and $\psi^u_{1,l}$, all together making up a boundary record $\mathcal{C}^u$ for $\Sigma^u$, in such a way that
\begin{equation}
\label{eq:concheck3}
F_Q(\tilde\beta^u) = \mathscr{A}_{\mathcal{G}}(\gamma \circ j^{u}, \mathcal{C}^u)\text{.}
\end{equation} 
Next we compute the difference between the boundary records $\mathcal{B}^u$ and $\mathcal{C}^u$. They differ only in the trivialization over $e_r^u$ and in the two 2-isomorphisms attached to it. Namely, $\mathcal{B}^u$ has the trivialization $\mathcal{T}^u_r$ coming from the section $\beta^o_r$ while $\mathcal{C}^u$ has a trivialization $\mathcal{T}_r$ coming from the section $\beta_r^u$. By Lemma \ref{lem:sectriv}, the bundle $R \df  (\beta_r^o,\beta_r^u)^{*}P$ comes with a 2-isomorphism $\rho\maps \mathcal{T}_r^u \otimes R \Rightarrow \mathcal{T}_r$. Consider trivializations $t_k\maps R|_k \to \trivlin_0$ over the endpoints $k=0,1$, that combine with $\rho$ to 2-isomorphisms $\rho_k\maps \mathcal{T}_r^u|_k \Rightarrow \mathcal{T}_r|_k$. 
We might have chosen $\psi^u_{0,r} = \varphi^o_{0,r} \circ \rho_0$. Then, $\psi^u_{1,r} =  a_{t_0,t_1} \otimes  \varphi^o_{1,r} \circ \rho_1$, where $a_{t_0,t_1} \in A$ is produced by parallel transport in $R$ in combination with the trivializations $t_0$ and $t_1$ at the endpoints. Thus
\begin{equation}
\label{eq:concheck4}
\mathscr{A}_{\mathcal{G}}(\exd\gamma \circ j^{u}, \mathcal{C}^u) = \mathscr{A}_{\mathcal{G}}(\exd\gamma \circ j^{u}, \mathcal{B}^u) \cdot  a_{t_0,t_1}\text{.} 
\end{equation} 
It remains to identify $a_{t_0,t_1}$. Let $F_P\maps PP \to A$ be the smooth map corresponding to the connection on $P$. We denote by $\eta \in PP$ the projection of the path $\alpha^{-1} \circ \Gamma$ to $P$. Notice that $\eta$ projects to the path $(\beta_r^o,\beta_r^u)$ in $Y^{[2]}$. We may have chosen the trivializations $t_0$ and $t_1$ such that $\eta(k)=t_k$ for $k=1,2$, which implies 
\begin{equation}
\label{eq:concheck5}
F_P(\eta) = a_{t_0,t_1}^{-1}. 
\end{equation}
Summarizing equations \erf{eq:concheck2} to \erf{eq:concheck5} we have
\begin{equation}
F(\tilde\gamma) =F_Q(\tilde\beta^{o})^{-1} \cdot F_Q(\tilde\beta^u) \cdot F_P(\eta)\text{.}
\end{equation}
This shows that $\alpha$ preserves the connections.
\end{proof}

\begin{proof}[Proof of Lemma \ref{lem:concheck} (ii)]
We work on the total space of $Q$, and shall prove that the object $F_Q$ in $\fun QA$ defines a 1-morphism in $\tfun QA$ between the pullbacks of $G_{L\mathcal{G}}$ and $G_B$ to $Q$, where $G_{L\mathcal{G}} \in \tfun {\px Mx}A$ is constructed in Section \ref{sec:reconcon} and  $G_B \df  \mathfrak{P}_2(B) \in \tfun YA$ is the 2-functor associated to the curving $B$. That is, we have to prove condition \erf{eq:pseudotrans}, namely
\begin{equation}
\label{eq:connpres1}
F_Q(\tilde\beta_o) \cdot G_{L\mathcal{G}}(\Sigma_{P}) = G_B(\Sigma_Y) \cdot F_Q(\tilde\beta_u)
\end{equation}
for all bigons $\Sigma \in BQ$, where $\ev(\Sigma)=(\tilde\beta_o,\tilde\beta_u)$, and $\Sigma_P$ and $\Sigma_Y$ are the projections to $\px Mx$ and $Y$, respectively. Further, we denote by $\eta_o,\eta_u$ the projections of $\tilde\beta_o$ and $\tilde\beta_u$ to $\px Mx$. 
Consider the path $\gamma_{\Sigma_P}\in PLM$ used in the definition of $G_{L\mathcal{G}}$ in Section \ref{sec:reconcon},  and let $\mathcal{T}_0$ and $\mathcal{T}_1$ be the canonical trivializations over $\gamma_{\Sigma_P}(0)$ and $\gamma_{\Sigma_P}(1)$, respectively. 
Then, 
\begin{equation}
\label{eq:connpres2}
G_{L\mathcal{G}}(\Sigma_P) = \mathscr{A}_{\mathcal{G}}(\exd{\gamma_{\Sigma_P}}, \mathcal{T}_0,\mathcal{T}_1)\text{,}
\end{equation}
with $\exd{\gamma_{\Sigma_P}}\maps [0,1] \times S^1 \to M$ the adjoint of $\gamma_{\Sigma_P}$. The plan is to cut this surface holonomy into three pieces; these will provide the other three terms in \erf{eq:connpres1}. We consider three embeddings $[0,1] \to S^1$ defined by
\begin{equation*}
\iota_o(t) \df  {\textstyle\frac{1}{4}t}
\quomma
\iota_r(t) \df   {\textstyle\frac{1}{4}t+\frac{1}{4}}
\quand
\iota_u(t) \df   {\textstyle\frac{1}{4}t+\frac{3}{4}}
\end{equation*} 
and notice that $\exd{\gamma_{\Sigma_P}} \circ (\id \times \iota_o) = \tilde \eta_o$ and $\exd{\gamma_{\Sigma_P}} \circ (\id \times \iota_u) = \tilde \eta_u$, while $\eta_r \df  \exd{\gamma_{\Sigma_P}} \circ (\id \times \iota_r) = B\pi(\Sigma_Y)$. In particular, $\eta_{r}$ comes with a lift to $Y$, namely $\Sigma_Y$. As a consequence, there is a canonical trivialization $\mathcal{S}\maps \eta_r^{*}\mathcal{G} \to \mathcal{I}_{\rho}$ with $\rho \df  \Sigma_Y^{*}B$. We shall choose 2-isomorphisms $\psi_k\maps \iota_r^{*}\mathcal{T}_k \Rightarrow \mathcal{S}$ for $k=0,1$ defining boundary records $\mathscr{B}_o \df  \left \lbrace \iota_o^{*}\mathcal{T}_0, \iota_o^{*}\mathcal{T}_1, \mathcal{S}, \psi_0,\psi_1 \right \rbrace$ and $\mathscr{B}_u \df  \left \lbrace \iota_u^{*}\mathcal{T}_0, \iota_u^{*}\mathcal{T}_1, \mathcal{S}, \psi_0,\psi_1 \right \rbrace$. Then,
\begin{equation}
\label{eq:connpres3}
\mathscr{A}_{\mathcal{G}}(\exd{\gamma_{\Sigma_P}}, \mathcal{T}_0,\mathcal{T}_1) = \mathscr{A}_{\mathcal{G}}(\eta_o,\mathscr{B}_o)^{-1} \cdot \mathscr{A}_{\mathcal{G}}(\eta_u,\mathscr{B}_u) \cdot \mathscr{A}_{\mathcal{G}}(\eta_r,\mathcal{S})\text{.}
\end{equation}
Strictly speaking, there is a fourth factor corresponding to the forth quarter of the circle. But there, $\exd{\gamma_{\Sigma_P}}$ has rank one so it does not contribute. 
We have 
\begin{equation}
\label{eq:connpres4}
\mathscr{A}_{\mathcal{G}}(\eta_r,\mathcal{S}) = \exp \left ( \int_{[0,1]} \rho \right ) = \exp \left ( \int_{[0,1]} \Sigma_Y^{*}B \right ) = G_B(\Sigma_Y)\text{.}
\end{equation}
Further, we see that 
\begin{equation}
\label{eq:connpres5}
\mathscr{A}_{\mathcal{G}}(\eta_o,\mathscr{B}_o) = F_Q(\tilde\beta_o)
\quand
\mathscr{A}_{\mathcal{G}}(\eta_u,\mathscr{B}_u) = F_Q(\tilde\beta_u)
\end{equation}
Formulae \erf{eq:connpres2} -- \erf{eq:connpres5} show \erf{eq:connpres1}.
\end{proof}

\subsubsection{Proof that $\mathcal{A}_{\mathcal{G}}$ is natural in $\mathcal{G}$}

We show that the 1-isomorphisms $\mathcal{A}_{\mathcal{G}}$ are the components of a natural transformation. Let $\mathcal{B}\maps \mathcal{G} \to \mathcal{H}$ be a connection-preserving isomorphism. It suffices to find a 2-isomorphism
\begin{equation*}
\alxydim{@R=1.2cm@C=1.5cm}{\uncon_x(L\mathcal{G}) \ar[r]^-{\uncon_x(L\mathcal{B})} \ar[d]_{\mathcal{A}_{\mathcal{G}}} & \uncon_x(L\mathcal{H}) \ar[d]^{\mathcal{A}_\mathcal{H}} \ar@{=>}[dl] \\ \mathcal{G} \ar[r]_-{\mathcal{B}} & \mathcal{H}}
\end{equation*}
in the 2-category $\diffgrbcon AM$. 
Clockwise, the 1-isomorphism $\mathcal{A}_{\mathcal{H}} \circ \uncon_x(L\mathcal{B})$ consists of the principal $A$-bundle $l^{*}L\mathcal{H} \otimes Q_{\mathcal{H}}$ over $\px Mx^{[2]} \times_M Y_{\mathcal{H}}$. Counter-clockwise, the 1-isomorphism $\mathcal{B} \circ \mathcal{A}_{\mathcal{G}}$ consists of the principal $A$-bundle $Q_{\mathcal{G}} \otimes B$ over $\px Mx \times_M Y_{\mathcal{G}} \times_M Y_{\mathcal{H}}$, where $B$ is the principal $A$-bundle over $Y_{\mathcal{G}} \times_M Y_{\mathcal{H}}$ of the isomorphism $\mathcal{B}$.  To construct the 2-isomorphism, we have the freedom to choose a subduction $\omega\maps W \to \px Mx^{[2]} \times_M Y_{\mathcal{G}} \times_M Y_{\mathcal{H}}$, and instead of choosing the identity, we choose one for which the pullback $\omega^{*}B$ is trivializable. There is even a canonical choice: $W \df  \px Mx^{[2]} \times_M B$, for which we  have a canonical trivialization $g_B\maps \omega^{*}B \to \trivlin_{\eta}$, where $\eta\in\Omega^1_{\mathfrak{a}}(B)$ is the connection on $B$. Now we need an isomorphism  
\begin{equation*}
\phi\maps L\mathcal{H} \otimes Q_{\mathcal{H}} \to Q_{\mathcal{G}} \otimes \trivlin_{\eta}
\end{equation*}
of principal $A$-bundles over $W$. Over a point $w\in W$ with $\omega(w) = (\gamma_1,\gamma_2,b)$ and $b$ projecting to $(g,h) \in Y_{\mathcal{G}} \times_M Y_{\mathcal{H}}$, this is the  map
\begin{equation*}
\phi\maps  L\mathcal{H}|_{l(\gamma_1,\gamma_2)} \otimes Q_{\mathcal{H}}|_{(\gamma_2,h)} \to Q_{\mathcal{G}}|_{(\gamma_1,g)} \maps \mathcal{S} \otimes (\mathcal{S}_2,s_0,s) \mapsto  (\mathcal{S}_1 \circ \gamma_1^{*}\mathcal{B}, b_0 \otimes s_0, b \otimes  s)
\end{equation*}
where $\mathcal{S}_k \df  \iota_k^{*}\mathcal{S}$, and $b_0$ is a fixed choice of an element in $B$ projecting to $(g_0,h_0)$, where these are the choices of lifts of $x$ to $Y_{\mathcal{G}}$ and $Y_{\mathcal{H}}$ we have made to construct $Q_{\mathcal{G}}$ and $Q_{\mathcal{H}}$, respectively.  This is smooth and $A$-equivariant. 

\begin{lemma}
The bundle isomorphism $\phi$ defines a  2-iso\-mor\-phism.
\end{lemma}

\begin{proof}
Let us first show that $\phi$ is compatible with the bundle isomorphisms involved in $\mathcal{B} \circ \mathcal{A}_{\mathcal{G}}$ and $\mathcal{A}_{\mathcal{H}} \circ \uncon_x(L\mathcal{B})$. This is a condition over $W \times_M W$; over a point $(w,w')$ it is the commutativity of the diagram
\begin{equation*}
\alxydim{@C=1.6cm@R=1.2cm}{L\mathcal{G}_{l(\gamma_1,\gamma_1')} \otimes L\mathcal{H}_{l(\gamma_1',\gamma_2')} \otimes Q_{\mathcal{H}}|_{\gamma_2',h'}   \ar[d]_{\psi_{\mathcal{B}} \otimes \id} \ar[r]^-{\id \otimes \phi} & L\mathcal{G}_{l(\gamma_1,\gamma_1')} \otimes Q_{\mathcal{G}}|_{\gamma_1',g'} \ar[d]^{\alpha_{\mathcal{G}}}  \\    L\mathcal{H}_{l(\gamma_1,\gamma_2)} \otimes L\mathcal{H}_{l(\gamma_2,\gamma_2')} \otimes Q_{\mathcal{H}}|_{\gamma_2',h'} \ar[d]_{\id \otimes \alpha_{\mathcal{H}}} & Q_{\mathcal{G}}|_{\gamma_1,g} \otimes P_{\mathcal{G}}|_{g,g'} \ar[d]^{\id \otimes \tilde\beta} \\ L\mathcal{H}_{l(\gamma_1,\gamma_2)} \otimes Q_{\mathcal{H}}|_{\gamma_2,h} \otimes P_{\mathcal{H}}|_{h,h'}\ar[r]_-{\phi \otimes \id} & Q_{\mathcal{G}}|_{\gamma_1,g} \otimes P_{\mathcal{H}}|_{h,h'}\text{.}}
\end{equation*}
Here, $\psi_{\mathcal{B}}$ is the bundle isomorphism of the 1-isomorphism $\uncon_x(L\mathcal{B})$, and $\tilde\beta$ is the pullback of the bundle isomorphism $\beta$ of $\mathcal{B}$ to $B$.
The proof of the commutativity is a tedious but straightforward calculation which we leave out for the sake of brevity.

Now it remains to check that the isomorphism $\phi$ preserves the connections. We consider a path $\Gamma =(\tilde\gamma,\tilde\beta) \in P( L\mathcal{H} \times Q_{\mathcal{H}})$. Its various projections are: $\gamma$ the projection to $PLM$, $\beta$ the projection to $PZ$, $(\beta_1,\beta_2)$ the projection to $P\px Mx^{[2]}$, $\tilde b$ the projection to $PB$, $b$ the projection to $P(Y_{\mathcal{G}} \times_M Y_{\mathcal{H}})$ and $b_{\mathcal{H}},b_{\mathcal{G}}$ the projections to $Y_{\mathcal{G}}$ and $Y_{\mathcal{H}}$, respectively. We note that $\beta = (\beta_2,b_{\mathcal{H}})$. Further, we choose representatives $\mathcal{S}^t$ of $\tilde\gamma(t)$, and $(\mathcal{S}^t_2,s^t_0,s^t)$ of $\tilde\beta(t)$. Then, we have
\begin{equation*}
\tilde\beta'(t) \df  (\phi \circ \Gamma)(t) = (\mathcal{S}_1^t \circ \mathcal{B},b_0 \otimes s_0^t, \tilde b(t) \otimes s^t)\text{.} 
\end{equation*}
We have to show that
\begin{equation}
\label{eq:phiprescon}
F(\tilde\gamma) = F_{Q_{\mathcal{H}}}(\tilde\beta)^{-1} \cdot F_{Q_{\mathcal{G}}}(\tilde\beta') \cdot \exp \left ( \int_{\tilde b} \eta \right )\text{.}
\end{equation}
First of all, we have $F(\tilde\gamma)=\mathscr{A}_{\mathcal{H}}(\gamma,\mathcal{S}^0,\mathcal{S}^1)$. To compute $F_{Q_{\mathcal{H}}}(\tilde\beta)$, we have to choose connection-preserving 2-isomorphisms $\varphi_{0l}\maps \mathcal{S}_2^0 \Rightarrow \mathcal{S}_l$ and $\varphi_{0r}\maps \mathcal{S}_2^0 \Rightarrow \mathcal{S}_r$, where $\mathcal{S}_l$ and $\mathcal{S}_r$ are the canonical trivializations determined by the sections $\sigma_l(s) \df  (s,h_0)$ and $\sigma_r(s) \df  (s,b_{\mathcal{H}}(s))$. As explained in the definition of the connection on $Q_{\mathcal{H}}$, these determine further 2-isomorphisms $\varphi_{1l}$ and $\varphi_{1r}$ forming a boundary record 
\begin{equation*}
\mathscr{B} \df  \left \lbrace \mathcal{S}_2^0,\mathcal{S}_2^1, \mathcal{S}_l,\mathcal{S}_r, \varphi_{0l},\varphi_{0r},\varphi_{1l},\varphi_{1r} \right \rbrace\text{.}
\end{equation*}
Then, $F_{Q_{\mathcal{H}}}(\tilde\beta) = \mathscr{A}_{\mathcal{H}}(\beta_2,\mathscr{B})$. For $F_{Q_{\mathcal{G}}}(\tilde\beta') = \mathscr{A}_{\mathcal{G}}(\beta_1,\mathscr{B}\,')$ we have the boundary record  
\begin{equation}
\label{eq:brman1}
\mathscr{B}\,' = \left \lbrace \mathcal{S}_1^0 \circ \mathcal{B}, \mathcal{S}_1^1 \circ \mathcal{B},\mathcal{T}_l,\mathcal{T}_r,\psi_{0l},\psi_{0r},\psi_{1l},\psi_{1r}  \right \rbrace\text{,}
\end{equation} 
where $\mathcal{T}_l$ and $\mathcal{T}_r$ are the canonical trivializations determined by the sections $\tau_l(s) \df  (s,g_0)$ and $\tau_r(s) \df  (s,b_{\mathcal{G}}(s))$. The 2-isomorphisms $\psi_{0l}$ and $\psi_{0r}$ are chosen in the following way. First we infer that the data of the 1-isomorphism $\mathcal{B}$ determines  connection-preserving 2-isomorphisms
$\beta_l\maps \mathcal{S}_l \circ \mathcal{B} \Rightarrow \mathcal{T}_l$
 and $\beta_r\maps \mathcal{S}_r \circ \mathcal{B} \Rightarrow b^{*}B \otimes \mathcal{T}_r$
over $[0,1]$. Then we define $\psi_{0l} \df  \beta_l \bullet  (\varphi_{0l} \circ \id)$. A computation shows that then $\psi_{1l} = \beta_l \bullet (\varphi_{1l} \circ \id)$. Since $b\maps [0,1] \to Y_{\mathcal{G}} \times_M Y_{\mathcal{H}}$ has a lift $\tilde b\maps [0,1] \to B$, the bundle $b^{*}B$ has a canonical trivialization $t\maps b^{*}B \to \trivlin_0$. We define $\psi_{0r} \df  (t \otimes \id) \bullet \beta_r \bullet (\varphi_{0r} \circ \id)$. Another computation shows that then $\psi_{1r}= (t \otimes \id) \bullet \beta_r \bullet (\varphi_{1r} \circ \id)$. Our first manipulation with the boundary record \erf{eq:brman1} is to replace $\mathcal{T}_r$ by $b^{*}B \otimes \mathcal{T}_r$. The new boundary record is 
 \begin{equation*}
\mathscr{B}\,'' = \left \lbrace \mathcal{S}_1^0 \circ \mathcal{B}, \mathcal{S}_1^1 \circ \mathcal{B},\mathcal{T}_l,b^{*}B \otimes \mathcal{T}_r,\psi_{0l}, \beta_r \bullet (\varphi_{0r} \circ \id),\psi_{1l},\beta_r \bullet (\varphi_{1r} \circ \id)  \right \rbrace\text{,}
\end{equation*} 
and
\begin{equation}
\label{eq:brman2}
\mathscr{A}_{\mathcal{G}}(\beta_1,\mathscr{B}\,'') = \mathscr{A}_{\mathcal{G}}(\beta_1,\mathscr{B}\,') \cdot \exp \left ( \int_{\tilde b } \eta \right ) \text{,}
\end{equation}
since the latter factor is precisely the parallel transport in $b^{*}B$ along $[0,1]$, in the trivialization given by $\tilde b$. Our second manipulation is to replace $\mathcal{T}_l$ and $b^{*}B \otimes \mathcal{T}_r$ by $\mathcal{S}_l \circ \mathcal{B}$ and $\mathcal{S}_r \circ \mathcal{B}$, respectively. The new boundary record is
 \begin{equation*}
\mathscr{B}\,''' = \left \lbrace \mathcal{S}_1^0 \circ \mathcal{B}, \mathcal{S}_1^1 \circ \mathcal{B},\mathcal{S}_l \circ \mathcal{B},\mathcal{S}_r \circ \mathcal{B},\varphi_{0l}\circ \id,\varphi_{0r}\circ \id,\varphi_{1l}\circ \id,\varphi_{1r} \circ \id  \right \rbrace\text{,}
\end{equation*} 
and $\mathscr{A}_{\mathcal{G}}(\beta_1,\mathscr{B}\,'') = \mathscr{A}_{\mathcal{G}}(\beta_1,\mathscr{B}\,''')$. Now we use that the surface holonomy of the bundle gerbe $\mathcal{G}$ is equal to the one of the bundle gerbe $\mathcal{H}$, in virtue of the 1-isomorphism $\mathcal{B}$. That is,
\begin{equation*}
\mathscr{A}_{\mathcal{G}}(\beta_1,\mathscr{B}\,''') = \mathscr{A}_{\mathcal{H}}(\beta_1,\mathscr{B}\,'''')
\quere
\mathscr{B}\,'''' = \left \lbrace \mathcal{S}_1^0, \mathcal{S}_1^1,\mathcal{S}_l,\mathcal{S}_r,\varphi_{0l},\varphi_{0r},\varphi_{1l},\varphi_{1r} \right \rbrace\text{.}
\end{equation*}
Comparing $\mathscr{B}\,''''$ with $\mathscr{B}$, the gluing formula for surface holonomy yields
\begin{equation*}
\mathscr{A}_{\mathcal{H}}(\gamma,\mathcal{S}^0,\mathcal{S}^1) = \mathscr{A}_{\mathcal{H}}(\beta_2,\mathscr{B})^{-1} \cdot \mathscr{A}_{\mathcal{H}}(\beta_1,\mathscr{B}\,'''')  \text{.}
\end{equation*}
Together with \erf{eq:brman2}, this shows \erf{eq:phiprescon}.
\end{proof}

\subsection{Transgression after Regression}

\label{sec:trafterreg}

We have to associate to each fusion bundle $P$ with superficial connection over $L M$ a fusion-preserving isomorphism \begin{equation*}
\varphi_P\maps \trcon(\uncon_x(P)) \to P\text{.}
\end{equation*}
The construction of $\varphi_P$ is the content of the following lemmata, and ends below with Proposition \ref{prop:varphiconpres}. The lemmata are proven under the following slightly weaker assumptions (so that we can use them in the forthcoming Part III \cite{waldorf11}):
\begin{equation}
\label{assumptions}
\begin{minipage}[c]{0.8\textwidth}
Let $(P,\lambda)$ be a fusion bundle with  connection $\omega$, i.e. $\omega$ is compatible and symmetrizing. Instead of requiring that $\omega$ is superficial we only require condition (i) of Definition \ref{def:superficial}. Let $B \in \Omega^2(\px Mx)$ be a curving on the bundle gerbe $\un_x(P)$, which is compatible with the connection $l^{*}\omega$ on the principal $A$-bundle $l^{*}P$ of $\un_x(P)$ in the sense of Definition \ref{def:conn}. The resulting bundle gerbe with connection is denoted $\un_x^B(P)$.
\end{minipage}
\end{equation}
The point is that since $\omega$ may not be superficial,  the canonical curving $B_P$ on the bundle gerbe $\un_x(P)$ constructed in Section \ref{sec:reconcon} is not available. If $\omega$ \emph{is} superficial, we have $\un_x^{B_P}(P) = \uncon_x(P)$.

The following construction relates points in $P$ with  trivializations of $\un^B_x(P)$. We decompose a loop $\beta \in LM$ into several pieces in the following way. First we divide the loop $\beta$ into its two halves $\gamma_1,\gamma_2\in PM$ satisfying $\gamma_1(0)=\gamma_2(0)=\beta(0)$ and $\gamma_1(1)=\gamma_2(1)=\beta(\frac{1}{2})$. For this purpose we choose a smoothing function $\phi$ which adds sitting instants. Then we choose a path  $\gamma \in \px Mx$ with $\gamma(1)=\beta(0)$ and a thin homotopy $h\maps l(\gamma_1 \pcomp \gamma,\gamma_2 \pcomp \gamma) \to \beta$. 
Consider a trivialization $\mathcal{T}\maps \beta^{*}\un^B_x(P) \to \mathcal{I}_0$. This is a principal $A$-bundle $T$ with connection over the diffeological space $Z \df  S^1 \lli{\beta}\times_{\ev_1} \px Mx$, together with a connection-preserving isomorphism
\begin{equation*}
\tau\maps l^{*}P \otimes \zeta_2^{*}T \to \zeta_1^{*}T 
\end{equation*}
over $Z \times_{S^1} Z$. A point in that space is a triple $(t,\eta_1,\eta_2)$, where $t\in S^1$ and $\eta_1,\eta_2\in \px Mx$ such that $\beta(t) = \eta_1(1)=\eta_2(1)$. Over  such a point point, $\tau$ is a map
\begin{equation*}
\tau_{(t,\eta_1,\eta_2)}\maps P_{l(\eta_1,\eta_2)} \otimes T_{t,\eta_2} \to T_{t,\eta_1}\text{.}
\end{equation*}
We construct an element $p \in P_{l(\gamma_1 \pcomp \gamma,\gamma_2 \pcomp \gamma)}$ as follows. The smoothing function $\phi$ defines for $i=1,2$ thin paths  $\alpha_i \in PZ$ which  go from  $(0,\id \pcomp \gamma)$ to $(\frac{1}{2},\gamma_i \pcomp \gamma)$. Let $\alpha \df  \alpha_2\pcomp \prev{\alpha_1}$, and 
choose an element $q\in T_{(\frac{1}{2},\gamma_1 \pcomp \gamma)}$. Then, $p$ is defined by requiring
\begin{equation}
\label{eq:defp}
\tau(p \otimes \ptr\alpha(q))= q\text{,}
\end{equation}
where $\ptr{\alpha}$ is the parallel transport in the bundle $T$. Since $\tau$ and $\ptr{\alpha}$ are $A$-equivariant, $p$ does not depend on the choice of $q$. 
We define
\begin{equation}
\label{eq:defpt}
p_{\mathcal{T}} \df  \ptr{h}(p) \in P_{\beta}\,\text{;}
\end{equation}
this establishes the relation between trivializations of $\un_x^B(P)$ and points in $P$ we were aiming at.

\begin{lemma}
\label{lem:pTindep}
Under the assumptions \erf{assumptions}, the element $p_{\mathcal{T}}$ defined in \erf{eq:defpt} is independent of the choice of the path $\gamma$, the thin homotopy $h$, and the smoothing function $\phi$. Moreover, if $\mathcal{T}$ and $\mathcal{T}'$ are 2-isomorphic trivializations, $p_{\mathcal{T}} = p_{\mathcal{T}'}$.  
\end{lemma}

\begin{proof}
The equality $p_{\mathcal{T}} = p_{\mathcal{T}'}$ is easy to see (for the same choices of $\gamma,h,\phi$).
The independence of the smoothing function follows from the fact that two smoothing functions $\phi,\phi':[0,1]\to [0,1]$ are homotopic
(because $[0,1]^{2}$ is simply connected). Moreover, every homotopy $H$ is automatically thin because $[0,1]$ is one-dimensional. It induces accordant thin homotopies between the various paths  involved in the construction of the elements $p$ (computed with paths $(\gamma_1,\gamma_2,\alpha,h)$ using $\phi$) and $p'$ (computed with paths $(\gamma'_1,\gamma_2',\alpha',h')$ using $\phi'$), in such a way that $p$ and $p'$ correspond to each other under the parallel transport 
\begin{equation*}
\tau_H:P_{l(\gamma_1 \pcomp \gamma,\gamma_2\pcomp \gamma)} \to P_{l(\gamma_1' \pcomp \gamma,\gamma_2'\pcomp \gamma)}
\end{equation*}
along the evident thin path induced by $H$, i.e. $\tau_H(p) = p'$.
Then, Lemma \ref{lem:thinpath} implies that $\tau_h(p) = \tau_{h'}(p')$, i.e. the claimed invariance.

Suppose we have -- for a fixed smoothing function -- two choices $\gamma$, $\gamma'$, and accordant choices of thin homotopies $h$ and $h'$. In the following we denote the two versions of the element $p$ obtained  using $\gamma$ and $\gamma'$, respectively, by $p$ and $p'$. 
Choose any element $\tilde q \in P_{l(\id \pcomp \gamma, \id \pcomp \gamma')}$, and let $\varepsilon_i \in PLM$ be the paths with
\begin{equation*}
\varepsilon_i(0) = l(\id \pcomp \gamma, \id \pcomp \gamma')
\quand
\varepsilon_i(1) = l(\gamma_i \pcomp \gamma, \gamma_i \pcomp \gamma')\text{;}
\end{equation*}
obtained by scaling $\gamma_i$ using the smoothing function.
We claim (1) that
\begin{equation}
\label{eq:claim}
\lambda(p \otimes \ptr{\varepsilon_2}(\tilde q)) = \lambda(\ptr{\varepsilon_1}(\tilde q) \otimes p')\text{,}
\end{equation}
where $\ptr{\varepsilon_i}$ is the parallel transport in $P$. We  claim (2) that the diagram
\begin{equation}
\label{eq:claim2}
\alxydim{@C=0.4cm@R=1.2cm}{P_{l(\gamma_1\pcomp \gamma,\gamma_2\pcomp \gamma)} \otimes P_{l(\gamma_2\pcomp \gamma,\gamma_2\pcomp \gamma')} \ar[r]^-{\lambda} \ar[d]_{\ptr{h} \otimes \ptr{{\varepsilon_2}}^{-1}} & P_{l(\gamma_1\pcomp \gamma,\gamma_2\pcomp\gamma')} \ar[r]^-{\lambda^{-1}} & P_{l(\gamma_1\pcomp \gamma,\gamma_1\pcomp \gamma')} \otimes P_{l(\gamma_1\pcomp \gamma',\gamma_2\pcomp \gamma')} \ar[d]^{\ptr{{\varepsilon_1}}^{-1} \otimes \ptr{h'}} \\ P_{\beta} \otimes P_{l(\id \pcomp \gamma, \id \pcomp \gamma')} \ar[rr]_{\text{flip}} &&  P_{l(\id \pcomp \gamma, \id \pcomp \gamma')} \otimes P_{\beta}}
\end{equation}
is commutative.
Both claims together prove that $\ptr{h}(p) = \ptr{h'}(p')$, which was to show.

In order to prove the first claim, let $q$ and $q'$ be the choices needed to determine $p$ and $p'$, respectively. We can arrange these choices such that
\begin{equation}
\label{eq:choicesofq}
\tau(\tilde q \otimes \ptr{\alpha_1'}^{-1}(q')) = \ptr{\alpha_1}^{-1}(q)\text{.}
\end{equation}
 We compute
\begin{equation*}
\tau(\lambda(p \otimes \ptr{\varepsilon_2}(\tilde q)) \otimes \ptr{\alpha'}(q')) = \tau(p \otimes \tau(\ptr{\varepsilon_2}(\tilde q) \otimes \ptr{{\alpha_2'}}(\ptr{\alpha_1'}^{-1}(q'))) = \tau(p \otimes \ptr{\alpha}(q)) = q\text{,}
\end{equation*}
where the first step uses the compatibility condition between $\tau$ and the fusion product $\lambda$, the second step uses that $\tau$ is a connection-preserving isomorphism and \erf{eq:choicesofq}, and the last step is \erf{eq:defp}. Similarly, we compute
\begin{equation*}
\tau(\lambda(\ptr{\varepsilon_1}(\tilde q) \otimes p') \otimes \ptr{\alpha'}(q')) = q\text{.}
\end{equation*}
This proves the first claim  \erf{eq:claim}.

In order to prove the second claim, we fill  diagram \erf{eq:claim2} by commutative subdiagrams. Let $\delta_1,\delta_2\in P(PM^{[3]})$ be thin paths satisfying:
\begin{eqnarray*}
\delta_1(0) =(\gamma_1 \pcomp \gamma,\gamma_2 \pcomp \gamma,\gamma_2 \pcomp \gamma') &\quad& \delta_1(1) = (\prev{\gamma_2} \pcomp \gamma_1,\id,\gamma'\pcomp \prev{\gamma})
\\
\delta_2(0) = (\gamma_1 \pcomp \gamma,\gamma_1 \pcomp \gamma',\gamma_2 \pcomp \gamma')&& \delta_2(1) = (\gamma \pcomp \prev{\gamma'},\id,\prev{\gamma_1} \pcomp \gamma_2)
\end{eqnarray*}
These paths can be chosen arbitrarily; but one can construct explicit examples.
Now we consider the following diagram of maps between $A$-torsors, in which the unlabelled arrows are parallel transport along obvious reparameterizations the obvious reparameterizations:
\begin{equation*}
\footnotesize
\alxydim{@R=0.8cm@C=-1.6cm}{P_{l(\gamma_1\pcomp \gamma,\gamma_2\pcomp \gamma)} \otimes P_{l(\gamma_2\pcomp \gamma,\gamma_2\pcomp \gamma')} \ar@{}[rrd]|*+[F-:<8pt>]{A} \ar@/_1.5pc/[dddd]_{\ptr{h} \otimes \ptr{{\varepsilon_2}}}="8" \ar@{}[drd];"8"|>>>>>>>*+[F-:<8pt>]{C} \ar[rdd]|{\ptr{e_{12}(\delta_1)} \otimes \ptr{e_{23}(\delta_1)}} \ar[rrr]^-{\lambda} &&& P_{l(\gamma_1\pcomp \gamma,\gamma_2\pcomp\gamma')} \ar[dr]^{\ptr{e_{13}(\delta_2)}} \ar[dl]_{\ptr{e_{13}(\delta_1)}} \ar[rrr]^-{\lambda^{-1}} &&& P_{l(\gamma_1\pcomp \gamma,\gamma_1\pcomp \gamma')} \otimes P_{l(\gamma_1\pcomp \gamma',\gamma_2\pcomp \gamma')} \ar@{}[lld]|*+[F-:<8pt>]{A} \ar[ddl]|{\ptr{e_{12}(\delta_2)} \otimes \ptr{e_{23}(\delta_2)}} \ar@/^1.5pc/[dddd]^{\ptr{{\varepsilon_1}} \otimes \ptr{h'}}="9" \ar@{}[ddl];"9"|>>>>>>>*+[F-:<8pt>]{C} \\  &&  P_{l(\prev{\gamma_2} \pcomp \gamma_1,\gamma'\pcomp \prev{\gamma})} \ar[rr]_{R_{\pi}}="10" \ar@{}[ur];"10"|*+[F-:<8pt>]{D} & \ar@{}[dd]|*+[F-:<8pt>]{B}& P_{l(\gamma \pcomp \prev{\gamma'},\prev{\gamma_1} \pcomp \gamma_2)} \ar[rd]^-{\lambda^{-1}} &  \\ & P_{l(\prev{\gamma_2} \pcomp \gamma_1,\id)} \otimes P_{l(\id,\gamma'\pcomp \prev{\gamma})} \ar[ddl] \ar[rrd]_<<<<<<<{\text{flip}} \ar[ru]^-{\lambda} &&    && P_{l(\gamma \pcomp \prev{\gamma'},\id)} \otimes P_{l(\id,\prev{\gamma_1} \pcomp \gamma_2)} \ar[ddr] \\ &&&P_{l(\id,\gamma'\pcomp \prev{\gamma})} \otimes P_{l(\prev{\gamma_2} \pcomp \gamma_1,\id)}  \ar[rur]_>>>>>>>{R_{\pi} \otimes R_{\pi}}&&&\\ P_{\beta} \otimes P_{l(\id \pcomp \gamma, \id \pcomp \gamma')} \ar[rrrrrr]_-{\text{flip}}="5" \ar@{}[urrr];"5"|*+[F-:<8pt>]{C} &&&&&& P_{l(\id \pcomp \gamma, \id \pcomp \gamma')} \otimes P_{\beta}}
\end{equation*} 
We argue now that all subdiagrams are commutative.
Diagrams $\alxydim{@=0cm}{\ar@{}[r]|*+[F-:<13pt>]{A} &}$ commute since the connection $\omega$ is compatible with $\lambda$.
Diagram $\alxydim{@=0cm}{\ar@{}[r]|*+[F-:<13pt>]{B} &}$ commutes since the connection $\omega$ symmetrizes  $\lambda$.
Diagrams $\alxydim{@=0cm}{\ar@{}[r]|*+[F-:<13pt>]{C} &}$ and $\alxydim{@=0cm}{\ar@{}[r]|*+[F-:<13pt>]{D} &}$ commute because the connection $\omega$ satisfies condition (i) of Definition \ref{def:superficial}. More precisely, diagrams $\alxydim{@=0cm}{\ar@{}[r]|*+[F-:<13pt>]{C} &}$ commute already as diagrams of direct products of $A$-torsors. 
This proves the second claim.
\end{proof}

We recall that isomorphism classes of  trivializations of $\un_x^B(P)$ form the elements of the transgression $\trcon(\un_x^B(P))$. From that point of view, Lemma \ref{lem:pTindep} assures that the assignment $\mathcal{T} \mapsto p_{\mathcal{T}}$ yields a well-defined map
\begin{equation*}
\varphi_{P}: \trcon(\un_x^B(P)) \to P\text{.}
\end{equation*}

\begin{lemma}
\label{lem:varphiPfus}
 $\varphi_P$ is a smooth, fusion-preserving bundle morphism. 
\end{lemma}

\begin{proof}
$\varphi_P$ is by definition fibre-preserving, and its smoothness is straightforward to prove. 
The $A$-equivariance can be seen as follows. Let $a\in A$ and $P_a$ the corresponding principal $A$-bundle over $S^1$ with $\mathrm{Hol}_{P_a}(S^1)=a$. We have to show that 
\begin{equation}
\label{eq:fusisoequiv}
p_{P_a \otimes \mathcal{T}} = p_{\mathcal{T}} \cdot a\text{.}
\end{equation}
To compute $p_{\mathcal{T}}$ may have chosen $q\in T_{(1,\gamma_1 \pcomp \gamma)}$. The new bundle is $T \otimes \mathrm{pr}_1^{*}P_a$, so here we may choose $q \otimes f_a$, for $f_a \in P_a|_1$. Notice that the projection of the path $\alpha \in PZ$ to $S^1$ is the path that runs once around $S^1$. Thus, $\ptr\alpha(q \otimes f_a ) = (\ptr\alpha(q) \otimes f_a) \cdot a$. Then,
\begin{multline*}
\tau(p_{P_a \otimes \mathcal{T}} \otimes q \otimes f_a)= \ptr\alpha(q \otimes f_a) = (\ptr\alpha(q) \otimes f_a) \cdot a = (\tau(p_{\mathcal{T}} \otimes q) \otimes f_a) \cdot a = \tau(p_{\mathcal{T}}\cdot a \otimes  q \otimes f_a)\text{.}
\end{multline*}
This proves \erf{eq:fusisoequiv}.
It remains to show that $\varphi_P$ is fusion-preserving. 
Assume we have a triple $(\gamma_1,\gamma_2,\gamma_3) \nobr\in\nobr PM^{[3]}$ and trivializations $\mathcal{T}_{12}$, $\mathcal{T}_{23}$ and $\mathcal{T}_{13}$ such that $(\mathcal{T}_{12},\mathcal{T}_{23})\sim\mathcal{T}_{13}$, in virtue of 2-isomorphisms
\begin{equation}
\phi_1\maps \iota_1^{*}\mathcal{T}_{12} \Rightarrow \iota_1^{*}\mathcal{T}_{13}
\quomma
\phi_2\maps \iota_2^{*}\mathcal{T}_{12} \Rightarrow \iota_1^{*}\mathcal{T}_{23}
\quand
\phi_3\maps \iota_2^{*}\mathcal{T}_{23} \Rightarrow \iota_2^{*}\mathcal{T}_{13}
\end{equation}
between trivializations of bundle gerbes over the interval $[0,1]$ such that
\begin{equation}
\label{eq:condfus22}
\phi_1|_0 = \phi_{3}|_0 \bullet \phi_{2}|_0
\quand
\phi_1|_1 = \phi_{3}|_1 \bullet \phi_{2}|_1\text{.}
\end{equation}
In order to determine the elements $p_{\mathcal{T}_{ij}}$ we may choose the same path $\gamma$ for all of them. 
We further choose $q_{12} \in T_{12}|_{(\frac{1}{2},\gamma_1 \pcomp \gamma)}$, then we put
\begin{equation*}
q_{13} \df  \phi_1(q_{12}) \in T_{13}|_{(\frac{1}{2},\gamma_1 \pcomp \gamma)}
\quand
q_{23} \df  \phi_2(\ptr{\alpha_{12}}(q_{12})) \in T_{23}|_{(\frac{1}{2},\gamma_2 \pcomp \gamma)}\text{.}
\end{equation*}
A computation shows
\begin{equation}
\label{eq:q13}
\tau_{13}(\lambda(p_{12} \otimes p_{23}) \otimes \ptr{\alpha_{13}}(q_{13})) = q_{13}\text{,}
\end{equation}
Equation \erf{eq:q13} implies
\begin{equation}
\label{eq:prefus}
\lambda(p_{12} \otimes p_{23})=p_{13}\text{.}
\end{equation}
Finally, we may choose a path $H \in P(PM^{[3]})$ with $H(0)= (\gamma_1 \pcomp \gamma, \gamma_2 \pcomp \gamma, \gamma_3 \pcomp \gamma)$ and $H(1) = (\gamma_1,\gamma_2,\gamma_3)$, such that $h_{ij} := e_{ij}(H)$ are thin homotopies. Since the connection $\omega$ is compatible with  $\lambda$,  equation \erf{eq:prefus} implies $\lambda(p_{\mathcal{T}_{12}} \otimes p_{\mathcal{T}_{23}}) = p_{\mathcal{T}_{13}}$. \end{proof}

Next we explore in which sense the bundle isomorphism $\varphi_P$ preserves the connections. Let $\gamma\in PLM$ be a path and let $\exd\gamma\maps C \to M$ be its adjoint map, i.e.  $C \df  [0,1] \times S^1$ is the standard cylinder and $\exd\gamma(t,z) \df  \gamma(t)(z)$.  Let $\mathcal{T}_0$ and $\mathcal{T}_1$ be trivializations of $\un^B_x(P)$ over the loops $\gamma(0)$ and $\gamma(1)$, respectively.
The following technical lemma computes the surface holonomy $\mathscr{A}_{\un_x^B(P)}(\exd\gamma,\mathcal{T}_0,\mathcal{T}_1)$.

We consider the embedding $\iota\maps Q \to C$ defined by $\iota(t,s) \df  (t,s-\frac{1}{2})$, and write $\Phi := \exd\gamma \circ \iota:Q \to M$. With respect to the standard orientations (the counter-clockwise one on $Q$ and the one on $C$ that induces the counter-clockwise one the the boundary loop $\left \lbrace 1 \right \rbrace \times S^1$), the embedding $\iota$ is orientation-reversing. We define a lift $\tilde\Phi\maps Q \to \px Mx$ of $\Phi$ along the end-point evaluation $\ev_1\maps \px Mx \to M$. Let $\psi_t$ and $\xi_{t,s}$ be paths in $Q$ defined by $\psi_t(\tau) \df  (t\tau,\frac{1}{2})$ and $\xi_{t,s}(\tau) \df  (t,(s-\frac{1}{2})\tau + \frac{1}{2})$, with sitting instants produced with some smoothing function, see Figure \ref{fig:square} (a).
\begin{figure}
\hspace{1.3cm}
\begin{tabular}{ccc}
\includegraphics{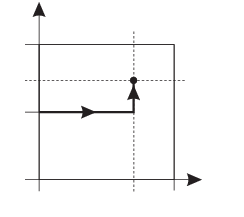}\setlength{\unitlength}{1pt}\begin{picture}(0,0)(203,180)\put(86.60324,259.27427){$1$}\put(86.82066,225.59244){$\textstyle\frac{1}{2}$}\put(87.25550,193.60460){$0$}\put(101.91685,180.78208){$0$}\put(166.95553,180.78208){$1$}\put(148.95553,180.78208){$t$}\put(86.60324,241.27427){$s$}\put(122.95559,214.91830){$\psi_t$}\put(154.48110,232.95487){$\xi_{t,s}$}\end{picture}\hspace{-0.6cm}
& 
\includegraphics{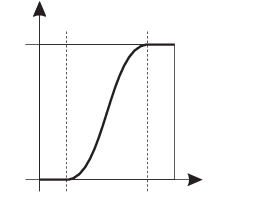}\setlength{\unitlength}{1pt}\begin{picture}(0,0)(345,179)\put(246.65868,179.78230){$\varepsilon$}\put(277.53647,179.56488){$1-\varepsilon$}\put(273.50759,222.28554){$\varphi$}\put(219.80324,259.27427){$1$}\put(220.45550,193.60460){$0$}\end{picture}\hspace{-0.7cm}
& 
\includegraphics{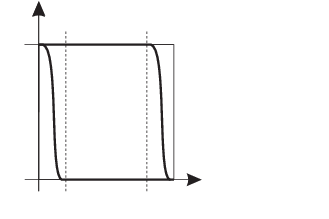}\setlength{\unitlength}{1pt}\begin{picture}(0,0)(543,179)\put(412.25868,179.78230){$\varepsilon$}\put(443.13647,179.56488){$1-\varepsilon$}\put(421.10759,225.88554){$\varepsilon_0$}\put(472.94787,225.88554){$\varepsilon_1$}\put(385.40324,259.27427){$1$}\put(386.05550,193.60460){$0$}\end{picture}\hspace{-1.8cm}
\\
(a) & (b) & (c)
\end{tabular}
\caption{(a) shows a lift of a point in the square to a path in the square. (b) shows a smoothing function which is a diffeomorphism away from its sitting instants. (c) shows two functions that are useful for drawing a bigon into a square.}
\label{fig:square}
\end{figure}
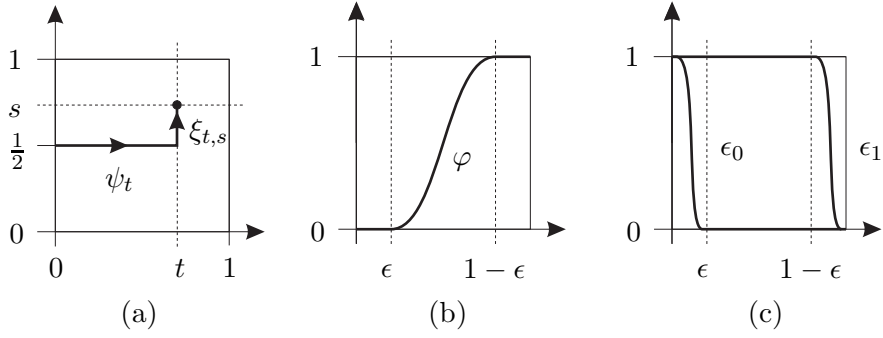
Further, let $\beta$ be a path connecting the base point $x\in M$ with $\Phi(0,\frac{1}{2})$. Then, we put
\begin{equation*}
\tilde\Phi(t,s) \df  P\Phi(\xi_{t,s} \pcomp \psi_t) \pcomp \beta\text{.}
\end{equation*}
Let $\sigma_u\in PLM$ be defined by $\sigma_u(t) \eq l(\tilde\Phi(t,0),\tilde\Phi(t,1))$. Notice that  $\sigma_u(t)$ is thin homotopic to $\gamma(t)$ for every $t\in [0,1]$.

\begin{lemma}
\label{lem:pTconnpres}
Let the structure \erf{assumptions} be given. Suppose further   a path $\gamma \in PLM$ and trivializations $\mathcal{T}_0,\mathcal{T}_1$ of $\un_x^B(P)$ over the end-loops of $\gamma$. Let $\tilde\Phi: Q \to \px Mx$ be the lift constructed above, and let $h_0,h_1 \in PLM$ be thin paths from $h_k(0)=\gamma(k)$ to $h_k(1) = \sigma_u(k)$. Then, we have
\begin{equation*}
\label{eq:holglue}
\mathscr{A}_{\un_x^B(P)}(\exd\gamma,\mathcal{T}_0,\mathcal{T}_1)^{-1} =   \exp \left (  \int_{Q} \tilde\Phi^{*}B \right ) \cdot  \mathrm{PT}(\sigma_u^{*}P,\tau_{h_0}(p_{\mathcal{T}_0}),\tau_{h_1}(p_{\mathcal{T}_1})) \text{,}
\end{equation*}
where $\tau_{h_k}$ denotes the parallel transport in $P$ along $h_k$, and $\mathrm{PT}$ is defined in \erf{eq:defbigpt}.
\end{lemma}

\begin{proof}
In the proof we write $\mathcal{G} \df  \un^B_x(P)$ in order to simplify the notation. We are in the situation of the gluing formula of Lemma \ref{lem:dbraneholprop} (c), and choose a trivialization $\mathcal{T}$ of $\mathcal{G}$ over $[0,1] \times \left \lbrace \frac{1}{2} \right \rbrace \subset C$, together with 2-isomorphisms $\varphi_0\maps \mathcal{T}_0|_{\frac{1}{2}} \Rightarrow \mathcal{T}|_{0,\frac{1}{2}}$ and $\varphi_k\maps \mathcal{T}_1|_{\frac{1}{2}} \Rightarrow \mathcal{T}|_{1,\frac{1}{2}}$. Then, the collection
\begin{equation*}
\mathscr{B} \df  \left \lbrace \iota^{*}\mathcal{T}_0, \iota^{*}\mathcal{T}_1, \iota^{*}\mathcal{T},\iota^{*}\mathcal{T},\varphi_0,\varphi_0,\varphi_1,\varphi_1 \right \rbrace
\end{equation*}
is a boundary record for $\Phi$, and
\begin{equation}
\label{eq:cylcomp1}
\mathscr{A}_{\mathcal{G}}(\exd\gamma, \mathcal{T}_0,\mathcal{T}_1) = \mathscr{A}_{\mathcal{G}}(\Phi, \mathscr{B})^{-1}\text{,}
\end{equation} 
where the sign accounts for the change of the orientation.

Note that $\tilde\Phi$ is a section into the subduction of $\Phi^{*}\mathcal{G}$. 
Let $\mathcal{S}:\Phi^{*}\mathcal{G} \to \mathcal{I}_{\rho}$ be the trivialization defined by this section  according to Lemma \ref{lem:sectriv}, with $\rho = \tilde\Phi^{*}B$. The trivialization $\mathcal{T}$ used above can be taken such that $\iota^{*}\mathcal{T}|_{[0,1] \times \left \lbrace 1 \right \rbrace} = \mathcal{S}|_{[0,1] \times \left \lbrace 1 \right \rbrace}$.
The holonomy $\mathscr{A}_{\mathcal{G}}(\Phi, \mathscr{B})$ we want to compute has the usual two terms \erf{eq:holtriv}: the  integral $I(\tilde\Phi) := \exp ( \int_Q \tilde\Phi^{*}B)$, and a term containing the boundary contributions.

In order to  treat the  boundary term, we first we gather some information about the trivializations over the boundary components of $Q$. 
For $k\eq0,1$, we consider the boundary components $b_k \df  \left \lbrace (k,s) \;|\; s\in[0,1] \right \rbrace$ and  the trivializations $\iota^{*}\mathcal{T}_k$ of $\Phi^{*}\mathcal{G}|_{b_k}$. 
They consist of principal $A$-bundles $T_k$ with connection over  $Z_k \df  [0,1] \times_{M} \px Mx$, and isomorphisms $\tau_k$ over $Z_k^{[2]}$. 
We have the sections $\sigma_k \maps  b_{k} \to Z_k$ defined by $\sigma_k(s) \df (s,\tilde\Phi(k,s))$. 
According to Lemma \ref{lem:sectriv} (b), the isomorphisms $\tau_k$ induce 2-isomorphisms $\tau_k\maps \mathcal{S} \otimes \sigma_k^{*}T_k \Rightarrow \iota^{*}\mathcal{T}_k$. 
Over the boundary component $b_{u}:= \left \lbrace (t,0) \;|\; t \in [0,1] \right \rbrace$ we have the section $\sigma_u$, and the corresponding 2-isomorphism $\lambda\maps \mathcal{S} \otimes \sigma_u^{*}P \to \mathcal{T}$. 
Over the remaining boundary component $b_o \df \left \lbrace (t,1) \;|\; t\in [0,1] \right\rbrace$ we have $\mathcal{S}|_{_{b_o}}=\mathcal{T}|_{b_o}$ by assumption.

Now we  substitute the trivialization $\mathcal{S}$ for each of the trivializations in $\mathscr{B}$ using the formula \erf{eq:trivrep}, producing a new boundary record $\mathscr{B}^{\;\!\prime}$.  Thus, we have to choose trivializations of the \quot{difference bundles} $\sigma_k^{*}T_k$ and $\sigma_u^{*}P$ over the endpoints of their base intervals. Such trivializations are given by selecting points in the fibres: $q_k \in T_k|_{0,\tilde\Phi(k,0)}$, $\tilde q_k \in T_k|_{1,\tilde\Phi(k,1)}$ and $p_k \in P|_{\tilde\Phi(k,0),\tilde\Phi(k,1)}$. Using the paths $\alpha^k\in PZ_k$ from  the definitions of $p_{\mathcal{T}_k}$ (see \erf{eq:defp}) we can  choose $\tilde q_k \eq \ptr{\alpha^k}(q_k)$.
According to \erf{eq:trivrep}, the new boundary record $\mathscr{B}^{\;\!\prime}$ has over  the four corners of $Q$  the 2-isomorphisms
\begin{eqnarray}
\label{eq:id1}
&\alxydim{}{\mathcal{S} \ar@{=>}[r]^-{\varphi_0^{-1}} & \iota^{*}\mathcal{T}_0 \ar@{=>}[r]^-{\tau_0^{-1}} & \mathcal{S} \otimes \sigma_0^{*}T_0  \ar@{=>}[r]^-{\tilde q_0} & \mathcal{S}}&\text{ over }(0,1)\quad\quad\quad
\\
\label{eq:id2}
&\alxydim{}{\mathcal{S}  \ar@{=>}[r]^-{\tilde q_1} &  \mathcal{S} \otimes \sigma_1^{*}T_1 \ar@{=>}[r]^-{\tau_1} & \iota^{*}\mathcal{T}_1 \ar@{=>}[r]^-{\varphi_1} & \mathcal{S}}&\text{ over }(1,1)
\\
\label{eq:id3}
&\alxydim{}{\mathcal{S}  \ar@{=>}[r]^-{q_0} & \mathcal{S} \otimes \sigma_0^{*}T_0  \ar@{=>}[r]^-{\tau_0} & \iota^{*}\mathcal{T}_0 \ar@{=>}[r]^-{\varphi_0} & \mathcal{S} \ar@{=>}[r]^-{\lambda^{-1}} & \mathcal{S} \otimes \sigma_u^{*}P  \ar@{=>}[r]^-{p_0} & \mathcal{S}}&\text{ over }(0,0)
\\
\label{eq:id4}
&\alxydim{}{\mathcal{S}  \ar@{=>}[r]^-{p_1} & \mathcal{S} \otimes \sigma_u^{*}P \ar@{=>}[r]^-{\lambda} & \mathcal{S} \ar@{=>}[r]^-{\varphi_1} & \iota^{*}\mathcal{T}_1 \ar@{=>}[r]^-{\tau_1^{-1}} & \mathcal{S} \otimes \sigma_1^{*}T_1 \ar@{=>}[r]^-{ q_1} & \mathcal{S}}&\text{ over }(1,0)
\end{eqnarray}
Obviously, it is  possible to make the choices  $\varphi_0$, $\varphi_1$, $p_0$, $p_1$ in such a way such that all four 2-isomorphisms are \emph{identities}. Accordingly, all boundary contributions vanish, and $\mathscr{A}_{\mathcal{G}}(\Phi,\mathscr{B}^{\;\!\prime})= I(\tilde\Phi)$. However, by going from $\mathscr{B}$ to $\mathscr{B}^{\;\!\prime}$ we must  compensate the changes according to \erf{eq:trivrep} by the following three terms: $\mathrm{PT}(\sigma_k^{*}T_k,q_k,\tilde q_k)$ for $k=0,1$, and $\mathrm{PT}(\sigma_u^{*}P,p_0,p_1)$. The first vanishes for $k=0,1$ due to the definition of $\tilde q_k$. Thus,
\begin{equation}
\label{eq:cylcomp2}
\mathscr{A}_{\mathcal{G}}(\Phi, \mathscr{B}) = \mathscr{A}_{\mathcal{G}}(\Phi,\mathscr{B}^{\;\!\prime})  \cdot \mathrm{PT}(\sigma_u^{*}P,p_0,p_1) = I(\tilde\Phi) \cdot \mathrm{PT}(\sigma_u^{*}P,p_0,p_1)\text{.}
\end{equation}
We look at equations \erf{eq:id1} to \erf{eq:id4}. Combining all relations we have collected, \erf{eq:id1} and \erf{eq:id3} imply
$\tau_0(p_0 \otimes \tilde q_0)=q_0$, and \erf{eq:id2} and \erf{eq:id4} imply
$\tau_1(p_1 \otimes \tilde q_1) = q_1$.
Thus, the elements $p_0$ and $p_1$ are exactly those used in the definition of the elements $p_{\mathcal{T}_0}$ and $p_{\mathcal{T}_1}$. Due to the independence proved in Lemma \ref{lem:pTindep}, we may choose the thin paths $\prev{h_0}$ and $\prev{h_1}$, so that $p_{\mathcal{T}_0} = \tau_{\prev{h_0}}(p_0)$ and $p_{\mathcal{T}_1} = \tau_{\prev{h_1}}(p_1)$. Together with \erf{eq:cylcomp1} and \erf{eq:cylcomp2}, this shows the claimed formula.
\end{proof}

In the following we specialize the discussion  to the case that the connection $\omega$ on $P$ is superficial, and that the curving $B$ on $\un_x(P)$ is the canonical curving $B_P$ of Section \ref{sec:reconcon}.
We compute the two terms in the formula of Lemma \ref{lem:pTconnpres}:
\begin{lemma}
\label{lem:calcterms}
In the situation of Lemma \ref{lem:pTconnpres}, assume that $\omega$ is superficial and that $B = B_P$. Then, we have:
\begin{enumerate}[(i)]
\item 
\label{eq:bulkint}
$\displaystyle \exp \left (  \int_{Q} \tilde\Phi^{*}B \right )=1$.

\item
$\ptr{\gamma}(p_{\mathcal{T}_0})\cdot \mathrm{PT}(\sigma_u^{*}P,\tau_{h_0}(p_{\mathcal{T}_0}),\tau_{h_1}(p_{\mathcal{T}_1}))  = p_{\mathcal{T}_1}$.
\end{enumerate}
\end{lemma}

\begin{proof}
(i) 
We recall that the 2-form $B_P$ has been defined as corresponding to a smooth map $G_P\maps B\px Mx \to A$ under the relation $G_P(\Sigma) = I(\Sigma)$ for any bigon $\Sigma\in B\px Mx$, where $I(\Sigma) := \exp \int_{Q} (\Sigma^{*}B)$. Unfortunately, the map $\tilde\Phi\maps Q \to \px Mx$ is not directly a bigon. To remedy this problem we define a bigon $\Sigma_Q \in BQ$  such that  $\Sigma \df  B\tilde\Phi (\Sigma_Q) \in B\px Mx$ satisfies $I(\Sigma) = I(\tilde\Phi)^{-1}$. Then we prove that $G_P(\Sigma)=1$, which implies \erf{eq:bulkint}.

We give an explicit definition of $\Sigma_Q$. We use the choice of a particular smoothing function $\varphi:[0,1] \to [0,1]$: we require a sitting instant $0 < \varepsilon  < \frac{1}{2}$ such that $\varphi(t)=0$ for all $t < \varepsilon$ and $\varphi(t)=1$ for all $t>1-\varepsilon$, and we require that $\varphi$ restricted to $(\varepsilon,1-\varepsilon)$ is a diffeomorphism onto $(0,1)$, see Figure \ref{fig:square} (b). Let $\varepsilon_0,\varepsilon_1\maps [0,1] \to [0,1]$ be smooth maps with sitting instants such that $\varepsilon_0(0)=1$ and $\varepsilon_0(t)=0$ for all $t \geq \varepsilon$, and similar $\varepsilon_1(t)= 1$ for $t\leq 1-\varepsilon$ and $\varepsilon_1(1)=0$, see Figure \ref{fig:square} (c). Then, we define
\begin{equation*}
\Sigma_Q\maps [0,1]^2 \to Q\maps (t,s) \mapsto (\varphi(t), \varepsilon_1(t) - \varphi(s)(\varepsilon_1(t) - \varepsilon_0(t)) )\text{.}
\end{equation*}
Consider the small square $Q' \df  (\varepsilon,1-\varepsilon)^2$. The restriction of $\Sigma_Q$ to $Q'$ is an orientation-reversing diffeomorphism to $(0,1)^2$. The restriction of $\Sigma_Q$ to  $Q\setminus Q'$ has rank one, in particular, the pullback of any 2-form vanishes in that region. Thus, we have
\begin{equation*}
I(\tilde\Phi) = 
\exp \left ( \int_{\Sigma_Q(Q')} \tilde\Phi^{*}B_P \right )= \exp \left ( \int_{Q'} -\Sigma_Q^{*}\tilde\Phi^{*}B_P \right )
= I(B\tilde\Phi(\Sigma_Q))^{-1}=I(\Sigma)^{-1}\text{.}
\end{equation*}

Next we prove that $G_P(\Sigma)=1$. We recall that $G_P(\Sigma)$ was defined in Section \ref{sec:reconcon} via parallel transport  along a path $\gamma_{\Sigma}$ in $LM$. In the notation used there, we find:
\begin{multline*}
\Sigma^o(t) = P\Phi(\xi_{\varphi(t),\varepsilon_1(t)} \pcomp \psi_{\varphi(t)}) \pcomp \beta
\quomma
\Sigma^m(t)(s) =   \Phi(\Sigma_Q(t,s))
\\\quand
\Sigma^u(t) = P\Phi(\xi_{\varphi(t),\varepsilon_0(t)} \pcomp \psi_{\varphi(t)}) \pcomp \beta\text{.}
\end{multline*}
We observe that the loop $\gamma_{\Sigma}(t)$ is thin homotopy equivalent to the loop $l(\beta_t,\beta_t)$ with $\beta_t \df  P\Phi(\psi_{\varphi(t)}) \pcomp \beta$. Moreover, these thin homotopies can be chosen smoothly depending on $t$, and are identities for $t=0,1$. This shows that the path $\gamma_{\Sigma}$ is rank-two-homotopic to the path $t \mapsto l(\beta_t,\beta_t)$ that factors through the support of the flat section $\can$. Together with Lemma \ref{lem:pointwisethin}, this shows that $G_P(\Sigma)=1$.

(ii) We choose a rank-two-homotopy $h$ between the original path $\gamma\in PLM$ and the path $\sigma_u\in PLM$, which can be obtain by reparameterizations and by retracting the path $\gamma$. Let $h_k$ be the restrictions of $h$ to the end-loops. Since $\omega$ is superficial, Lemma \ref{lem:pointwisethin} implies that the diagram
\begin{equation*}
\alxydim{@=1.2cm}{P_{\gamma(0)} \ar[d]_{\ptr{h_0}} \ar[r]^{\ptr{\gamma}} & P_{\gamma(1)} \ar[d]^{\ptr{h_1}} \\ P_{\tilde\Phi(0,0),\tilde\Phi(0,1)} \ar[r]_{\tau_{\sigma_u}}  & P_{\tilde\Phi(1,0),\tilde\Phi(1,1)} }
\end{equation*}
is commutative. The commutativity implies
\begin{multline*}
\tau_{h_1}(\tau_{\gamma}(p_{\mathcal{T}_0})) \cdot \mathrm{PT}(\sigma_u^{*}P,\tau_{h_0}(p_{\mathcal{T}_0}),\tau_{h_1}(p_{\mathcal{T}_1}))\\=\tau_{\sigma_u}(\tau_{h_0}(p_{\mathcal{T}_0})) \cdot \mathrm{PT}(\sigma_u^{*}P,\tau_{h_0}(p_{\mathcal{T}_0}),\tau_{h_1}(p_{\mathcal{T}_1})) = \tau_{h_1}(p_{\mathcal{T}_1})\text{,} 
\end{multline*}
where the last equality is the definition of $\mathrm{PT}$. This shows (ii).
\end{proof}

The following proposition summarizes the results of all lemmata above in the case of a superficial connection. 

\begin{proposition}
\label{prop:varphiconpres}
Suppose $(P,\lambda)$ is a fusion bundle with superficial connection. Then, the assignment $\mathcal{T} \mapsto p_{\mathcal{T}}$ defines a  fusion-preserving, connection-preserving bundle isomorphism
\begin{equation*}
\varphi_P: \trcon(\uncon_x(P)) \to P\text{.}
\end{equation*}
\end{proposition}

\begin{proof}
First we note that the assumptions \erf{assumptions} are satisfied, so that $\varphi_P$ is a well-defined, smooth, fusion-preserving bundle isomorphism according to Lemmata \ref{lem:pTindep} and \ref{lem:varphiPfus}.
By definition of the connection on $\trcon(\uncon_x(P))$, 
\begin{equation*}
\tau_{\gamma}(\mathcal{T}_0) = \mathcal{T}_1 \cdot \mathscr{A}(\exd\gamma,\mathcal{T}_0,\mathcal{T}_1)\text{,}
\end{equation*}
where $\mathscr{A}(\exd\gamma, \mathcal{T}_0,\mathcal{T}_1)$ denotes the surface holonomy of $\uncon_x(P)$, $\gamma\in PLM$ is a path and $\exd\gamma$ is its adjoint map. According to Lemmata \ref{lem:pTconnpres} and \ref{lem:calcterms} we have  
\begin{multline*}
\ptr{\gamma}(\varphi_P(\mathcal{T}_0)) = \ptr{\gamma}(p_{\mathcal{T}_0})\cdot \mathrm{PT}(\sigma_u^{*}P,\tau_{h_0}(p_{\mathcal{T}_0}),\tau_{h_1}(p_{\mathcal{T}_1})) \cdot \mathscr{A}(\exd\gamma,\mathcal{T}_0,\mathcal{T}_1) \\=  \varphi_P(\mathcal{T}_1) \cdot \mathscr{A}(\exd\gamma,\mathcal{T}_0,\mathcal{T}_1) = \varphi_P(\ptr{\gamma}(\mathcal{T}_0))\text{.}
\end{multline*}
This shows that $\varphi_P$ exchanges the parallel transport of the connections $\omega$ and $\omega_{\mathcal{G}}$. 
\end{proof}

Finally, we prove that the isomorphism $\varphi_P$ is natural in $P$. Let $\psi\maps P_1 \to P_2$ be an isomorphism in $\fusbunconsf A{LM}$. We write $\mathcal{G}_1$ and $\mathcal{G}_2$ for the regressed bundle gerbes over $M$, and $\mathcal{A}_{\psi}\maps \mathcal{G}_1 \to \mathcal{G}_2$ for the regressed isomorphism, i.e.
\begin{equation*}
L\mathcal{A}_{\psi} \maps L\mathcal{G}_1 \to L\mathcal{G}_2 \maps \mathcal{T} \mapsto \mathcal{T} \circ \mathcal{A}_{\psi}^{-1}\text{.}
\end{equation*}
Let us first compute the trivialization $\mathcal{T}'\df \mathcal{T} \circ \mathcal{A}_{\psi}^{-1}$ of $\mathcal{G}_2$ over a loop $\beta \in LM$. Recall that $\mathcal{T}$ consists of a principal $A$-bundle $T$ over $Z$, and of an isomorphism $\tau$ over $Z \times_{S^1} Z$. From the rules of inverting and composing 1-isomorphisms between bundle gerbes \cite{waldorf1} and the definition of regression we see that $\mathcal{T}'$ has the subduction $Z'\df  S^1 \lli{\beta} \times_{\ev_1} \px Mx^{[2]}$ with the projection $\zeta\maps Z' \to Z\maps (t,\eta_1,\eta_2)\mapsto (t,\eta_1)$ to the subduction of $\mathcal{G}_2$. Over $Z'$, it has the principal $A$-bundle 
\begin{equation*}
T' \df  \mathrm{pr}^{*}_{12} s^{*}P_2^{\vee} \otimes \mathrm{pr}^{*}_{13}T\text{,}
\end{equation*}
where $s\maps \px Mx^{[2]} \to \px Mx^{[2]}$ flips the two factors, and $P_2^{\vee}$ is the dual bundle (it has the same total space but with $A$ acting through inverses). Further, the trivialization $\mathcal{T}'$ has an isomorphism 
\begin{equation*}
\tau'\maps P_2 \otimes \zeta_2^{*}T' \to \zeta_1^{*}T'
\end{equation*}
over $Z' \times_{S^1} Z'$, which is a combination of $\lambda_2$ and $\psi$. 

We claim that $\mathcal{T}'$ is 2-isomorphic to the trivialization $\mathcal{S}$ consisting of the principal $A$-bundle $S \df  T$ over $Z$ and the 2-isomorphism $\sigma\maps P_2 \otimes \zeta_2^{*}T \to \zeta_1^{*}T$, which is defined over a point $(t,\eta,\tilde\eta) \in Z \times_{S^1} Z$ by
\begin{equation*}
\alxydim{@C=1.3cm}{P_2|_{\eta,\tilde\eta} \otimes T_{t,\tilde\eta} \ar[r]^{\psi^{-1} \otimes \id} & P_1|_{\eta,\tilde\eta} \otimes T_{t,\tilde\eta} \ar[r]^-{\tau} & T_{t,\eta}\text{.}}
\end{equation*}
Indeed, a 2-isomorphism $\phi\maps \mathcal{T}' \Rightarrow \mathcal{S}$ is given by
\begin{equation*}
\alxydim{@C=1.3cm}{P_{2}^{\vee}|_{\eta_2,\eta_1} \otimes T_{t,\eta_2} \ar[r]^-{t_2 \otimes \id} & P_2|_{\eta_1,\eta_2} \otimes T_{t,\eta_2} \ar[r]^-{\psi^{-1} \otimes \id} &  P_1|_{\eta_1,\eta_2} \otimes T_{t,\eta_2} \ar[r]^-{\tau} & T_{t,\eta_1}\text{.}}
\end{equation*}
It is straightforward to show that it satisfies the necessary condition $\zeta_2^{*}\phi \circ \tau' = \sigma \circ (\id \otimes \zeta_1^{*}\phi)$, using that $\psi$ is fusion-preserving.

The result of the above computation is that $L\mathcal{A}_{\psi}(\mathcal{T}) = \mathcal{S}$. Now suppose we have determined $\varphi_{P_1}(\mathcal{T})$ from the equation $\tau(p_{\mathcal{T}} \otimes \ptr\alpha(q))= q$. Then, the calculation
\begin{equation*}
\sigma(\psi^{-1}(p_{\mathcal{T}}) \otimes \ptr\alpha(q)) = \tau(p_{\mathcal{T}} \otimes \tau_{\alpha}(q)) = q
\end{equation*}
shows that $p_{\mathcal{S}} = \psi^{-1}(p_{\mathcal{T}})$. All together, we obtain a commutative diagram
\begin{equation*}
\alxydim{@=1.2cm}{L\mathcal{G}_1 \ar[d]_{\varphi_{P_1}} \ar[r]^{L\mathcal{A}_{\psi}} & L\mathcal{G}_2 \ar[d]^{\varphi_{P_2}} \\ P_1 \ar[r]_{\psi} & P_2}
\end{equation*}
showing that the isomorphisms $\varphi_P$ are the components of a natural equivalence.

\setsecnumdepth{2}

\begin{appendix}

\setsecnumdepth{1}

\section{Smooth 2-Functors and 2-Forms}

\label{app:functorsandforms}

Bigons in a smooth manifold are discussed in detail in \cite{schreiber5}. Here we recall some aspects.
A \emph{bigon} in a diffeological space $X$ is a path $\Sigma \in PPX$ with $P\ev_0(\Sigma) = \id_x$ and $P\ev_1(\Sigma)=\id_y$, for some $x,y\in X$. The diffeological space of bigons in $X$ is denoted  by $BX$. It is functorial: for $f:X \to Y$ a smooth map, there is an induced smooth map $Bf\maps BX \to BY$ with $B\id = \id_{BX}$ and $B(g \circ f)= Bg \circ Bf$. Bigons can be composed in two ways. \emph{Vertical composition} is simply the composition of paths in $PX$, namely $\Sigma' \pcomp \Sigma$, defined whenever $\Sigma(1) = \Sigma'(0)$. \emph{Horizontal composition} is defined whenever $P\ev_1(\Sigma) = P\ev_0(\Sigma')$, namely by
$\Sigma' \circledast \Sigma \df  (P\pcomp)(\Sigma',\Sigma)$,
where $P\pcomp\maps P(PX \times_X PX) \to PPX$ is induced from path composition, applied to the path $(\Sigma',\Sigma)$ in $PX \times_X PX$.

Two bigons $\Sigma$ and $\Sigma'$ are called \emph{thin homotopy equivalent}, if there exists a path $h \in PBX$ in the space of bigons with $\ev(h)=(\Sigma,\Sigma')$,
such that $P\ev_0(h) \in PPX$ and $P\ev_1(h)\in PPX$ are thin homotopies between $\Sigma(0)$ and $\Sigma'(0)$, and $\Sigma(1)$ and $\Sigma'(1)$, respectively, and such that the adjoint map $\exd h\maps [0,1]^3 \to X$ has rank two. 
The diffeological space of thin homotopy classes of bigons in $X$ is denoted $\mathcal{B}X$, and we have a  smooth projection $\mathrm{pr}\maps BX \to \mathcal{B}X$. 

On a smooth manifold $M$ there is a  relation between differential forms $\Omega^k_{\mathfrak{a}}(M)$ with values in the Lie algebra $\mathfrak{a}$ of an abelian Lie group $A$, and certain smooth and functorial maps
$F\maps P^kM \to A$, where $P^k$ is an iterated path space.
Heuristically, the relation is the one between connections on higher, trivial principal bundles and their higher parallel transport maps. 

For $k=1$  a precise formulation is developed in \cite{schreiber3} (including the case of a non-abelian Lie group). Let $\fun M A$ be the groupoid whose objects are smooth maps $F\maps PM \to A$ which are constant on thin homotopy classes of paths and satisfy
\begin{equation*}
F(\gamma_2 \pcomp \gamma_1) = F(\gamma_2) \cdot F(\gamma_1)\text{,}
\end{equation*}
and whose morphisms from $F_1$ to $F_2$ are smooth maps $g:M \to A$ satisfying
\begin{equation}
\label{eq:trans}
g(\gamma(1)) \cdot F_1(\gamma) = F_2(\gamma) \cdot g(\gamma(0))
\end{equation}
for all $\gamma \in PM$. Composition is multiplication. On the other side, consider the groupoid $\mathcal{Z}_M^1(A)$ whose objects are 1-forms $\omega \in \Omega^1_{\mathfrak{a}}(M)$ and whose morphisms between $\omega_{1}$ and $\omega_2$ are smooth maps $g:M \to A$ such that $\omega_2 = \omega_1 - g^{*}\bar\theta$, where $\bar\theta$ is the right-invariant Maurer-Cartan form on $A$. We have shown \cite[Proposition 4.7]{schreiber3} that
the functor
\begin{equation*}
\mathfrak{P}_1^{\infty} \maps \mathcal{Z}_M^1(A) \to \fun MA 
\end{equation*}
that sends a 1-form $\omega \in \Omega^1_{\mathfrak{a}}(M)$ to the map
\begin{equation*}
F_{\omega}\maps PM \to A \maps \gamma \mapsto  \exp \left ( - \int_{\gamma} \omega \right ) \text{,}
\end{equation*}
and that is the identity on morphisms, 
is an \emph{isomorphism} between categories.
\cite[Theorem B.2]{waldorf9} assures that this results  remains true upon replacing the smooth manifold $M$ by a general diffeological space $X$. We have used this theorem in our construction of the connection on the bundle $L\mathcal{G}$ in Section \ref{sec:constrconn}.

In Section \ref{sec:reconcon} we also need the version for $k=2$.  For smooth manifolds and a  (possibly non-abelian) Lie 2-group it has been formulated and proved in \cite{schreiber5}. Here we restrict ourselves to an abelian Lie group $A$ (making it \emph{much} simpler) and generalize \emph{this} version to diffeological spaces. 

We consider the 2-groupoid $\tfun MA$  \cite[Section 2.1]{schreiber5} whose  objects  are smooth maps $G\maps \mathcal{B}M \to A$ that exchange both the vertical and horizontal composition of bigons with the multiplication in $A$. The 1-morphisms between such maps $G_1$ and $G_2$ are objects $F\maps \pt M \to A$ in $\fun MA$ satisfying 
\begin{equation}
\label{eq:pseudotrans}
F(\gamma_2) \cdot G_2(\Sigma) = G_1(\Sigma) \cdot F(\gamma_1)
\end{equation}
for all bigons $\Sigma\in \mathcal{B}X$ with $\ev(\Sigma) =: (\gamma_1,\gamma_2) \in \pt M^{[2]}$. Equation \erf{eq:pseudotrans} can be seen as the analog of  \erf{eq:trans}.   
The 2-morphisms between $F_1$ and $F_2$ are simply the morphisms in $\fun MA$.
On the other hand, we consider a 2-groupoid $\mathcal{Z}^2_M(A)$ \cite[Definition 2.12]{schreiber5} whose objects are 2-forms $\psi \in \Omega^2_{\mathfrak{a}}(M)$. A 1-morphism from $\psi_1$ to $\psi_2$ is a 1-form $\omega\in \Omega^1_{\mathfrak{a}}(M)$ such that
\begin{equation}
\label{eq:morfun}
\psi_2  = \psi_1 - \mathrm{d}\omega\text{,}
\end{equation}
and a 2-morphism between $\omega_1$ and $\omega_2$ is simply a morphism in $\mathcal{Z}_M^1(A)$. By
\cite[Theorem 2.21]{schreiber5} and \cite[Lemma 4.1.2]{schreiber2}, the 2-functor
\begin{equation*}
\mathfrak{P}_2^{\infty}\maps \mathcal{Z}^2_M(A) \to \tfun MA
\end{equation*}
that sends 2-form $\psi\in\Omega^2_{\mathfrak{a}}(M)$ to the map
\begin{equation}
\label{eq:defG}
G_{\psi}\maps \mathcal{B}M \to A\maps \Sigma \mapsto \exp \left ( -\int_{\Sigma} \psi \right )\text{,}
\end{equation}
and that is on 1-morphisms and 2-morphisms given by the functor $\mathfrak{P}_1^{\infty}$, is an \emph{isomorphism} between 2-groupoids.

All this is functorial in the manifold $M$: for a smooth map $f:M \to N$ one has on  one  side the obvious pullback 2-functor $f^{*}\maps \mathcal{Z}^2_N(A) \to \mathcal{Z}^2_M(A)$, and
on the other side the 2-functor $f^{*}\maps \tfun NA \to \tfun MA$
 defined as the pre-composition with the induced map $\mathcal{B}f\maps \mathcal{B}M \to \mathcal{B}N$. The isomorphism $\mathfrak{P}_2^{\infty}$ exchanges these two 2-functors with each other; in other words, $\mathfrak{P}_2^{\infty}$ is an isomorphism between presheaves of 2-groupoids over the category of smooth manifolds. 

The generalized version we need in Section \ref{sec:reconcon} is the following. All definitions are literally valid for  an arbitrary diffeological space $X$: the 2-groupoids $\mathcal{Z}^2_X(A)$ and $\tfun XA$, of the pullback 2-functors induced by smooth maps $f:X \to Y$, and  the 2-functor $\mathfrak{P}_2^{\infty}$ (we denote the generalized version by $\mathfrak{P}_2$). Namely, on objects $\mathfrak{P}_2$ is defined by an  integral of the ordinary  2-form $\Sigma^{*}\psi$ over $[0,1]^2$, while on 1-morphisms and 2-morphisms, the functor $\mathfrak{P}_1^{\infty}$ has already been generalized to a functor $\mathfrak{P}_1$ in the diffeological setup \cite[Theorem B.2]{waldorf9}. We claim:

\begin{theorem}
\label{th:functorsvsforms}
For any diffeological space $X$, the 2-functor
\begin{equation*}
\mathfrak{P}_2\maps \mathcal{Z}^2_X(A) \to \tfun XA
\end{equation*}
is an isomorphism of 2-groupoids. 
\end{theorem} 

\def\ep#1{\widetilde{#1}{}}

\begin{proof}
On 1-morphisms and 2-morphisms, it is an isomorphism because the functor $\mathfrak{P}_1$ is an isomorphism between groupoids \cite[Theorem B.2]{waldorf9}. On objects, an inverse is  easy to define. If an  object $G$ in $\tfun XA$ is given, one has for each plot $c:U \to X$ a 2-form $\psi_c \in \Omega^2_{\mathfrak{a}}(U)$, obtained by applying the inverse of $\mathfrak{P}_2^{\infty}$ to $c^{*}G$.  These 2-forms  define an object $\psi =\left \lbrace \psi_c\right \rbrace$ in $\mathcal{Z}^2_X(A)$. It is straightforward to check that $G_{\psi} = G$. 
\end{proof}

\end{appendix}


\kobib{../../bibliothek/tex}


\begin{thebibliography}{GSW11}
\addcontentsline{toc}{section}{\refname}

\bibitem[Alv85]{alvarez1}
O.~Alvarez, \quot{Topological quantization and cohomology}.
\newblock {\em Commun. Math. Phys.}, 100:279--309, 1985.
\bibitem[Ati85]{atiyah2}
M.~F. Atiyah, \quot{Circular symmetry and stationary phase approximation}.
\newblock {\em Ast\'erisque}, 131:43--59, 1985.
\bibitem[Bar91]{barret1}
J.~W. Barrett, \quot{Holonomy and path structures in general relativity and
  {Y}ang-{M}ills theory}.
\newblock {\em Int. J. Theor. Phys.}, 30(9):1171--1215, 1991.
\bibitem[BH11]{baez6}
J.~C. Baez and A.~E. Hoffnung, \quot{Convenient categories of smooth spaces}.
\newblock {\em Trans. Amer. Math. Soc.}, 363(11):5789--5825, 2011.
\newblock \kobiburl{http://arxiv.org/abs/0807.1704}
\bibitem[BM94]{brylinski4}
J.-L. Brylinski and D.~A. McLaughlin, \quot{The geometry of degree four
  characteristic classes and of line bundles on loop spaces {I}}.
\newblock {\em Duke Math. J.}, 75(3):603--638, 1994.
\bibitem[Bry93]{brylinski1}
J.-L. Brylinski, {\em Loop spaces, characteristic classes and geometric
  quantization}, volume 107 of {\em Progr. Math.}
\newblock Birkh\"auser, 1993.
\bibitem[Bun02]{Bunke2002}
U.~Bunke, \quot{Transgression of the index gerbe}.
\newblock {\em Manuscripta Math.}, 109(3):263--287, 2002.
\newblock \kobiburl{http://arxiv.org/abs/math/0109052}
\bibitem[CJM02]{carey2}
A.~L. Carey, S.~Johnson, and M.~K. Murray, \quot{Holonomy on {D}-branes}.
\newblock {\em J. Geom. Phys.}, 52(2):186--216, 2002.
\newblock \kobiburl{http://arxiv.org/abs/hep-th/0204199}
\bibitem[CP98]{Coquereaux1998}
R.~Coquereaux and K.~Pilch, \quot{String structures on loop bundles}.
\newblock {\em Commun. Math. Phys.}, 120:353--378, 1998.
\bibitem[Gaj96]{gajer1}
P.~Gajer, \quot{The geometry of deligne cohomology}.
\newblock {\em Invent. Math.}, 127(1):155--207, 1996.
\newblock \kobiburl{http://arxiv.org/abs/alg-geom/9601025}
\bibitem[Gaw88]{gawedzki3}
K.~Gaw\c{e}dzki, \quot{Topological actions in two-dimensional quantum field
  theories}.
\newblock In G.~{'}t Hooft, A.~Jaffe, G.~Mack, K.~Mitter, and R.~Stora,
  editors, {\em Non-perturbative quantum field theory}, pages 101--142. Plenum
  Press, 1988.
\bibitem[Gaw05]{gawedzki4}
K.~Gaw\c{e}dzki, \quot{Abelian and non-abelian branes in {WZW} models and
  gerbes}.
\newblock {\em Commun. Math. Phys.}, 258:23--73, 2005.
\newblock \kobiburl{http://arxiv.org/abs/hep-th/0406072}
\bibitem[Gom03]{gomi3}
K.~Gomi, \quot{Connections and curvings on lifting bundle gerbes}.
\newblock {\em J. Lond. Math. Soc.}, 67(2):510--526, 2003.
\newblock \kobiburl{http://arxiv.org/abs/math/0107175}
\bibitem[GR02]{gawedzki1}
K.~Gaw\c{e}dzki and N.~Reis, \quot{{WZW} branes and gerbes}.
\newblock {\em Rev. Math. Phys.}, 14(12):1281--1334, 2002.
\newblock \kobiburl{http://arxiv.org/abs/hep-th/0205233}
\bibitem[GR03]{gawedzki2}
K.~Gaw\c{e}dzki and N.~Reis, \quot{Basic gerbe over non simply connected
  compact groups}.
\newblock {\em J. Geom. Phys.}, 50(1--4):28--55, 2003.
\newblock \kobiburl{http://arxiv.org/abs/math.dg/0307010}
\bibitem[Gro72]{grothendieck1}
A.~Grothendieck, \quot{Compl\'ements sur les biextensions: Propri\'et\'es
  g\'en\'erales des biextensions des sch\'emas en groupes}.
\newblock In {\em Groupe de Monodromie en G\'eom\'etrie Alg\'ebrique}, volume
  288 of {\em Lecture Notes in Math.} Springer, 1972.
\bibitem[GSW11]{gawedzki8}
K.~Gaw\c{e}dzki, R.~R. Suszek, and K.~Waldorf, \quot{Bundle gerbes for
  orientifold sigma models}.
\newblock {\em Adv. Theor. Math. Phys.}, 15(3):621--688, 2011.
\newblock \kobiburl{http://arxiv.org/abs/0809.5125}
\bibitem[GW09]{gawedzki9}
K.~Gaw\c{e}dzki and K.~Waldorf, \quot{{P}olyakov-{W}iegmann formula and
  multiplicative gerbes}.
\newblock {\em J. High Energy Phys.}, 09(073), 2009.
\newblock \kobiburl{http://arxiv.org/abs/0908.1130}
\bibitem[HW41]{hurewicz1}
W.~Hurewicz and H.~Wallman, {\em Dimension theory}.
\newblock Bull. Amer. Math. Soc., 1941.
\bibitem[IZ]{iglesias1}
P.~Iglesias-Zemmour, \quot{Diffeology}.
\newblock Preprint.
\newblock \kobiburl{http://math.huji.ac.il/\~{}piz/documents/Diffeology.pdf}
\bibitem[KM97]{kriegl1}
A.~Kriegl and P.~W. Michor, {\em The convenient setting of global analysis}.
\newblock AMS, 1997.
\bibitem[Lot]{Lott}
J.~Lott, \quot{The index gerbe}.
\newblock Preprint.
\newblock \kobiburl{http://arxiv.org/abs/math.DG/0106177}
\bibitem[McL92]{mclaughlin1}
D.~A. McLaughlin, \quot{Orientation and string structures on loop space}.
\newblock {\em Pacific J. Math.}, 155(1):143--156, 1992.
\bibitem[Mei02]{meinrenken1}
E.~Meinrenken, \quot{The basic gerbe over a compact simple {L}ie group}.
\newblock {\em Enseign. Math., II. Sér.}, 49(3--4):307--333, 2002.
\newblock \kobiburl{http://arxiv.org/abs/math/0209194}
\bibitem[Mur96]{murray}
M.~K. Murray, \quot{Bundle gerbes}.
\newblock {\em J. Lond. Math. Soc.}, 54:403--416, 1996.
\newblock \kobiburl{http://arxiv.org/abs/dg-ga/9407015}
\bibitem[Seg01]{Segal2001}
G.~Segal, \quot{Topological structures in string theory}.
\newblock {\em Phil. Trans. R. Soc. Lond. A}, 359:1389--1398, 2001.
\bibitem[ST]{stolz3}
S.~Stolz and P.~Teichner, \quot{The spinor bundle on loop spaces}.
\newblock Preprint.
\newblock
  \kobiburl{http://people.mpim-bonn.mpg.de/teichner/Math/Surveys_files/MPI.pdf}
\bibitem[SW]{schreiber2}
U.~Schreiber and K.~Waldorf, \quot{Connections on non-abelian gerbes and their
  holonomy}.
\newblock Preprint.
\newblock \kobiburl{http://arxiv.org/abs/0808.1923}
\bibitem[SW09]{schreiber3}
U.~Schreiber and K.~Waldorf, \quot{Parallel transport and functors}.
\newblock {\em J. Homotopy Relat. Struct.}, 4:187--244, 2009.
\newblock \kobiburl{http://arxiv.org/abs/0705.0452v2}
\bibitem[SW11]{schreiber5}
U.~Schreiber and K.~Waldorf, \quot{Smooth functors vs. differential forms}.
\newblock {\em Homology, Homotopy Appl.}, 13(1):143--203, 2011.
\newblock \kobiburl{http://arxiv.org/abs/0802.0663}
\bibitem[Wal]{waldorf9}
K.~Waldorf, \quot{Transgression to loop spaces and its inverse, {I}:
  Diffeological bundles and fusion maps}.
\newblock {\em Cah. Topol. G\'eom. Diff\'er. Cat\'eg.}, to appear.
\newblock \kobiburl{http://arxiv.org/abs/0911.3212}
\bibitem[Wal07]{waldorf1}
K.~Waldorf, \quot{More morphisms between bundle gerbes}.
\newblock {\em Theory Appl. Categ.}, 18(9):240--273, 2007.
\newblock \kobiburl{http://arxiv.org/abs/math.CT/0702652}
\bibitem[Wal10]{waldorf5}
K.~Waldorf, \quot{Multiplicative bundle gerbes with connection}.
\newblock {\em Differential Geom. Appl.}, 28(3):313--340, 2010.
\newblock \kobiburl{http://arxiv.org/abs/0804.4835v4}
\bibitem[Wal11]{waldorf13}
K.~Waldorf, \quot{A loop space formulation for geometric lifting problems}.
\newblock {\em J. Aust. Math. Soc.}, 90:129--144, 2011.
\newblock \kobiburl{http://arxiv.org/abs/1007.5373}
\bibitem[Wal12]{waldorf11}
K.~Waldorf, \quot{Transgression to loop spaces and its inverse, {III}: Gerbes
  and thin fusion bundles}.
\newblock {\em Adv. Math.}, 231:3445--3472, 2012.
\newblock \kobiburl{http://arxiv.org/abs/1109.0480}
\end{thebibliography}
\end{document}